\documentclass[a4paper,12pt]{amsart}

\usepackage[colorlinks,citecolor = red, linkcolor=blue,hyperindex]{hyperref}
\usepackage{euscript,eufrak,verbatim}
\usepackage[psamsfonts]{amssymb}
\usepackage[usenames]{color}
\usepackage[curve]{xypic}
\usepackage{graphicx}
\usepackage{mathrsfs}
\usepackage{calligra,amsmath,amssymb}

\usepackage{float}
\usepackage{bbm}
 \usepackage{color}
\newtheorem{thm}{Theorem}[section]
\newtheorem{lem}[thm]{Lemma}

\newtheorem{prop}[thm]{Proposition}
\newtheorem{cor}[thm]{Corollary}
\newtheorem{defn}[thm]{Definition}
\newtheorem{rem}[thm]{Remark}

\newtheorem*{prob*}{Problem}

\def\R{\mathbb{R}}
\def\N{\mathbb{N}}
\def\Z{\mathbb{Z}}
\def\Q{\mathbb{Q}}
\def\E{\mathbb{E}}

\def\FF{F}
\def\O{\mathcal{O}}

\def\F{\mathbb{F}}

\def\X{\mathcal{X}}

\def\C{\mathbb{C}}
\def\K{\mathcal{K}}

\def\PP{\mathbb{P}}

\def\1{\mathbbm{1}}
\def\an {\text{\, and \,}}
\def\tr{\mathrm{tr}}

\def\rank{\mathrm{rank}}

\def\vol{\mathrm{vol}}
\def\GL{\mathrm{GL}}
\def\Mat{\mathrm{Mat}}
\def\Cut{\mathrm{Cut}}
\def\diag{\mathrm{diag}}

\def\supp{\mathrm{supp}}

\def\modulo{\mathrm{mod\,}}
\def\ORB{\mathscr{ORB}}

\def\Sym{\mathrm{Sym}}

\def\Xint#1{\mathchoice
{\XXint\displaystyle\textstyle{#1}}%
{\XXint\textstyle\scriptstyle{#1}}%
{\XXint\scriptstyle\scriptscriptstyle{#1}}%
{\XXint\scriptscriptstyle\scriptscriptstyle{#1}}%
\!\int}
\def\XXint#1#2#3{{\setbox0=\hbox{$#1{#2#3}{\int}$ }
\vcenter{\hbox{$#2#3$ }}\kern-.6\wd0}}

\def\dashint{\Xint-}

\begin{document}

\title[Ergodic measures on infinite  matrices]{Ergodic measures on  spaces of infinite matrices over non-Archimedean locally compact fields}

%\date{\today}

\author%[authorlabel1]
{Alexander I. Bufetov}
\address%[authorlabel1]
{Alexander I. Bufetov: Aix-Marseille Universit{\'e}, Centrale Marseille, CNRS, I2M, UMR7373,  39 Rue F. Juliot Curie 13453, Marseille; Steklov Institute of Mathematics, Moscow; Institute for Information Transmission Problems, Moscow; National Research University Higher School of Economics, Moscow}

\email{bufetov@mi.ras.ru}
\author%[authorlabel1]
{Yanqi Qiu}
\address%[authorlabel1]
{Yanqi QIU: CNRS, Institut de Math{\'e}matiques de Toulouse, Universit{\'e} Paul Sabatier, 118 Route de Narbonne, F-31062 Toulouse Cedex 9, France}
\email{yqi.qiu@gmail.com}

%\thanks{}

\begin{abstract}
Let $F$ be a non-discrete non-Archimedean locally compact  field and  $\mathcal{O}_F$  the ring of integers in $F$. The main results of this paper are Theorem \ref{CT-ns} that  classifies ergodic probability measures on the space $\mathrm{Mat}(\mathbb{N}, F)$ of infinite matrices with enties in $F$ with respect to the natural action of the group $\mathrm{GL}(\infty,\mathcal{O}_F) \times \mathrm{GL}(\infty,\mathcal{O}_F)$ and Theorem \ref{CT-sym} that,  for non-dyadic $F$,   classifies ergodic probability measures on the space $\mathrm{Sym}(\mathbb{N}, F)$ of infinite symmetric matrices with respect to the natural action of the group $\mathrm{GL}(\infty,\mathcal{O}_F)$.
\end{abstract}


\maketitle

\tableofcontents

\setcounter{equation}{0}

\section{Introduction}

Given a  group action on a topological space,  it is natural to try to describe the corresponding space of  ergodic invariant probability measures.  For some very classical actions, such as, for example, that 
of the shift on the space of infinite binary sequences, the space of ergodic measures is huge and does not seem to admit a reasonable description. On the other hand, for a number of  natural actions of infinite-dimensional groups, a complete classification is possible. For example,   De Finetti's  Theorem  (1937) claims that for the action of the infinite symmetric group on the space of infinite binary sequences, 
all ergodic probability measures are Bernoulli, and Schoenberg's Theorem (1951) claims that for the action of the infinite orthogonal group on the space of infinite ${\mathbb R}$-valued sequences, all ergodic probability measures are Gaussian. In both these examples, the space of ergodic probability measures is one-dimensional.
 Pickrell \cite{Pickrell3, Pickrell-jfa90, Pickrell1} and, by a different method, Olshanski and Vershik \cite{OV-ams96}, classified all ergodic unitarily invariant measures on the space of infinite Hermitian matrices. In this case, an ergodic measure is determined by infinitely many parameters.

In this paper, we study classification of ergodic measures for actions related to the following  inductive limit group 
\begin{align}\label{def-gl-infty}
\GL(\infty, \O_F):= \lim_{\longrightarrow} \GL(n, \O_F),
\end{align}
 where $\O_F$ is the ring of integers in a  non-discrete  locally compact  non-Archimedean  field $\FF$ and $\GL(n, \O_F)$ is the compact group of invertible $n\times n$ matrices over $\O_F$.  Denote by $\Mat(\N, \FF)$ (resp. $\Sym(\N, \FF)$) the space  of infinite matrices (resp. infinite symmetric matrices) over $\FF$. Our main results are: 

(i)  the classification of the ergodic probability measures for the group action of $\GL(\infty, \O_F) \times \GL(\infty, \O_F)$ on $\Mat(\N, \FF)$ defined by
\begin{align*}
((g_1, g_2), M) \mapsto g_1 M g_2^{-1}, \quad g_1, g_2 \in \GL(\infty, \O_F), M \in \Mat(\N, \FF); 
\end{align*}

(ii) the classification of the ergodic probability measures for the group action of $\GL(\infty, \O_F)$ on $\Sym(\N, \FF)$ defined by 
\begin{align*}
(g, M) \mapsto g M g^{t}, \quad g \in \GL(\infty, \O_F), S \in \Sym(\N, \FF),
\end{align*}
where  and $g^t$ is the transposition of $g$.

We proceed to the precise formulation.  Let $\FF$ be a non-discrete locally compact  non-Archimedean field (for example,   the field of $p$-adic numbers). Let $| \cdot |$ be the absolute value on $\FF$. The  ring $\O_F$ of integers in $F$ is   given by $\{x\in \FF: |x| \le 1\}$. The unique maximal and principal ideal of  $\O_F$ is given by  $\{x\in \FF: | x| < 1\}$.  Throughout the paper, we fix any generator $\varpi$ of $\{x\in \FF: | x| < 1\}$, that is, $\{x\in \FF: | x| < 1\}= \varpi \O_F$.  The quotient $\O_F/\varpi \O_F$ is a finite field with $q$ elements.

Define the inductively compact group $\GL(\infty, \O_F)$ by \eqref{def-gl-infty}. Set 
 \begin{align*}
 \Mat(\N, \FF):  = \{ X= (X_{ij})_{i, j \in \N}| X_{ij} \in \FF\}. 
 \end{align*}
 Let $\Mat(\infty, \FF)$ denote the subspace of $\Mat(\N,\FF)$ consisting of matrices whose all but a finite number of  coefficients are zero. 
Define also 
 \begin{align*}
 \Sym(\N, \FF):  = \{ X\in \Mat(\N, \FF) | X_{ij}  = X_{ji},  \forall  i, j \in \N\},
 \end{align*}
and let  $\Sym(\infty, \FF ) : = \Sym(\N, \FF) \cap \Mat(\infty, \FF)$.

\subsection{Classification of ergodic measures on $\Mat(\N, F)$}

Let $\Delta$ be the set of non-increasing sequences in $\Z\cup\{-\infty\}$, that is, 
\begin{align}\label{def-Delta}
\Delta: =\left\{ \mathbbm{k}= (k_j)_{j = 1}^\infty\Big|  k_j \in \Z \cup\{-\infty\}; k_1 \ge k_2 \ge \cdots \right\}. 
\end{align}
By the inclusion $\Delta\subset (\Z\cup\{-\infty\})^\N$, we equip $\Delta$ with the induced topology of the Tychonoff's product topology on  $(\Z\cup\{-\infty\})^\N$.

To each sequence $\mathbbm{k} \in \Delta$, we now assign an ergodic  $\GL(\infty, \O_F) \times \GL(\infty, \O_F)$-invariant probability measure on $\Mat(\N, F)$.  Let 
 \begin{align*}
 X_i^{(n)},  \quad Y_i^{(n)},  \quad Z_{ij}, \quad i,j, n = 1, 2, \cdots
 \end{align*}
 be independent random variables, each sampled with respect to the normalized Haar measure on the compact additive group $\O_F$.  In what follows, we use the convention $\varpi^\infty=0$.

\begin{defn}\label{defn-rm}
Given an element $\mathbbm{k} \in\Delta$, let 
\begin{align*}
\mu_{\mathbbm{k}} : =  \mathcal{L} (M_{\mathbbm{k}})
\end{align*}
be the probability distribution of the  infinite random matrix $M_{\mathbbm{k}}$ defined as follows. Denote $k: = \lim\limits_{n\to\infty}k_n \in \Z\cup\{-\infty\}$ and set 
\begin{align*}
M_{\mathbbm{k}} : =  \Big [ \sum\limits_{n: \, k_n > k} \varpi^{-k_n}X_i^{(n)} Y_j^{(n)} + \varpi^{-k}Z_{ij}\Big]_{i, j\in\N}.
\end{align*}

\end{defn}

Let $\mathcal{P}_{\mathrm{erg}}(\Mat(\N, F))$ be the space of ergodic $\GL(\infty, \O_F) \times \GL(\infty, \O_F)$-invariant probability measures on $\Mat(\N, F)$, endowed with the induced weak topology.  The classification of  $\mathcal{P}_{\mathrm{erg}}(\Mat(\N, F))$ is given by the following

\begin{thm}\label{CT-ns}
The  map $\mathbbm{k} \mapsto \mu_{\mathbbm{k}}$ is a homeomorphism between  $\Delta$ and $\mathcal{P}_{\mathrm{erg}}(\Mat(\N, F))$. 
\end{thm}

\begin{rem}
By Theorem \ref{CT-ns},  the space $\mathcal{P}_{\mathrm{erg}}(\Mat(\N, F))$  is weakly closed in the space of all Borel measures  on $\Mat(\N, F)$ and  is $\sigma$-compact;  moreover, any measure  $\mu_{\mathbbm{k}}\in \mathcal{P}_{\mathrm{erg}}(\Mat(\N, F))$ is compactly supported.
\end{rem}

Let us explain Theorem \ref{CT-ns} in more detail.   We have the following   elementary ergodic measures:
\begin{itemize}
\item (Haar type measures) For any $k\in \Z$, the normalized Haar measure on $\Mat(\N, \varpi^{-k} \O_F)$  is $\GL(\infty, \O_F) \times \GL(\infty, \O_F)$-ergodic.
\item (Non-symmetric Wishart type measures) Let $X_1, Y_1, X_2, Y_2, \cdots$ be independent and uniformly distributed on $\O_F$.  For any $k\in \Z$, the  probability distribution of  the random matrix: 
\begin{align*}
 [\varpi^{-k} X_i Y_j ]_{i, j \in\N}
\end{align*}
is  $\GL(\infty, \O_F) \times \GL(\infty, \O_F)$-ergodic.
\end{itemize}
Theorem \ref{CT-ns}  implies that any  ergodic $\GL(\infty, \O_F) \times \GL(\infty, \O_F)$-invariant probability measure on $\Mat(\N, F)$  can be obtained as a possibly infinite convolution of the above two types of elementary ones.

\subsection{Classification of ergodic measures on $\Sym(\N, F)$}
In what follows, when dealing with the ergodic measures on $\Sym(\N, F)$, we always assume that the field $F$ is non-dyadic, that is, the cardinality of the field of residue class $\O_F/\varpi\O_F$ is not a power of $2$.

The  group of units of $\O_F$ is given by  $\O_F^\times :  = \{x\in \FF: |x| = 1\}.$ 
Denote by $(\O_F^{\times})^2$ the subgroup of  $\O_F^{\times}$ defined by
\begin{align*}
(\O_F^{\times})^2: = \{ x \in \O_F^\times: \text{there exists $a\in F$ such that $x = a^2$} \}.
\end{align*}
If  $F$ is non-dyadic,  then the quotient $\O_F^{\times} / (\O_F^{\times})^2$ has two elements. Throughout the paper,  we fix a non-square unit $\varepsilon \in \O_F^{\times} \setminus (\O_F^{\times})^2$.   

We now explicitly  describe the parametrization of ergodic $\GL(\infty, \O_F)$-invariant probability measures on $\Sym(\N, F)$. 

Recall the definition \eqref{def-Delta} of the set $\Delta$ of non-increasing sequences in $\Z\cup\{-\infty\}$. A sequence $(k_j)_{j\in\N} \in \Delta$ is called  finite if $k_j=-\infty$ for all sufficiently large $j$.  In this case, either $k_1 = -\infty$, then we identity the sequence with an empty sequence, or  $j_0: = \max\{j| k_j \in \Z\}\in\N$, then we identify $(k_j)_{j\in\N}$ with $(k_j)_{j=1}^{j_0}$ and $j_0$ is called the {\it length} of the sequence.  Conversely, for any finite non-increasing sequence $(k_j)_{j=1}^n$ in $\Z$, we identify it with the element in $\Delta$ by adding infinitely many $-\infty$ at the end of $(k_j)_{j=1}^n$. 
 
For any $k\in\Z$, let $\Z_{>k}$ denote the set of integers strictly larger than $k$. We introduce the following  four subsets  of $\Delta$:    
\begin{itemize}
\item $\Delta[k]$,  the set of {\it non-increasing} sequences  of {\it finite length} in $\Z_{>k}$; 
\item $\Delta^{\sharp}[k]$, the set of {\it strictly decreasing} sequences  in $\Z_{>k}$ (which are automatically of finite length). 
\item $\Delta[-\infty]$,  the set of {\it  non-increasing} sequences in $\Z$  {\it of finite length} or {\it of infinite length tending to $-\infty$}; 
\item $\Delta^{\sharp}[-\infty]$, the set of {\it strictly decreasing} sequences of {\it finite or infinite length} in $\Z$. 
\end{itemize}
Note that  for any $k\in\Z\cup\{-\infty\}$, the following relations hold: 
\begin{align*}
\Delta^{\sharp}[k] \subset \Delta[k], \quad     \Delta^{\sharp}[k] \subset \Delta^{\sharp}[-\infty] \an  \Delta^{\sharp}[k] \subset \Delta[-\infty].  
\end{align*}

We introduce the parameter space 
\begin{align}\label{def-Omega}
\Omega: = \bigsqcup_{k\in\Z\cup\{-\infty\}} \{k\} \times  \Delta[k] \times \Delta^{\sharp}[k],
\end{align}
where $\{k\}$ is the singleton with a single element $k$. The space $\Omega$ is equipped with the topology induced by the inclusion: 
\begin{align*}
\Omega \subset (\Z\cup\{-\infty\}) \times (\Z\cup\{-\infty\})^{\N}\times (\Z\cup\{-\infty\})^{\N}.
\end{align*}

To each element $\mathbbm{h}\in\Omega$, we assign an ergodic $\GL(\infty, \O_F)$-invariant probability measure on $\Sym(\N, F)$ as follows.  Let 
 \begin{align*}
 X_i^{(n)}, \quad Y_i^{(n)},  \quad H_{ij}, \quad i \le j, \, n = 1, 2, \cdots
 \end{align*}
 be   independent random variables uniformly distributed on $\O_F$. In particular, 
 define 
 \begin{align}\label{def-H}
H = [H_{ij}]_{i, j \in \N}
\end{align} 
by setting $H_{ij} = H_{ji}$ if $i > j$, then $H$ is an infinite symmetric random matrix sampled uniformly from $\Sym(\N, \O_F)$.

\begin{defn}\label{defn-sym-1}
For any  $\mathbbm{h} \in  \Omega$  given by 
 \begin{align*}
 \mathbbm{h} =(k; \mathbbm{k},  \mathbbm{k}'), \text{ with }  k\in\Z\cup \{-\infty\}, \, \mathbbm{k} \in \Delta[k], \, \mathbbm{k}' \in \Delta^{\sharp}[k],
 \end{align*} 
we define 
\begin{align*}
\nu_{\mathbbm{h}} : =  \mathcal{L} (S_{\mathbbm{h}}),
\end{align*} 
as the probability distribution of the  infinite symmetric random matrix $S_{\mathbbm{h}}$ defined as follows. First let
\begin{align*}
W_{\mathbbm{k}}: =  \Big [ \sum\limits_{n=1}^\infty \varpi^{- k_n}X_i^{(n)} X_j^{(n)}\Big]_{i, j\in\N},\quad W_{\mathbbm{k}'}: =  \Big [ \sum\limits_{n=1}^\infty \varpi^{- k_n'}Y_i^{(n)} Y_j^{(n)}\Big]_{i, j\in\N},
\end{align*}
we set
\begin{align*}
S_{\mathbbm{h}} : =   W_{\mathbbm{k}} +   \varepsilon W_{\mathbbm{k}'} + \varpi^{-k} H.
\end{align*}
\end{defn}

\begin{rem}\label{rem-id-2}
The strictly decreasing assumption on the sequences $\mathbbm{k}' \in \Delta^\sharp[k]$ is imposed for the uniqueness of parametrization. The reason is the following: 
\begin{align}\label{id-2}
\mathcal{L} \Big(\Big [ \sum\limits_{n=1}^2 X_i^{(n)} X_j^{(n)}\Big]_{i, j\in\N}\Big)  = \mathcal{L} \Big(  \varepsilon \Big [ \sum\limits_{n=1}^2 X_i^{(n)} X_j^{(n)}\Big]_{i, j\in\N}\Big). 
\end{align}
For the detail, see Remark \ref{rem-id-2-pf}  and  the proof of Proposition \ref{prop-unique-sym} below. 
\end{rem}

Let $\mathcal{P}_{\mathrm{erg}}(\Sym(\N, F))$ be the space of ergodic $\GL(\infty, \O_F)$-invariant probability measures on $\Sym(\N, F)$, endowed with the induced weak topology.  The classification of  $\mathcal{P}_{\mathrm{erg}}(\Sym(\N, F))$ is given by the following

\begin{thm}\label{CT-sym}
Assume that $F$ is non-dyadic. Then the  map $\mathbbm{h} \mapsto \nu_{\mathbbm{h}}$ is a homeomorphism between  $\Omega$ and $\mathcal{P}_{\mathrm{erg}}(\Sym(\N, F))$. 
\end{thm}

\begin{rem}
By Theorem \ref{CT-sym},  the space $\mathcal{P}_{\mathrm{erg}}(\Sym(\N, F))$  is weakly closed in the space of all Borel measures  on $\Sym(\N, F)$ and  is $\sigma$-compact. Moreover,  any measure  $\nu_{\mathbbm{h}}\in \mathcal{P}_{\mathrm{erg}}(\Sym(\N, F))$ is compactly supported.
\end{rem}

Theorem \ref{CT-sym} can be explained  as follows.   We have the following elementary ergodic measures. 
\begin{itemize}
\item (Haar type measures)  For any $k\in \Z$, the normalized Haar measure on $\Sym(\N,    \varpi^{-k} \O_F)$ is  $\GL(\infty, \O_F)$-ergodic.
\item (Symmetric Wishart type measures) Let $X_1, X_2, \cdots$ be independent copies of $\FF$-valued random variables,  all of which are uniformly distributed on $\O_F$.  For any $k\in \Z$, the   distributions of the infinite rank one  random matrices: 
$$
  [ \varpi^{-k} X_i X_j ]_{i, j \in\N} \an [ \varepsilon \varpi^{-k} X_i X_j ]_{i, j \in\N}
$$ 
are $\GL(\infty, \O_F)$-ergodic.
\end{itemize}
Theorem \ref{CT-sym}  implies that any  ergodic $\GL(\infty, \O_F)$-invariant probability measure on $\Sym(\N, F)$  can be obtained as a possibly infinite convolution of the above two types of elementary ones.

\subsection{Characteristic functions  of ergodic measures.}

Let $\chi\in\widehat{\FF}$ be a {\it fixed character} of $\FF$ such that
\begin{align}\label{ass-chi-intro}  
\text{  $\chi|_{\O_F} \equiv 1 $ and $\chi$ is  not constant on  $\varpi^{-1} \O_F$.}
\end{align}
Given a Borel probability measure $\mu$ on $\Mat(\N, \FF)$, its characteristic function, or Fourier transform,   $\widehat{\mu}$ is defined on $\Mat(\infty, \FF)$  by the formula 
\begin{align*}
\widehat{\mu} (A): =  \int_{\Mat(\N,\FF)} \chi(\tr(AM)) \mu(dM), \quad A \in\Mat(\infty, F).
\end{align*}
Similarly, given a Borel probability measure $\nu$ on $\Sym(\N, \FF)$, its characteristic function  $\widehat{\nu}$ is defined on $\Sym(\infty, \FF)$  by the formula 
\begin{align*}
\widehat{\nu} (A): =  \int_{\Sym(\N,\FF)} \chi(\tr(AS)) \nu(dS), \quad  A \in\Sym(\infty,F).
\end{align*}

Let $\Mat(n, F)$ be the space of   $n\times n$ matrices with entries in $F$. Every $A\in \Mat(n, F)$ can be written (see Lemma \ref{lem-s-num} below) in the form
\begin{align*}
A = a \cdot \diag(\varpi^{-k_1}, \cdots, \varpi^{-k_n}) \cdot b, \quad a, b \in \GL(n, \O_F), k_i \in \Z\cup\{-\infty\}.
\end{align*}
 These numbers $k_1, k_2, \cdots, k_n $, taken with  multiplicities, are uniquely determined by  $A$  and are called the {\it singular numbers} of $A$. The collection of the singular numbers of the matrix $A$ is   denoted $\mathrm{Sing}(A)$.

After the computation of characteristic functions for the probability  measures in the list of measures defined in  Definition \ref{defn-rm}  (see Proposition \ref{prop-explicit} below),   Theorem \ref{CT-ns} can be reformulated in the following form.
\begin{thm}\label{thm-equiv}
The characteristic functions of ergodic $\GL(\infty, \O_F) \times \GL(\infty, \O_F)$-invariant probability measures on $\Mat(\N, \FF)$ are exactly of the form 
\begin{align*}
\varphi(A) = \prod_{ \ell \in  \mathrm{Sing}(A)}     \exp\Big(- \log q  \cdot \sum\limits_{j=1}^\infty (k_j+\ell)\mathbbm{1}_{\{k_j+\ell\ge 1\}}\Big), \quad A \in  \Mat(\infty, F), 
\end{align*}
where  $\mathbbm {k} = (k_j)_{j\in\N}\in\Delta$ is the parameter sequence.  
\end{thm}

For formulating a similar statement in symmetric case, we need to introduce a function $\theta: F \rightarrow \C$ by 
\begin{align}\label{defn-theta}
\theta (x) : = \int_{\O_F} \chi(z^2 \cdot x ) dz.
\end{align}
Properties of the function $\theta$ are summarized in Proposition \ref{prop-theta-detail} below. 

Let $\Sym(n, F)$ be the space of $n\times n$ symmetric matrices with entries in $F$.  Note that if the field $F$ is non-dyadic, then any $A\in \Sym(n, F)$ can be written  (see Lemma \ref{lem-q-diag} below) in the form 
\begin{align*}
A = g \cdot \diag (x_1, \cdots, x_n) \cdot g^t, \quad g \in \GL(n, \O_F). 
\end{align*}
After the computation of characteristic functions for the probability  measures in the list of measures defined in Definition \ref{defn-sym-1} (see Proposition \ref{prop-ch-sym}),  Theorem \ref{CT-sym} can be reformulated in the following form.

\begin{thm}\label{thm-equiv-sym}
Assume that $F$ is non-dyadic.  Then the characteristic functions of the ergodic $\GL(\infty, \O_F)$-invariant probability measures on $\Sym(\N, F)$ are exactly given by
\begin{align*}
 \Phi(\diag(x_1, \cdots, x_m, 0, \cdots))  = \prod_{i=1}^m \Big[  \1_{\O_F} (\varpi^{-k} x_i)    \prod_{j=1}^\infty  \theta(  \varpi^{- k_j} x_i) \prod_{j=1}^\infty  \theta(  \varepsilon \varpi^{- k_j'} x_i) \Big],
\end{align*}
where $\mathbbm{h} = (k; (k_j)_{j\in\N},  (k_j')_{j\in\N})$ is a parameter in $\Omega$ introduced in \eqref{def-Omega}.
 \end{thm}

\subsection{Spherical representations}
Our classification theorems, Theorem \ref{CT-ns} and Theorem \ref{CT-sym},  can equivalently be formulated as a classification  of {\it spherical representations} of the infinite dimensional Cartan motion groups 
\begin{align*}
\Mat(\infty, \FF) \rtimes (\GL(\infty, \O_F) \times \GL(\infty, \O_F))\  \mathrm{and} \  \Sym(\infty, \FF) \rtimes \GL(\infty, \O_F)
\end{align*}
respectively. We explain this in more detail for  $\Sym(\infty, \FF) \rtimes \GL(\infty, \O_F)$. 

Recall that $\Sym(n, F) \rtimes \GL(n, \O_F)$ is the semi-direct product of the additive group $\Sym(n, F)$ and the general linear group $\GL(n, \O_F)$. Elements of  $\Sym(n, \FF) \rtimes \GL(n, \O_F)$ are pairs $(A, g), A \in \Sym(n, F), g\in \GL(n, \O_F)$ and the rule of multiplication  is given by 
$$
(A, g ) \cdot (B, h) =  ( A + gBg^{t},    gh). 
$$
The group $\Sym(\infty, \FF) \rtimes \GL(\infty, \O_F)$ is defined in a similar way and is of course the inductive limit of  the sequence $\Sym(n, \FF) \rtimes \GL(n, \O_F)$.  The groups $\Sym(\infty, \FF)$ and  $\GL(\infty, \O_F)$ are canonically identified with subgroups of  $\Sym(\infty, \FF) \rtimes \GL(\infty, \O_F)$ by the following embeddings 
$$
A \mapsto (A, 1) \an    g \mapsto (0, g), 
$$
where $A \in \Sym(\infty, \FF)$ and $g \in \GL(\infty, \O_F)$.

A unitary representation $\rho$ of $\Sym(\infty, \FF) \rtimes \GL(\infty, \O_F)$ in a Hilbert space $H(\rho)$  is called spherical if it is irreducible and the subspace $H(\rho)^{\GL(\infty, \O_F)}$ of $\GL(\infty, \O_F)$-invariant vectors in $H(\rho)$  is nontrivial; in which case, by irreducibility,   the subspace $H(\rho)^{\GL(\infty, \O_F)}$ has dimension one.  A vector $h\in H(\rho)^{\GL(\infty, \O_F)}$ of norm 1 is called a {\it spherical vector} of $\rho$  and the function
$$
\varphi_\rho(g) : = (\rho(g)h, h), \quad g \in \GL(\infty, \O_F)
$$
is called the {\it spherical function} of $\rho$. The spherical function $\varphi_\rho$ is  an invariant of the spherical representation $\rho$ and it uniquely determines $\rho$. As a bi-invariant function with respect to the subgroup $\GL(\infty, \O_F)\subset   \Sym(\infty, \FF) \rtimes \GL(\infty, \O_F)$, the function $\varphi_\rho$ is uniquely determined by its restriction $\varphi_\rho|_{\Sym(\infty, \FF)}$.

Given an ergodic $\GL(\infty, \O_F)$-invariant probability measure $\nu$ on $\Sym(\N, F)$, one may define a spherical representation $\rho_\nu$  in the Hilbert space $L^2(\Sym(\N, F), \nu)$ as follows: 
\begin{align*}
&(\rho(g) \xi) (S) = \xi ( g^{-1}S), \quad \quad  g  \in \GL(\infty, \O_F), 
\\
&(\rho(A) \xi)(S)  = \chi(\tr(AS)) \xi(S), \quad A \in \Sym(\infty, \FF),
\end{align*}
where $\xi\in L^2(\Sym(\N, F), \nu)$ and $S\in \Sym(\N, F)$.  The spherical vector can be chosen as the constant function $\xi_0(S) \equiv 1$.  
\begin{prop}\label{prop-spherical}
The map $\nu \mapsto \rho_\nu$ defines a bijection between the set of ergodic $\GL(\infty, \O_F)$-invariant probability measures on $\Sym(\N, F)$ and the set of spherical representations of the group $ \Sym(\infty, \FF) \rtimes \GL(\infty, \O_F)$.
\end{prop}
The proof of Proposition \ref{prop-spherical} is the same as that of Olshanski and Vershik  \cite[Proposition 1.5]{OV-ams96}.

\subsection{An outline of the argument}

Our argument relies on the Vershik-Kerov ergodic method in the spirit of Olshanski and Vershik  \cite{OV-ams96}. The  implementation of individual steps is, however, quite different. In the case of measures on $\Mat(\N, F)$ and the case of measures on $\Sym(\N, F)$, the main steps are
\begin{itemize}
\item {\it Explicit construction of ergodic measures, see Definition \ref{defn-rm} and Definition \ref{defn-sym-1}.}
\item {\it The  asymptotic formulae for the analogues of Harish-Chandra--Izykson-Zuber orbital integrals, see Theorem  \ref{thm-asy-mul}, Theorem \ref{thm-uam} in the case of measures on $\Mat(\N, F)$ and Theorem \ref{thm-sym-asy}, Theorem \ref{thm-sym-uam} in the case of measures on $\Sym(\N, F)$.} 
\item {\it Proof of completeness of the lists of ergodic measures, see Theorem \ref{ct-ns} and Theorem \ref{ct-sym} respectively.}   
\end{itemize}

We now explain our method in greater detail in the  case of measures on $\Sym(\N, F)$. 

{\flushleft \it 1) The Vershik-Kerov method: approximation of ergodic measures by orbital measures.}
While we follow the general scheme of Vershik and Kerov, a number of details are different. 

 Given $x\in\Sym(\N, F)$ and $n\in\N$, let  $m_{\GL(n, \O_F)}(x)$  denote the unique $\GL(n, \O_F)$-invariant probability measure on $\Sym(\N, F)$  supported on the orbit $\GL(n, \O_F)\cdot x \subset \Sym(\N, F)$.  Let $\ORB_\infty(\Sym(\N, F))$ be the class of  probability measures $\nu$ on $\Sym(\N, F)$ verifying the  condition:  there exists a sequence of positive integers $n_1< n_2 < \cdots$  and  a sequence $(\nu_{n_k})_{k\in\N}$  of probability measures with $\nu_{n_k}$ being a $\GL(n_k, \O_F)$-orbital measure supported on $\Sym(n_k, F) \subset \Sym(\N, F)$,  such that $\nu_{n_k}$ converges weakly to $\nu$. 

As a variant of Vershik's Theorem (see Theorem \ref{Vershik-thm} below),  we obtain the following  inclusion: 
\begin{align*}
\mathcal{P}_{\mathrm{erg}}(\Sym(\N, F)) \subset \ORB_\infty(\Sym(\N, F)). 
\end{align*}
Note that, a priori, we do not know whether the inverse inclusion holds.  

{\flushleft \it 2) Main ingredients: Classification of $\ORB_\infty(\Sym(\N, F))$.}

{ \it i): Computation of orbital integrals.}

To describe $\ORB_\infty(\Sym(\N, F))$, we need to understand the asymptotic behaviour of the characteristic functions of orbital measures of the compact groups $\GL(n, \O_F)$.  Recalling the assumption on the character $\chi \in \widehat{F}$ in \eqref{ass-chi-intro},  we obtain an asymptotic formula for the following orbital integral:
\begin{align}\label{orbital-int-intro}
 \int\limits_{\GL(n, \O_F)}    \chi   (\tr ( g  \cdot  \diag ( x_1, \cdots, x_n)  \cdot g^t  \cdot \diag(a_1, \cdots, a_r, 0, \cdots )) d g, 
\end{align}
where $dg$ is the normalized Haar measure of  $ \GL(n, \O_F)$.  The formula we obtain for the integral \eqref{orbital-int-intro} is {\it uniformly asymptotically multiplicative},  that is, the orbital integral  \eqref{orbital-int-intro} has the same asymptotic behaviour, uniformly on the choices of $x_1, \cdots, x_n$,  as the following product of much simpler orbital integrals:
\begin{align*}
 \prod_{j=1}^r \int\limits_{\GL(n, \O_F)}    \chi   (\tr ( g  \cdot  \diag ( x_1, \cdots, x_n)  \cdot g^t  \cdot \diag(a_j, 0, 0, \cdots )) d g.
\end{align*}
See Theorem \ref{thm-sym-uam} for the details.  

Explicit computation of the above orbital integral requires some Fourier analysis on the field $\FF$ and quite a few combinatorial  arguments in which we compute the cardinality of various sets of matrices over the finite field $\F_q$.

{ \it ii): Multiplicativity of characteristic functions for limits of orbital measures.}

 An immediate consequence of the uniform asymptotic multiplicativity for the orbital integral \eqref{orbital-int-intro} is that for any  $\nu \in\ORB_\infty(\Sym(\N, F))$, its characteristic function $\widehat{\nu}$ possesses an exact multiplicativity property, that is, for any $r \in\N$ and any $x_1, \cdots, x_r \in F$, 
 \begin{align}\label{intro-mul}
  \widehat{\nu} (\diag (x_1, \cdots, x_r, 0, 0, \cdots)) = \prod_{j=1}^r  \widehat{\nu}  (x_j e_{11}),
  \end{align}
  where $e_{11}$ is the elementary matrix whose $(1,1)$-coefficient is $1$.  
   This multiplicativity result implies in particular that the classification of the class $\ORB_\infty(\Sym(\N, F))$  is reduced  to the classification of the class  of functions on $F$ defined by $x \mapsto \widehat{\nu}(x e_{11)}$.

{\flushleft \it  3) Ergodicity: Proof of  the inclusion: \begin{align*}\ORB_\infty(\Sym(\N, F)) \subset \mathcal{P}_{\mathrm{erg}}(\Sym(\N, F)).\end{align*}}
Our direct proof of ergodicity for all measures in $\ORB_\infty(\Sym(\N, F))$ uses an  argument of Okounkov and Olshanski \cite{OkOl}: ergodicity for measures in $\ORB_\infty(\Sym(\N, F))$ is derived from the De Finetti Theorem, see Theorem \ref{thm-erg-2} below for the details.  This approach of proving ergodicity can be also applied to different situations, such as  that of Olshanski and Vershik in \cite{OV-ams96}.

{\flushleft \it 4) Proof of the equality \begin{align}\label{equ-erg-list}\mathcal{P}_{\mathrm{erg}}(\Sym(\N, F))  = \{\nu_{\mathbbm{h}}: h \in \Omega\}.\end{align}}
By comparing the characteristic functions $\widehat{\nu_{\mathbbm{h}}}$ for all $\Omega$ and that of measures in $\ORB_\infty(\Sym(\N, F))$, we obtain the equality 
\begin{align*}
\ORB_\infty(\Sym(\N, F)) = \{\nu_{\mathbbm{h}}: \mathbbm{h} \in \Omega\}.
\end{align*}
Combining this equality with the results obtained in the previous steps, we finally get the desired equality \eqref{equ-erg-list}.

\subsection{Organization of the paper}
The exposition, which we tried to make essentially self-contained, is organized as follows. 

In \S \ref{sec-pre}, we  recall the definition of ergodic measures and  the necessary definitions related to non-discrete locally compact non-Archimedean fields, linear groups over them and Fourier transforms in this setting.

In \S \ref{sec-inv-erg},  we prove that all the measures on $\Mat(\N, F)$ from the family 
$\{ \mu_{\mathbbm{k}} | \mathbbm{k} \in \Delta \}$ introduced in Definition \eqref{defn-rm}  are $\GL(\infty, \O_F) \times \GL(\infty, \O_F)$-invariant and ergodic and that all the measures on $\Sym(\N, F)$ from the family 
$\{ \nu_{\mathbbm{h}} | \mathbbm{h} \in \Omega \}$ introduced in Definition \eqref{defn-sym-1}  are $\GL(\infty, \O_F)$-invariant and ergodic.

In  \S \ref{sec-explicit}, we give explicit formulae for characteristic functions of measures from the two families  $\{\mu_{\mathbbm{k}}:  \mathbbm{k}\in\Delta\}$ and $\{\nu_{\mathbbm{h}}:  \mathbbm{h}\in\Omega\}$. 

In \S \ref{sec-unique}, we prove that the parametrization maps $\mathbbm{k} \to \mu_{\mathbbm{k}} $  from $\Delta$ to $\mathcal{P}_{\mathrm{erg}}(\Mat(\N, F))$ and $\mathbbm{h} \mapsto \nu_{\mathbbm{h}}$ from $\Omega$ to $\mathcal{P}_{\mathrm{erg}}(\Sym(\N, F))$ are  injective.

In \S \ref{sec-erg}, we  introduce  orbital measures and recall the Vershik-Kerov ergodic method for dealing with ergodic measures for inductively compact groups. 

In \S \ref{sec-ch},  we obtain the asymptotic formula for orbital integrals of the type \eqref{orbital-int-intro}. 

In \S \ref{sec-cl-F}, we complete the classifications by proving that  the parametrization maps $\mathbbm{k} \to \mu_{\mathbbm{k}} $ and $\mathbbm{h} \mapsto \nu_{\mathbbm{h}}$ are  surjective.

In \S \ref{sec-cont-p}, we show that   the parametrization maps $\mathbbm{k} \to \mu_{\mathbbm{k}} $   and $\mathbbm{h} \mapsto \nu_{\mathbbm{h}}$  are  homeomorphisms between corresponding topological spaces. 

Proofs of some  routine technical lemmata are given in the appendix.

\subsection{Acknowledgements}
We are deeply grateful to  Grigori Olshanski for helpful discussions.    We are deeply grateful to Alexei Klimenko for useful discussions and for suggesting to us a simpler proof of Lemma \ref{lem-inj-con} than our original one. 

The research of A. Bufetov on this project has received funding from the European Research Council (ERC) under the European Union's Horizon 2020 research and innovation programme under grant agreement No 647133 (ICHAOS).
It was also supported  by a subsidy granted to the HSE by the Government of the Russian Federation for the implementation of the Global Competitiveness Programme.  Y.Qiu is supported by the grant IDEX UNITI - ANR-11-IDEX-0002-02, financed by Programme ``Investissements d'Avenir'' of the Government of the French Republic managed by the French National Research Agency.
This work started  as ``research in pairs" at the CIRM, and we are deeply grateful to the  CIRM for its warm hospitality.

\section{Preliminaries}\label{sec-pre}

\subsection{Ergodic measures}
Let $\mathcal{X}$ be a Polish space, that is, it is homeomorphic to a complete metric space that has a countable dense subset. Denote by $\mathcal{P}(\X)$ the set of Borel probability measures on $\X$.   Denote  by $C_b(\X)$  the space of bounded continuous complex-valued functions on $\X$.   Recall that a sequence  $(\mu_n)_{n\in\N}$ in $\mathcal{P}(\X)$  is said to converge weakly to  $\mu \in \mathcal{P}(\X)$ and is denoted by $\mu_n \Longrightarrow \mu$ if for any $f\in C_b(\X)$, we have
 \begin{align*}
\lim_{n\to\infty }\int_{\X} f(x) \mu_n(d x)= \int_{\X} f(x) \mu(d x).
\end{align*}

Given a group action of a  group $G$ on $\mathcal{X}$, we denote by $\mathcal{P}_{\mathrm{inv}}^{G}(\X)$ the set of $G$-invariant Borel probability measures on $\X$. By definition,  a $G$-invariant Borel probability measure $\mu \in \mathcal{P}_{\mathrm{inv}}^{G}(\X)$  is ergodic,  if  for any $G$-invariant Borel subset $\mathcal{A} \subset \mathcal{X}$, either $\mu(\mathcal{A}) =0$ or $\mu(\X\setminus \mathcal{A}) = 0$. The totality of ergodic $G$-invariant probability measures on $\X$ is denoted by  $\mathcal{P}_{\mathrm{erg}}^{G}(\X)$. If the group action  is clear from the context, we denote $\mathcal{P}_{\mathrm{inv}}^{G}(\X)$ and $\mathcal{P}_{\mathrm{erg}}^{G}(\X)$ simply by $\mathcal{P}_{\mathrm{inv}}(\X)$ and $\mathcal{P}_{\mathrm{erg}}(\X)$ respectively.

%%%%%%%%%
%%%%%%%%%
%%%%%%%%%
%%%%%%%%%
%%%%%%%%%
%%%%%%%%%
%%%%%%%%%
%%%%%%%%%
%%%%%%%%%
%%%%%%%%%
%%%%%%%%%
%%%%%%%%%
%%%%%%%%%
%%%%%%%%%

\subsection{Fields and integers}

Let $\FF$ be a non-discrete locally compact  non-Archimedean  field.   The classification of local fields (see, e.g.,  Ramakrishnan and Valenza's book \cite[Theorem 4-12]{DR-fourier}) implies that $F$ is isomorphic to one of the following fields: 
\begin{itemize}
\item a finite extension of the field $\Q_p$ of $p$-adic numbers for some prime $p$.
\item the field of formal Laurent series  over a finite field.
\end{itemize}

Let $| \cdot |$ be the  {\em absolute value} on $\FF$ and denote by $d(\cdot, \cdot)$ the ultrametric on $F$ defined by $d(x, y) = | x-y|$.     The ring of integers in $\FF$ is given by 
\begin{align*}
\O_F: = \{x\in \FF: |x| \le 1\}.
\end{align*}
 The subset $\mathfrak{m}:  =\{x\in \FF: | x| < 1\}$
 is the unique maximal  ideal of $\O_F$.  The ideal $\mathfrak{m}$ is principal. Any generator of $\mathfrak{m}$ is called a {\it uniformizer} of $F$. Throughout the paper, we fix any uniformizer $\varpi$ of $F$, that is,  $\mathfrak{m}= \varpi \O_F$.
 The field  $\O_F/\varpi\O_F$ is  finite with $q = p^f$ elements for a prime number $p$ and a positive integer $f \in\N$.     If $q = 2^f$, then we say that $F$ is {\it dyadic}, otherwise, we say that $F$ is {\it non-dyadic}. 
 
 We denote  $\F_q : =\O_F/\varpi\O_F.$
The quotient map is denoted by  
\begin{align*}
\pi: \O_F \rightarrow \F_q= \O_F/\varpi \O_F.
\end{align*} 
Fix a {\it complete set of representatives}  $ \mathcal{C}_q \subset \O_F$ of cosets of $\varpi\O_F$  in $\O_F$ and assume that $0\in \mathcal{C}_q$. The restriction of the quotient map $\pi$ on the finite set $\mathcal{C}_q$ is a bijection: 
 \begin{align}\label{pi-1}
 \pi: \mathcal{C}_q \xrightarrow{\text{bijection}} \F_q. 
 \end{align}
 
  Any element of $\FF$ is uniquely expanded as a convergent series in $\FF$: \begin{align}\label{p-adic-expansion}
    x= \sum_{n= v}^\infty a_n \varpi^{n} \qquad (v\in \mathbb{Z}, a_n \in \mathcal{C}_q, a_{v}\neq 0).
\end{align}
If $x\in F$ is given by the series \eqref{p-adic-expansion}, then we define the {\it $F$-valuation} of $x$ by $\mathrm{ord}_F(x): = v.$
By convention, we set $\mathrm{ord}_F(0) = \infty$. The absolute value and the $F$-valuation of any element $x\in F$ are related by the formula $|x|=q^{-\mathrm{ord}_F(x)}.$

\subsection{Group actions}

Let $\GL(n, \FF)$ and $\GL(n, \O_F)$ denote the groups of invertible $n \times n$ matrices over $\FF$ and $\O_F$ respectively.  
 $\GL(n, \O_F)$  is embedded naturally into $\GL(n+1, \O_F)$ by 
\begin{align}\label{group-em}
a \in \GL(n, \O_F) \mapsto    
\left(
\begin{array}{cc} 
a & 0 
\\
 0 & 1
\end{array}
\right)
\in \GL(n+1, \O_F).
\end{align}
Define an inductive limit group 
\begin{align*}
\GL(\infty, \O_F): = \lim\limits_{\longrightarrow} \GL(n, \O_F).
\end{align*}
 Equivalently, $\GL(\infty, \O_F)$ is the group of  infinite invertible matrices $g = (g_{ij})_{i, j \in \N}$ over $\O_F$ such that $g_{ij} = \delta_{ij}$ if  $i +j$ is large enough.

Let $\Mat(n, \FF)$ and $\Mat(n, \O_F)$ denote the spaces of all $n \times n $ matrices over $\FF$ and $\O_F$ respectively.  Define 
$$
\Mat(\N, \FF) : = \{ X= (X_{ij})_{i, j \in \N}| X_{ij} \in \FF\}. 
$$
Let $\Mat(\infty, \FF)$ denote the subspace of $\Mat(\N,\FF)$ consists of matrices whose all but a finite number of  coefficients are zeros. Define also 
 \begin{align*}
\Sym(n, \FF): &= \{ X\in \Mat(n, \FF) | X_{ij}  = X_{ji},  \forall 1 \le i, j \le n\},
 \\
 \Sym(\N, \FF): & = \{ X\in \Mat(\N, \FF) | X_{ij}  = X_{ji},  \forall  i, j \in \N\}. 
 \end{align*}
Set $\Sym(\infty, \FF ) : = \Sym(\N, \FF) \cap \Mat(\infty, \FF)$.

Two natural group actions under consideration in this paper are the following:
\begin{itemize}
\item The  group action of $\GL(\infty, \O_F) \times \GL(\infty, \O_F)$ on $\Mat(\N, \FF)$ defined by: 
\begin{align*}
\qquad ((g_1, g_2), M ) \mapsto  g_1M g_2^{-1}, \quad g_1, g_2 \in \GL(\infty, \O_F), M \in \Mat(\N, F).  
\end{align*}
\item The group action of  $\GL(\infty, \O_F)$ on $\Sym(\N, \FF)$ defined by:
\begin{align*}
(g, M) \mapsto g M g^{t},  \quad  g \in \GL(\infty, \O_F), S \in \Sym(\N, F) ,
\end{align*}
where  $g^t$ is the transposition of $g$.
\end{itemize}

\subsection{Conventions}

Given a finite set $B$, we denote $\# B$ its cardinality.

Let  $(\Sigma,  \mathcal{B}, m)$ be a measured space and let $f$ be a real or complex valued integrable function  defined on $\Sigma$ . If $A\subset\Sigma$ is  measurable and $0 < m(A) < \infty$, then we denote 
\begin{align}\label{dashint}
\dashint_A f(x) dm(x):  = \frac{1}{m(A)} \int_A f(x)dm(x).
\end{align}

For any random variable $Y$, we denote its distribution by $\mathcal{L}(Y)$.

 Conventions concerning the empty set $\emptyset$: let $(r_i)_{i\in I}$ be a family of real numbers (or complex numbers for the last two formulae), we set

$$
\inf_{i \in \emptyset} r_i = +\infty, \quad  \sup_{i \in \emptyset} r_i  = -\infty, \quad \sum_{i \in\emptyset} r_i = 0, \quad \prod_{i \in\emptyset} r_i = 1.
$$ 

The following conventions will also  be used:
\begin{itemize}
\item As elements in $\FF$: $\varpi^{\infty} = \varpi^{+\infty} = 0\in\FF.$ 
\item As elements in $\R\cup\{+\infty\}$:  $q^{\infty}  = q^{+\infty} = +\infty \an q^{-\infty} = 0\in\R.$
\end{itemize}

%%%%%%%%%
%%%%%%%%%
%%%%%%%%%
%%%%%%%%%
%%%%%%%%%
%%%%%%%%%
%%%%%%%%%
%%%%%%%%%
%%%%%%%%%
%%%%%%%%%
%%%%%%%%%
%%%%%%%%%
%%%%%%%%%
%%%%%%%%%
\subsection{Haar measure on $\GL(n, \O_F)$}

For any $n \in \N$, denote by $d\vol_n $  the Haar measure on $\FF^n$ normalized by  the condition $\vol_n (\O_F^n) = 1.$ If there is no confusion, we will use the simplified notation $\vol(\cdot)$ for $\vol_n(\cdot)$.  

\begin{rem}\label{rem-mp}
The Haar measure $\vol_n$ on $F^n$ is preserved by any linear map represented by a matrix from  the group $\GL(n, \O_F)$.
\end{rem}

For any $n$, we fix a Haar measure $\vol(\cdot)$ on $\Mat(n, \FF)$ normalized by $\vol(\Mat(n, \O_F)) =1.$ Upper to a multiplicative constant, the Haar measure on the locally compact group $\GL(n, \FF)$ is uniquely given (see, e.g., Neretin \cite{Neretin-p-adic-Hua}) by
\begin{align}\label{haar-linear}
|\det(M)|^{-n} \cdot \vol(dM).
\end{align}

Let  $\GL(n, \F_q)$ be the group of invertible $n\times n$ matrices over $\F_q$.  Set
\begin{align*}
\GL(n, \mathcal{C}_q) :  = \{t= (t_{ij})_{1\le i, j \le n}\in  \GL(n, \O_F): t_{ij} \in \mathcal{C}_q\}. 
\end{align*}

\begin{prop}\label{prop-T}
A standard partition of $\GL(n, \O_F)$ is given by 
\begin{align}\label{partition}
\GL(n, \O_F) = \bigsqcup_{t \in \GL(n,\, \mathcal{C}_q)} (t + \Mat(n, \varpi \O_F)).
\end{align}
In particular, we have
\begin{align}\label{vol-gl}
\vol(\GL(n, \O_F) ) = \prod_{j=1}^n ( 1- q^{-j}).
\end{align}
\end{prop}

\begin{proof}
By definition, $a = (a_{ij})_{1\le i, j \le n} \in \GL(n, \O_F)$ implies that $a_{ij} \in \O_F$ and
$| \det (a) | =1$.  Now take any $x \in \Mat(n, \varpi \O_F)$. First, we have  $ a + x\in \Mat(n, \O_F)$. Second, write $x =  \varpi y $ with $y \in \Mat(n, \O_F)$. By  definition, there exists $z \in \O_F$, such that  $\det(a + x) = \det(a) + \varpi z. $
Since $|\varpi z| \le q^{-1}$, by ultrametricity , we obtain 
$
|\det(a +x)|  = |\det(a)  + \varpi z| =1,
$
whence  $a+ x\in \GL(n, \O_F)$ and  the set on the right hand side of \eqref{partition} is contained in $\GL(n,\O_F)$.   Conversely, since $\mathcal{C}_q$ is a complete set of representatives of the cosets of $\varpi \O_F$ in $\O_F$, for any $A\in\GL(n,\O_F)$, there exists a unique $t\in\GL(n,\mathcal{C}_q)$, such that  $A \equiv t (\modulo  \varpi \O_F). $
This completes the proof of \eqref{partition}. 

Recalling that 
\begin{align*}
  \# \GL(n, \mathcal{C}_q)  =  \# \GL(n, \F_q) = \prod_{j = 0}^{n-1} (q^n-q^j), 
\end{align*}
we arrive at \eqref{vol-gl}. 
\end{proof}

By Proposition \ref{prop-T}, $\GL(n, \O_F)$ is an open subgroup of $\GL(n, \FF)$. It follows that the restriction on $\GL(n, \O_F)$ of the Haar measure \eqref{haar-linear} on $\GL(n, \FF)$ is a Haar measure on $\GL(n, \O_F)$. Consequently, the normalized Haar measure on $\GL(n, \O_F)$ is given by 
\begin{align}\label{haar}
 \frac{\vol(\cdot)}{\prod_{j=1}^n ( 1- q^{-j})}.
\end{align}

Let  $T(n)$ be sampled uniformly from  the finite set $\GL(n, \mathcal{C}_q)$; let $V(n)$ be sampled with respect to the normalized Haar measure on  $\Mat(n, \varpi \O_F)$ and independent of $T(n)$.
\begin{prop}\label{prop-haar-rm}
The random matrix $T(n) + V(n)$ is a Haar random matrix on $\GL(n, \O_F)$;  that is, the distribution law $\mathcal{L}(T(n) + V(n))$ coincides with the normalized Haar measure on $\GL(n, \O_F)$. 
\end{prop}

\begin{proof}
The proof follows immediately from Proposition \ref{prop-T} and the formula \eqref{haar} for the normalized Haar measure on $\GL(n, \O_F)$. 
\end{proof}
%%%%%%%%%
%%%%%%%%%
%%%%%%%%%
%%%%%%%%%
%%%%%%%%%
%%%%%%%%%
%%%%%%%%%
%%%%%%%%%
%%%%%%%%%
%%%%%%%%%
%%%%%%%%%
%%%%%%%%%
%%%%%%%%%
%%%%%%%%%
\subsection{Diagonalization  in $\Mat(n, F)$}

\begin{lem}[{See, e.g., Neretin \cite[\S 1.3]{Neretin-p-adic-Hua}}]\label{lem-s-num}
Every matrix $A \in \Mat(n, \FF)$ can be written in the form
\begin{align}\label{A-dec}
A = a \cdot \diag (\varpi^{-k_1}, \varpi^{-k_2}, \cdots, \varpi^{-k_n}) \cdot b, \quad (a, b \in \GL(n, \O_F)),
\end{align}
where  $k_1 \ge k_2 \ge \cdots \ge k_n \ge - \infty$ and $\diag (\varpi^{-k_1}, \varpi^{-k_2}, \cdots, \varpi^{-k_n})$ is the diagonal matrix with diagonal coefficients $\varpi^{-k_1}, \varpi^{-k_2}, \cdots, \varpi^{-k_n}$.  Moreover, the $n$-tuple $(k_1, k_2, \cdots, k_n)$ is uniquely determined by the matrix $A$.
\end{lem}

\subsection{Square-units in $F$}
The group of units of the ring $\O_F$ is given by $\O_F^\times :  = \{x\in \FF: |x| = 1\}$.
Let $(\O_F^{\times})^2$ be the subgroup of  $\O_F^{\times}$ defined by 
\begin{align*}
(\O_F^{\times})^2: = \{ x \in \O_F^\times: \text{there exists $a\in F$ such that $x = a^2$} \}.
\end{align*}
Let $\mathcal{C}_q^{\times}$ denote the set $\mathcal{C}_q\setminus \{0\}$ and define
\begin{align*}
(\mathcal{C}_q^{\times})^2 : = \{ a\in \mathcal{C}_q^{\times} |\text{ there exists $ y \in F$ such that $a = y^2$} \}. 
\end{align*}
 Denote 
\begin{align*}
(\F_q^{\times})^2 : = \{ a \in \F_q^{\times} |\text{ there exists $ c \in \F_q^{\times}$ such that $a= c^2$} \}.
\end{align*}

\begin{lem}\label{lem-clopen}
Assume that $F$ is non-dyadic.  Then 
\begin{align}\label{square-st}
(\O_F^{\times})^2 =  \bigsqcup_{a \in (\mathcal{C}_q^{\times})^2} (a+ \varpi \O_F). 
\end{align}
Consequently,  the map $\pi$ in \eqref{pi-1} induces a bijection: 
\begin{align}\label{pi-2}
\pi: (\mathcal{C}_q^{\times})^2 \xrightarrow{\text{bijection}} (\F_q^{\times})^2. 
\end{align}
\end{lem}

\begin{proof}
Since $F$ is non-dyadic, we have $|2| = 1$. For any $x = \alpha^2 \in (\O_F^{\times})^2$,  since $\mathcal{C}_q$ is a complete set of representatives for $\O_F/\varpi\O_F$,  there exists $a\in \mathcal{C}_q^{\times}$, such that $x \equiv a ( \modulo \varpi \O_F)$, that is,   
\begin{align}\label{x-a-1}
| x  - a | <1.
\end{align}
 Take any $b \in \O_F^{\times}$  such that $| b - x | <1  =  | 2 \alpha|^2$.
Then the polynomial $P_b(X) = X^2 - b \in \O_F[X]$ satisfies 
\begin{align}\label{hensel-ass}
| P_b(\alpha) | < | P_b'(\alpha)|^2.
\end{align}
 By Hensel's lemma (see Cassels \cite[pp. 49-51]{Cassels}), the inequality \eqref{hensel-ass} implies that there exists $\beta \in F$ such that
 \begin{align}\label{root}
  P_b(\beta) = \beta^2 - b= 0 \an | \beta - \alpha | \le   \frac{|P_b(\alpha)|}{|P_b'(\alpha)|}  < | P_b'(\alpha)|= 1.
 \end{align}
In particular,  we have
  \begin{align}\label{x-ball}
  \{b \in \O_F^{\times}:  | b - x | < 1 \} = x+ \varpi \O_F \subset  (\O_F^{\times})^2. 
  \end{align}
  Combining \eqref{x-a-1} and \eqref{x-ball}, we get $a \in  (\O_F^{\times})^2$. Hence
\begin{align*}
(\O_F^{\times})^2\subset  \bigsqcup_{a \in (\mathcal{C}_q^{\times})^2} (a+ \varpi \O_F). 
\end{align*}
Conversely, since  $(\mathcal{C}_q^{\times})^2\subset (\O_F^{\times})^2$, for any $a\in (\mathcal{C}_q^{\times})^2$, replacing $x$ by $a$ in the above argument, the inclusion \eqref{x-ball} implies  $a+ \varpi \O_F \subset  (\O_F^{\times})^2$.  Hence
\begin{align*}
 \bigsqcup_{a \in (\mathcal{C}_q^{\times})^2} (a+ \varpi \O_F) \subset (\O_F^{\times})^2.
\end{align*}
\end{proof}

\begin{rem}\label{rem-card2}
Recall that if $F$ is non-dyadic, then the quotient group $\O_F^{\times} / (\O_F^{\times})^2$ has two elements.
\end{rem}

\subsection{Diagonalization  in $\Sym(n, F)$}

In what follows, we fix a non-square unit $\varepsilon \in \O_F^{\times} \setminus (\O_F^{\times})^2$. By Remark \ref{rem-card2}, the following set
\begin{align}\label{R-T}
\mathcal{T}: =   \{ \varpi^{-k}  | k \in \Z\} \sqcup  \{ \varpi^{-k} \varepsilon | k \in \Z\} \sqcup\{0\}
\end{align}
is a complete set of representatives for the quotient $F/(\O_F^{\times})^2$. 

\begin{lem}\label{lem-q-diag}
Assume that $F$ is non-dyadic. Then  any symmetric matrix $A\in \Sym(n, F)$ can be written in the form 
\begin{align*}
A = g \cdot \diag (x_1, \cdots, x_n) \cdot g^t, \quad x_1, \cdots, x_n \in \mathcal{T},  g \in \GL(n, \O_F). 
\end{align*}
\end{lem}

\begin{proof}
Since $\mathcal{T}$ is a complete set of representatives for the quotient $F/(\O_F^{\times})^2$, it suffices to show that any symmetric matrix  $A \in \Sym(n, F)$ is diagonalizable.   Assume that $A$ is not a zero matrix. We claim that, up to passing $A$ to $gAg^t$ for some $g\in \GL(n, \O_F)$, we may assume that the $(1,1)$-coefficient of $A$ has maximal absolute value. 

{\it Case 1:} Assume first that  there exists $1 \le i_0 \le n$, such that  $|A_{i_0i_0}| = \max_{1\le i, j \le n} | A_{ij}|$. If $i_0 = 1$, then there is nothing to prove. Otherwise,  take $g= M_{(1i_0)}$, where $M_{(1i_0)}$ is the permutation matrix associated with the transposition $(1i_0)$.  Then the matrix $gAg^t$ has $A_{i_0i_0}$ as its $(1,1)$-coefficient. 

 {\it Case 2: } Now assume that there exists $i_0 < j_0$ such that 
\begin{align}\label{off-diag-max}
|A_{i_0j_0}| = \max_{1\le i, j \le n}| A_{ij}| > \max_{1\le i \le n} | A_{ii}|.
\end{align} 
Let us do operations on the submatrix indexed by $\{i_0, j_0\} \times \{i_0, j_0\}$ as follows:
\begin{align*}
& \left[\begin{array}{cc}  2 A_{i_0j_0}  + A_{i_0i_0}  + A_{j_0j_0}  &  A_{i_0j_0}  + A_{j_0j_0} \\ A_{i_0j_0}  + A_{j_0j_0} &   A_{j_0j_0} \end{array}\right] 
\\
&= \left[\begin{array}{cc} 1 & 1 \\ 0 & 1 \end{array}\right] \left[\begin{array}{cc} A_{i_0i_0} & A_{i_0j_0} \\ A_{i_0j_0} & A_{j_0j_0} \end{array}\right] \left[\begin{array}{cc} 1 & 0 \\ 1 & 1 \end{array}\right].
\end{align*}
Since $F$ is non-dyadic and non-Archimedean,  \eqref{off-diag-max} implies 
\begin{align*}
|2 A_{i_0j_0}  + A_{i_0i_0}  + A_{j_0j_0} | = |  A_{i_0j_0}  + A_{j_0j_0}  | =  | A_{i_0j_0}|.
\end{align*}
This shows that we can reduce the second case to the first case where a diagonal coefficient has maximal absolute value.  

Now assume that 
\begin{align*}
A = \left[\begin{array}{cc} x & c^t \\ c & A_1 \end{array}\right],  \quad c\in F^{n-1},
\end{align*}
such that $x$ attains the  maximal absolute value of all coefficients of $A$. Then $x^{-1}c \in \O_F^{n-1}$ and  we have
\begin{align*}
 \left[\begin{array}{cc} 1&0 \\  - x^{-1}c & 1 \end{array}\right] \left[\begin{array}{cc} x & c^t \\ c & A_1 \end{array}\right]  \left[\begin{array}{cc} 1 & - x^{-1} c^t \\ 0& 1 \end{array}\right] =  \left[\begin{array}{cc} x & 0 \\ 0 & A_1 - x^{-1} cc^t\end{array}\right].
\end{align*} 
By continuing the above procedure on the submatrix $A_1 - x^{-1} cc^t$, we prove finally that $A$ is diagonalizable. 
\end{proof}

\begin{rem}
The assumption that $F$ is non-dyadic is necessary in Lemma \ref{lem-q-diag}. Indeed, if 
\begin{align*}
\left[\begin{array}{cc} \lambda_1 & 0 \\ 0 & \lambda_2 \end{array}\right] = \left[\begin{array}{cc} a & b \\ c & d \end{array}\right] \left[\begin{array}{cc} 0 & 1 \\ 1 & 0 \end{array}\right] \left[\begin{array}{cc} a & c \\ b & d \end{array}\right], \quad  \left[\begin{array}{cc} a & c \\ b & d \end{array}\right] \in \GL(2, \O_F),
\end{align*}
then 
\begin{align*}
bc + ad = 0, \quad ad - bc \in \O_F^{\times}.
\end{align*}
It follows that $2 ad \in \O_F^{\times}$ and hence $2 \in \O_F^{\times}$. This implies that $F$ is non-dyadic.  
\end{rem}

%%%%%%%%%
%%%%%%%%%
%%%%%%%%%
%%%%%%%%%
%%%%%%%%%
%%%%%%%%%
%%%%%%%%%
%%%%%%%%%
%%%%%%%%%
%%%%%%%%%
%%%%%%%%%
%%%%%%%%%
%%%%%%%%%
%%%%%%%%%

\subsection{Characteristic functions}\label{sec-def-char}
Denote by $\widehat{\FF}$ the Pontryagin dual of the additive group $\FF$. Elements in $\widehat{F}$ are  called characters of $F$.  Throughout the paper, we fix a non-trivial   character $\chi \in \widehat{\FF}$ such that 
\begin{align}\label{assumption-chi}  
\text{  $\chi|_{\O_F} \equiv 1 $ and $\chi$ is  not constant on  $\varpi^{-1} \O_F$.}
\end{align}
For any $y \in \FF$, define a character   $\chi_y \in \widehat{\FF}$ by
$\chi_y(x) =\chi(yx)$.  The map $y\mapsto \chi_{y}$ from $\FF$ to $\widehat{\FF}$ defines a group isomorphism.

We write explicitly characteristic functions of probability measures in the following situations. 

(i) If $\mu$ is a Borel probability measure    on  $F^m$, then $\widehat{\mu}$ is defined on $F^m$  by 
\begin{align*}
\widehat{\mu} (y): =  \int_{\FF^m} \chi( x\cdot y) \mu(dx), 
\end{align*}
where $x \cdot y : = \sum_{j = 1}^m x_j y_j$.

(ii) If $\mu$ is a Borel probability measure    on  $\Mat(n,\FF)$, then $\widehat{\mu}$ is defined on $\Mat(n, F)$  by 
\begin{align*}
\widehat{\mu} (A): =  \int_{\Mat(n,\FF)} \chi(\tr(AM)) \mu(dM).
\end{align*}

(iii) If $\mu$  is  a Borel probability measure   on  $\Mat(\N, F)$, then $\widehat{\mu}$ is defined on $\Mat(\infty, F)$  by 
\begin{align}\label{def-ch-inf}
\widehat{\mu} (A): =  \int_{\Mat(\N,\FF)} \chi( \tr(AM)) \mu(dM).
\end{align}

(iv)  If $\nu$ is a Borel probability measure    on  $\Sym(n,\FF)$,  then $\widehat{\nu}$ is defined on $\Sym(n,\FF)$ by 
\begin{align*}
\widehat{\nu} (A): =  \int_{\Sym(n,\FF)} \chi(\tr(AS)) \nu(dS). 
\end{align*}

(v) If $\nu$  is  a Borel probability measure   on  $\Sym(\N, \FF)$,  then $\widehat{\nu}$ is defined on $\Sym(\infty, F)$ by 
\begin{align*}
\widehat{\nu} (A): =  \int_{\Sym(\N,\FF)} \chi( \tr(AS)) \nu(dS).
\end{align*}

Since the corresponding groups are locally compact, Theorem 31.5 in Hewitt and Ross \cite[p.212]{Hewitt-Ross} implies that in the cases (i) (ii) and (iv), the characteristic function $\widehat{\mu}$ determines $\mu$ uniquely. The same statement holds for the cases (iii) and (v). Indeed,  although the additive groups $\Mat(\N, \FF)$  and $\Mat(\infty, \FF)$ are not locally compact and we can not apply  the result on locally compact groups directly, we may use the fact that any Borel probability  measure $\mu$ on $\Mat(\N, \FF)$ is uniquely determined by its finite dimensional projections $(\Cut_n^\infty)_{*}(\mu)$ and \eqref{def-ch-inf} contains all information for $\widehat{(\Cut_n^\infty)_{*}(\mu)}$, $n = 1, 2, \cdots$. The case (v) is treated similarly.

\begin{rem}\label{rem-fourier-inv}
If $\mu$ is a probability measure on $\Mat(n, \FF)$  which is invariant under the action of the group $\GL(n, \O_F)\times \GL(n, \O_F)$, then for any $a, b \in \GL(n, \O_F)$, we have 
\begin{align}\label{inv-radial}
\widehat{\mu} (a \cdot \diag (\varpi^{-k_1}, \varpi^{-k_2}, \cdots, \varpi^{-k_n}) \cdot b) = \widehat{\mu}(\diag (\varpi^{-k_1}, \varpi^{-k_2}, \cdots, \varpi^{-k_n})).
\end{align}
Similarly, if $\nu$ is a $\GL(n, \O_F)$-invariant probability measure on $\Sym(n,F)$, then for any $g \in \GL(n,\O_F)$, we have
\begin{align}\label{inv-sym-diag}
 \widehat{\nu}􏰀(g \cdot \diag(x_1, \cdots, x_n) \cdot g^t) =\widehat{\nu}􏰀(\diag(x_1, \cdots, x_n)).
\end{align}
Similar statements hold for  $\GL(\infty, \O_F)\times \GL(\infty, \O_F)$-invariant probability measures on $\Mat(\N, \FF)$ and for $\GL(\infty, \O_F)$-invariant probability measures on $\Sym(\N, \FF)$. 
\end{rem}

Let $m\in\N$.  Given any Borel probability measures $\mu_1, \cdots, \mu_m$ on $\Mat(\N, F)$ (resp. $\Sym(\N, F)$), their convolution $\mu_1 * \cdots * \mu_m$ is defined as follows:   let $M_1, \cdots, M_m$ be independent random matrices such that $\mathcal{L}(M_i) = \mu_i, i = 1, \cdots, m$ and set 
\begin{align*}
\mu_1 * \cdots * \mu_m := \mathcal{L} (M_1 + \cdots + M_m). 
\end{align*}
The characteristic function of $\mu_1 * \cdots * \mu_m$ is given by the formula
\begin{align}\label{ch-conv}
(\mu_1 * \cdots * \mu_m)^{\wedge}  = \prod_{i=1}^m \widehat{\mu_i}. 
\end{align}

\section{Invariance and Ergodicity}\label{sec-inv-erg}
In this section, we prove that all the measures on $\Mat(\N, F)$ from the family 
$\{ \mu_{\mathbbm{k}} = \mathcal{L}(M_{\mathbbm{k}})| \mathbbm{k} \in \Delta \}$ introduced in Definition \eqref{defn-rm}  are $\GL(\infty, \O_F) \times \GL(\infty, \O_F)$-invariant and ergodic and that all the measures on $\Sym(\N, F)$ from the family 
$\{ \nu_{\mathbbm{h}} = \mathcal{L}(S_{\mathbbm{h}})| \mathbbm{h} \in \Omega \}$ introduced in Definition \eqref{defn-sym-1}  are $\GL(\infty, \O_F)$-invariant and ergodic.

\subsection{$\GL(\infty, \O_F) \times \GL(\infty, \O_F)$-invariance for $\mu_{\mathbbm{k}}$'s}

\begin{prop}\label{prop-inv}
 For any $\mathbbm{k}\in \Delta$, the probability measure $\mu_{\mathbbm{k}} $ on $\Mat(\N, F)$  is  $\GL(\infty, \O_F) \times \GL(\infty, \O_F)$-invariant.
\end{prop}

Recall that the normalized integral $\dashint$ is introduced in \eqref{dashint}. 

\begin{rem}
For any  $n\ge 1$, we have 
\begin{align}\label{vanishing-ch}
\dashint_{\varpi^{-n} \O_F} \chi(x) dx  = 0. 
\end{align}
Indeed, for any fixed $n\ge 1$, the character $\chi$ defines a non-trivial character  $\widetilde{\chi}$ of the finite group $\Gamma_n: = \varpi^{-n} \O_F /\O_F$ by 
\begin{align*}
\widetilde{\chi}( \gamma ) : = \chi(x) \quad (\text{ if } \gamma  = x + \O_F,  x\in \varpi^{-n} \O_F). 
\end{align*}
By orthogonality of the character $\widetilde{\chi}$ and the trivial character, we have 
\begin{align*}
\dashint_{\varpi^{-n} \O_F} \chi(x) dx = \frac{1}{\# \Gamma_n}\sum_{\gamma \in \Gamma_n } \widetilde{\chi}(\gamma)  =0,
\end{align*}
and \eqref{vanishing-ch} is proved. 
\end{rem}

\begin{lem}\label{ball-fourier}
For any $y \in \FF$ and any $l\in \Z$, we have
\begin{align}\label{b-fourier}
\dashint_{  \varpi^l \O_F }  \chi(xy) dx = \mathbbm{1}_{\varpi^{-l}\O_F} (y).
\end{align}
\end{lem}

\begin{proof}
First assume that $y \in \varpi^{-l}\O_F $. Then for any $x\in \varpi^l\O_F$, we have $xy \in \O_F$. Consequently, by \eqref{assumption-chi}, we have
$
\dashint_{  \varpi^l \O_F }  \chi(xy) dx  =1. 
$
Now assume that $y \notin \varpi^{-l}\O_F$. Since the Haar measure $d \vol$ on $\FF$ is invariant under the multiplication action by any element $u \in \O_F^{\times}$, without loss of generality, we may assume that $y = \varpi^k$ with $k \le -l-1$. By \eqref{vanishing-ch}, we have
\begin{align*}
 \dashint_{  \varpi^l \O_F }  \chi(xy) dx = \dashint_{  \varpi^l \O_F }  \chi ( \varpi^k x ) dx = \dashint_{  \O_F }  \chi (\varpi^{k+l} z ) dz  =0. 
\end{align*}
This completes the proof of \eqref{b-fourier}. 
\end{proof}

 \begin{lem}\label{lem-inv-vec}
For any $m \in\N$, the distribution of the random vector $(X_i^{(1)})_{i =1}^m$  is $\GL(m, \O_F)$-invariant. 
\end{lem}
 \begin{proof}
 Denote $X = (X_i^{(1)})_{i =1}^m$. It suffices to prove that for any  $A\in\GL(m,\O_F)$ and any $y= (y_1, \cdots, y_m) \in F^m$, we have 
 \begin{align}\label{X-y-X}
 \E [\chi (X \cdot y)] =  \E [\chi( (A X) \cdot y )]. 
 \end{align}
 By the independence between $X_i^{(1)}: i =1, \cdots, m$ and  Lemma \ref{ball-fourier},  we have 
\begin{align*}
\E [\chi (X \cdot y)]  =\prod_{j = 1}^n\E[\chi( X_j y_j)] = \prod_{j = 1}^n\mathbbm{1}_{\O_F}(y_j) = \mathbbm{1}_{\O_F^m}(y).
\end{align*}
Similarly, 
\begin{align*}
 \E [\chi( (A X) \cdot y )]  = \E [\chi(X \cdot (A^{t}y) )]  =\mathbbm{1}_{\O_F^m}(A^{t}y) = \mathbbm{1}_{\O_F^m}(y). 
\end{align*}
Hence we get  \eqref{X-y-X}. The proof of Lemma \ref{lem-inv-vec} is completed.  
 \end{proof}
 
\begin{proof}[Proof of Proposition \ref{prop-inv}] 
It suffices to prove that the following probability measures
\begin{align*}
\mathcal{L}􏰇[X_i^{(1)} Y_j^{(1)}]_{i, j \in \N}  􏰈 \an \mathcal{L}([Z_{ij}]_{i, j\in\N} ).
\end{align*}
are $\GL(\infty, \O_F) \times \GL(\infty, \O_F)$-invariant. The invariance of both measures follows immediately from Lemma \ref{lem-inv-vec}.
\end{proof}

\subsection{$\GL(\infty, \O_F) \times \GL(\infty, \O_F)$-ergodicity for $\mu_{\mathbbm{k}}$'s}

\begin{thm}\label{thm-erg-1}
 For any $\mathbbm{k}\in \Delta$, the probability measure $\mu_{\mathbbm{k}} $ on $\Mat(\N, F)$  is  $\GL(\infty, \O_F) \times \GL(\infty, \O_F)$-ergodic.
\end{thm}

The   map $M  \mapsto   (M_{ii})_{i\in\N}$ from $\Mat(\N, F)$ to $F^\N$ induces an affine map 
\begin{align*}
\Psi: \mathcal{P}(\Mat(\N, F)) \rightarrow  \mathcal{P}(F^\N) .
\end{align*}
Let $S(n)$ denote the group of  permutations of the set $\{1, 2, \cdots, n\}$ and set
$S(\infty):  = \bigcup_{n\in\N} S(n)$. 
The group $S(\infty)$ acts naturally on $F^\N$ by permutations of coordinates.

\begin{lem}\label{lem-diag-inj1}
For any $\mu \in \mathcal{P}_{\mathrm{inv}}(\Mat(\N, F))$, we have $\Psi(\mu) \in \mathcal{P}_{\mathrm{inv}}^{S(\infty)}(F^\N)$. Moreover, the restriction map 
\begin{align}\label{diag-inj1}
\Psi: \mathcal{P}_{\mathrm{inv}}(\Mat(\N, F)) \rightarrow \mathcal{P}_{\mathrm{inv}}^{S(\infty)}(F^\N)
\end{align}
is an affine embedding. 
\end{lem}
\begin{proof}
Let $\mu \in \mathcal{P}_{\mathrm{inv}}(\Mat(\N, F))$. For any $\sigma \in S(\infty)$, the associated permutation matrix $M_\sigma$, defined by  
\begin{align*}
M_{\sigma} (i, j) : = 1_{\sigma(i)  = j},
\end{align*} is an element in $\GL(\infty, \O_F)$.  By the invariance of $\mu$ under the multiplication by all permutation matrices $M_\sigma, \sigma\in S(\infty)$ on left and on right, it is easy to see that $\Psi(\mu)\in \mathcal{P}_{\mathrm{inv}}^{S(\infty)}(F^\N)$. 

Now we  show that the map \eqref{diag-inj1} is injective. By the definition of pushforward map, the Fourier transform of $\Psi(\mu)$ is given as follows: for any $x_1, \cdots, x_r \in F$, 
\begin{align}\label{fourier-push}
\widehat{\Psi(\mu)} ((x_1, \cdots, x_r, 0, 0, \cdots)) = \widehat{\mu} (\diag (x_1, \cdots, x_r, 0, 0, \cdots)).
\end{align}  By Remark \ref{rem-fourier-inv},  $\widehat{\mu}$ is determined by 
\begin{align*}
\widehat{\mu} (\diag (x_1, \cdots, x_r, 0, 0, \cdots)), \quad x_1, \cdots, x_r \in F.
\end{align*}
Now,  the equality \eqref{fourier-push} implies that $\widehat{\mu}$ and hence $\mu$ itself is determined uniquely by $\Psi(\mu)$.  The injectivity of the map \eqref{diag-inj1} is proved.
\end{proof}

In what follows, for any convex set $C$,  we denote by $Ext(C)$ the set of extreme points of $C$. 

\begin{rem}\label{rem-erg-ext}
For any group action of a group $G$ on a Polish space $\X$, we have 
\begin{align*}
Ext(\mathcal{P}_{\mathrm{inv}}^G(\X)) \subset \mathcal{P}_{\mathrm{erg}}^G(\X).
\end{align*}
The converse inclusion is in general not true. However, the De Finetti Theorem claims that for the group action of $S(\infty)$ on $\X^\N$, we indeed have
\begin{align*}
Ext(\mathcal{P}_{\mathrm{inv}}^{S(\infty}(\X^\N)) = \mathcal{P}_{\mathrm{erg}}^{S(\infty)}(\X^\N).
\end{align*}
\end{rem}

\begin{proof}[Proof of Theorem \ref{thm-erg-1}]
By construction, for any $\mathbbm{k}\in \Delta$, the measure $\Psi(\mu_{\mathbbm{k}})$ is Bernoulli.  Consequently, by the De Finetti Theorem, 
\begin{align*}
\Psi(\mu_{\mathbbm{k}}) \in \mathcal{P}_{\mathrm{erg}}^{S(\infty)}(F^\N)  = Ext(\mathcal{P}_{\mathrm{inv}}^{S(\infty)}(F^\N)). 
\end{align*}
It is clear that for any convex subset $C$ of $Ext(\mathcal{P}_{\mathrm{inv}}^{S(\infty)}(F^\N))$, we have
$$
C \cap Ext(\mathcal{P}_{\mathrm{inv}}^{S(\infty)}(F^\N)) \subset Ext(C). 
$$
Taking $C =\Psi(\mathcal{P}_{\mathrm{inv}}(\Mat(\N, F))) $, we see that 
\begin{align}\label{ext}
\Psi( \mu_{\mathbbm{k}}) \in Ext(\Psi(\mathcal{P}_{\mathrm{inv}}(\Mat(\N, F)))).
\end{align}
 By Lemma \ref{lem-diag-inj1}, $\Psi$ is an affine embedding.  Hence
 \begin{align*}
 Ext(\Psi(\mathcal{P}_{\mathrm{inv}}(\Mat(\N, F)))) = \Psi(Ext(\mathcal{P}_{\mathrm{inv}}(\Mat(\N, F)))). 
 \end{align*}
  The relation \eqref{ext} implies $ \mu_{\mathbbm{k}} \in  Ext(\mathcal{P}_{\mathrm{inv}}(\Mat(\N, F))).$ By Remark \ref{rem-erg-ext}, we get the desired relation $\mu_{\mathbbm{k}}\in \mathcal{P}_{\mathrm{erg}}(\Mat(\N, F))$. 
\end{proof}

\begin{rem}
As a corollary of our classification theorem, in our situation, we indeed have 
\begin{align*}
Ext(\mathcal{P}_{\mathrm{inv}}(\Mat(\N, F))) =  \mathcal{P}_{\mathrm{erg}}(\Mat(\N, F)). 
\end{align*}
\end{rem}

\subsection{$\GL(\infty, \O_F)$-invariance for $\nu_{\mathbbm{h}}$'s}

\begin{prop}\label{prop-inv-sym}
For any  $\mathbbm{h} \in  \Omega$, the probability measure $\nu_{\mathbbm{h}}$ on $\Sym(\N, F)$  is $\GL(\infty, \O_F)$-invariant.
\end{prop}

\begin{lem}\label{lem-haar-inv}
The normalized Haar measure on $\Sym(n, \O_F)$ is invariant under the natural action of the group $\GL(n, O_F)$.
\end{lem}

\begin{proof} 
For any $g \in \GL(n, \O_F)$, the linear map
\begin{align}\label{linear-map}
\begin{split}
\begin{array}{ccc}
\Sym(n, \O_F) & \longrightarrow& \Sym(n, \O_F) 
\\
S 􏰅& \mapsto  & gSg^t
\end{array}
\end{split}
\end{align}
is invertible. Clearly, if we use the group identification  
$$
\Sym(n, \O_F) \simeq \O_F^{\frac{n^2 +n}{2}},
$$
then the linear map \eqref{linear-map} is represented by an invertible matrix from $\GL( \frac{n^2+n}{2}, \O_F)$. Hence by Remark \ref{rem-mp}, it preserves the normalized Haar measure. 􏰢
\end{proof}

\begin{proof}[Proof of Proposition \ref{prop-inv-sym}]
It suffices to prove the $\GL(\infty, \O_F)$-invariance of the following probability measures on $\Sym(\N, F)$:
\begin{align*}
\mathcal{L}( 􏰇[X^{(1)}_i X^{(1)}_j]_{i, j \in \N}) 􏰈\an \mathcal{L} ([H_{ij}]_{i, j \in\N} ).
\end{align*}
The $\GL(\infty, \O_F)$-invariance of  $\mathcal{L}( 􏰇[X^{(1)}_i X^{(1)}_j]_{i, j \in \N})$ 􏰈 follows immediately  from Lemma \ref{lem-inv-vec}, while the $\GL(\infty, \O_F)$-invariance of $\mathcal{L} ([H_{ij}]_{i, j \in\N} )$ follows from Lemma \ref{lem-haar-inv}. 􏰢
\end{proof}

\subsection{$\GL(\infty, \O_F)$-ergodicity for $\nu_{\mathbbm{h}}$'s}

\begin{thm}\label{thm-erg-2}
For any  $\mathbbm{h} \in  \Omega$, the probability measure $\nu_{\mathbbm{h}}$ on $\Sym(\N, F)$  is $\GL(\infty, \O_F)$-ergodic.
\end{thm}

The    map  $S \mapsto  (S_{ii})_{i\in\N}$ from $\Sym(\N, F)$ to $F^\N$ induces an affine map  
\begin{align*}
\Phi: \mathcal{P}(\Sym(\N, F)) \rightarrow  \mathcal{P}(F^\N) .
\end{align*}

\begin{lem}\label{lem-diag-inj2}
For any $\nu \in \mathcal{P}_{\mathrm{inv}}(\Sym(\N, F))$, we have $\Phi(\nu) \in \mathcal{P}_{\mathrm{inv}}^{S(\infty)}(F^\N)$. Moreover, the restriction map 
\begin{align*}
\Phi: \mathcal{P}_{\mathrm{inv}}(\Sym(\N, F)) \rightarrow \mathcal{P}_{\mathrm{inv}}^{S(\infty)}(F^\N)
\end{align*}
is an affine embedding. 
\end{lem}
\begin{proof}
The proof is similar to that of Lemma \ref{lem-diag-inj1}.
\end{proof}

\begin{proof}[Proof of Theorem \ref{thm-erg-2}]
The proof is similar to that of Theorem \ref{thm-erg-1} by using Lemma \ref{lem-diag-inj2} instead of Lemma \ref{lem-diag-inj1}.
\end{proof}

\section{Explicit computation of characteristic functions}\label{sec-explicit}

In this section, we give explicit formulae for characteristic functions of measures from the two families  $\{\mu_{\mathbbm{k}}:  \mathbbm{k}\in\Delta\}$ and $\{\nu_{\mathbbm{h}}:  \mathbbm{h}\in\Omega\}$.

\subsection{Measures on $\Mat(\N, F)$}

 By the elementary observation \eqref{inv-radial}, for studying the characteristic functions of $\mu_{\mathbbm{k}}$, it suffices to compute 
 \begin{align*}
 \widehat{\mu_{\mathbbm{k}}}(\diag (\varpi^{-\ell_1}, \cdots, \varpi^{-\ell_r}, 0, 0, \cdots))
  \end{align*}
   for any $r\in\N$ and any $\ell_1, \cdots, \ell_r \in\Z$.

\begin{prop}\label{prop-explicit}
For any $\mathbbm{k} = (k_n)_{n\in\N}\in \Delta$ and any $\ell \in \Z$, we have  
 \begin{align}\label{explicit-ns-1}
 \widehat{\mu_{\mathbbm{k}}}(\varpi^{-\ell} e_{11}) =   \exp\Big(- \log q  \cdot \sum\limits_{j=1}^\infty (k_j+\ell )\mathbbm{1}_{\{k_j+\ell \ge 1\}}\Big).
 \end{align}
More generally, for  any $\ell_1, \cdots, \ell_r\in\Z$, we have 
 \begin{align}\label{explicit-fourier}
 \widehat{\mu_{\mathbbm{k}}}(\diag (\varpi^{-\ell_1}, \cdots, \varpi^{-\ell_r}, 0, 0, \cdots)) = \prod_{ i=1}^r    \widehat{\mu_{\mathbbm{k}}}(\varpi^{-\ell_i} e_{11}).
 \end{align}
\end{prop}

Introduce a function  $\Theta: F \rightarrow \C$ by
\begin{align}\label{defn-Theta} 
\Theta (x) = \int_{\O_F\times \O_F} \chi(z_1 z_2 \cdot x ) dz_1 dz_2.
\end{align}

\begin{lem}\label{lem-Theta}
The function $\Theta$ is given by 
\begin{align*}
\Theta(x)= q^{-\ell \1_{\{\ell\ge 1\}}}, \quad |x| = q^{\ell}. 
\end{align*}
\end{lem}

\begin{proof}
 Let $|x| = q^{\ell}$. Then there exists $u\in \O_F^{\times}$, such that $x = \varpi^{-\ell} u$.  By rotation invariance, $\Theta(x)  = \Theta(\varpi^{-\ell})$.  
 Now by Lemma \ref{ball-fourier}, 
 \begin{align*}
 & \Theta(\varpi^{-\ell})  = \E [ \chi( \varpi^{-\ell}  X_1^{(1)} Y_1^{(1)} )] 
  = \E\Big( \E \Big[  \chi( \varpi^{-\ell}  X_1^{(1)} Y_1^{(1)} ) \Big| Y_1^{(1)} \Big]\Big) 
  \\
 & = \E(\1_{\O_F}(\varpi^{-\ell} Y_1^{(1)} ))
  = \PP(|Y_1^{(1)} | \le q^{-\ell}) =  q^{-\ell \1_{\{\ell\ge 1\}}}.
 \end{align*}
\end{proof}

\begin{rem}
In the formulae below, for graphical convenience, we write $\exp(-\log q \cdot  \ell  \1_{\{\ell \ge 1\}})$ instead of $q^{-\ell \1_{\{\ell\ge 1\}}}$.
\end{rem}

\begin{proof}[Proof of Proposition \ref{prop-explicit}]
The identity \eqref{explicit-fourier} follows from the independence between  all diagonal coefficient of $M_{\mathbbm{k}}$. So we only need to prove the identity \eqref{explicit-ns-1}. 

First assume that $\mathbbm{k} = (k_n)_{n\in\N}\in \Delta$ is such that $\lim\limits_{n\to\infty} k_n = -\infty.$
By the independence between all $X_1^{(n)}$ and $Y_1^{(n)}, n \in\N$, we have
 \begin{align*}
 &\widehat{\mu_{\mathbbm{k}}}(\varpi^{-\ell} e_{11})  
 = \E\Big[\chi\Big(   \tr \Big(    \Big [ \sum_{n=1}^\infty \varpi^{-k_n}X_i^{(n)} Y_j^{(n)}\Big]_{i, j\in\N}  \varpi^{-\ell} e_{11} \Big) \Big)\Big]
 \\
 &=   \E[\chi(     \sum_{n=1}^\infty \varpi^{-k_n- \ell}X_1^{(n)} Y_1^{(n)}  )]
 =        \prod_{n=1}^\infty \E[\chi(  \varpi^{-k_n- \ell}X_1^{(n)} Y_1^{(n)}  )]  .
 \end{align*}
 By Lemma \ref{lem-Theta}, we get 
  \begin{align*}
 \widehat{\mu_{\mathbbm{k}}}(\varpi^{-\ell} e_{11})   
 =&       \prod_{n=1}^\infty   \exp\Big(-\log q \cdot   (k_n+ \ell) \1_{\{   k_n+\ell \ge 1\}} \Big)
 \\
   =&     \exp\Big(- \log q  \cdot \sum\limits_{n=1}^\infty (k_n+\ell)\mathbbm{1}_{\{k_n+\ell\ge 1\}}\Big). 
 \end{align*}
 
Now assume that there exists $m \in\N\cup\{0\}$ and $k\in \Z$, such that 
\begin{align*}
\text{$k_1 \ge \cdots \ge k_m > k $ and  $k_n  = k$ for any $n \ge m+1$.}
\end{align*}
By previous computation and the formula \eqref{ch-conv} for the characteristic functions of  convolutions of probability measures, we only need to consider the case  when $\mathbbm{k} = (k_n)_{n\in\N} $ is such that 
\begin{align*}
\text{$k_n = k\in\Z$ for any $n\in\N$.}
\end{align*} In this case,  $\mu_\mathbbm{k}  = \mathcal{L}(  \varpi^{-k}Z)$  with $ Z$ an infinite   random matrix sampled uniformly from  $\Mat(\N, \O_F)$.  Hence by Lemma \ref{ball-fourier}, we obtain
 \begin{align*}
  & \widehat{\mu_\mathbbm{k}}(\varpi^{-\ell} e_{11})  
   = \E [ \chi( \varpi^{-\ell- k} Z_{11})]   =  \1_{ \O_F}(\varpi^{-\ell-k}) =  \1_{\{k + \ell \le 0 \}}.
   \end{align*}
   But if $k_n = k$ for any $n\in\N$, we have
   \begin{align}\label{indicator-tran}
   \1_{\{k + \ell \le 0 \}} =  \exp\Big(- \log q  \cdot \sum\limits_{n=1}^\infty (k_n+\ell)\mathbbm{1}_{\{k_n+\ell \ge 1\}}\Big). 
 \end{align}
 This proves the identity \eqref{explicit-ns-1} in the second case and we complete the proof of Proposition \ref{prop-explicit}. 
\end{proof}

\subsection{Measures on $\Sym(\N, F)$}

By the elementary observation \eqref{inv-sym-diag}, for studying the characteristic function of $\nu_{\mathbbm{h}}$, it suffices to compute 
\begin{align*}
\widehat{\nu_{\mathbbm{h}}} (\diag(x_1 ,\cdots ,x_r, 0, 0,\cdots))  
\end{align*}
for any $r \in\N$ and  any $ x_1,\cdots,x_r \in F$.

Recall the definition \eqref{defn-theta} for the function $\theta$:  
\begin{align*}
\theta (x) = \int_{\O_F} \chi(z^2 \cdot x ) dz.
\end{align*}

\begin{prop}\label{prop-ch-sym}
Let  $ \mathbbm{h} =(k; \mathbbm{k}, \mathbbm{k}') \in \Omega$. Then for any $x\in F$, we have
\begin{align}\label{ex-ch-1}
\widehat{\nu_{\mathbbm{h}}}(x e_{11})  =  \1_{\O_F}(\varpi^{-k} x) \cdot  \prod_{n=1}^\infty  \theta(  \varpi^{- k_n} x) \prod_{n=1}^\infty  \theta( \varepsilon \varpi^{- k_n'} x).
\end{align}
More generally, for any $r\in \N$ and  any $ x_1,\cdots,x_r \in F$, we have
\begin{align}\label{ex-ch-sym}
\widehat{\nu_{\mathbbm{h}}} (\diag(x_1 ,\cdots ,x_r, 0, 0,\cdots))    = \prod_{i = 1}^r\widehat{\nu_{\mathbbm{h}}}(x_i e_{11}). 
\end{align}
\end{prop}

Define a function $L_2: \O_F^{\times} \rightarrow \{-1, 1\}$ by setting $L_2(u) = 1$ if $u$ is a square element in $\O_F^{\times}$ and setting $L_2(u) = -1$ if $u$ is a non-square element in $\O_F^{\times}$.   Denote
\begin{align}\label{def-rho-q}
\begin{split}
\varrho_{q} : = 
\left\{ \begin{array}{cl}
                1, & \text{ if $q \equiv 1 (\modulo 4)$}
        \\ 
                 i, & \text{ if  $q \equiv 3 (\modulo 4)$} 
      \end{array}
 \right..
 \end{split}
\end{align}

Recall that for any $x \in F$, we have $| x| = q^{-\mathrm{ord}_F(x)}$.

\begin{prop}\label{prop-theta-detail}
The function $\theta: F \rightarrow \C$ is continuous and  satisfies the following properties:
\begin{itemize}
\item[(i)] If $ |x| \le 1$, then $\theta (x ) = 1$. 
\item[(ii)] If  $|x|  >1$ and $\mathrm{ord}_{F} (x) \equiv 0 (\modulo 2)$, then 
$\theta(x) = |x|^{-1/2}.$
\item[(iii)] If  $|x|  >1$ and $\mathrm{ord}_{F} (x) \equiv 1 (\modulo 2)$, then by writing $x  = \varpi^{-\ell}u$ with $\ell \in \N$ and $u\in \O_F^{\times}$, we have
\begin{align}\label{odd-ord}
 \theta(x)= s_{\chi} \varrho_q \cdot \frac{ L_2(u) }{|x|^{1/2}},
\end{align}
where $s_\chi \in\{-1, 1\}$ depends on the choice of $\chi$. 
\end{itemize}
In particular, $\theta$  satisfies the following property: 
\begin{align}\label{property-theta}
\theta(x)^2 = \theta(\varepsilon x)^2 \ne 0, \text{ for all $x\in F$.}
\end{align}   Moreover, if $x  = \varpi^{-\ell}u$ with $\ell \in \Z\cup\{-\infty\}$ and $u\in \O_F^{\times}$
\begin{align}\label{sup-1}
|\theta(x) |^2 &=    \exp(-  \log   |x|  \cdot \1_{|x| >1}) =  \exp(- \ell \log  q  \cdot \1_{\{\ell \ge 1\}}).
\end{align}
\end{prop}

Let us postpone the proof of  Proposition \ref{prop-theta-detail} to \S \ref{sec-theta-detail}.

\begin{lem}\label{lem-inf-prod}
Fix an element $a\in F$. Then for any $\mathbbm{k} = (k_j)_{j\in\N} \in\Delta$, the following infinite product
\begin{align}\label{inf-p-sym}
\prod_{j=1}^\infty \theta(a\cdot \varpi^{-k_j}) 
\end{align}
converges.  Moreover, if $\lim\limits_{j\to\infty} k_j = k \in \Z$,
then 
\begin{align}\label{inf-prod-red}
\prod_{j=1}^\infty \theta(a\cdot \varpi^{-k_j})  =  \1_{\O_F} (a\cdot \varpi^{-k})  \cdot \prod_{j\in \{n |k_n> k\}} \theta(a\cdot \varpi^{-k_j})  
\end{align}
\end{lem}

\begin{proof}
Let $\mathbbm{k}=(k_j)_{j\in \N}\in \Delta$. Since $\mathbbm{k}$  is a non-increasing sequence, we have 
\begin{align*}
| a\cdot \varpi^{-k_1}| \ge | a\cdot \varpi^{-k_2}| \ge \cdots \ge  | a\cdot \varpi^{-k_j}|\cdots.
\end{align*}
Then either there exists $j_0\in\N$, such that $ | a\cdot \varpi^{-k_j}|\le1$ for all  $ j \ge j_0$ or  $| a\cdot \varpi^{-k_j} | >1 $  for all $j\in\N$.   Consequently,  the infinite product \eqref{inf-p-sym} either is a finite product or equals $0$. The identity \eqref{inf-prod-red} follows immediately form \eqref{sup-1}. 
\end{proof}

\begin{proof}[Proof of Proposition \ref{prop-ch-sym}]

The identity \eqref{ex-ch-sym} follows from the independence between  all diagonal coefficient of $S_{\mathbbm{h}}$. So we only need to prove the identity \eqref{ex-ch-1}. 

{\flushleft \bf Case 1:} $\mathbbm{h} =(-\infty; \mathbbm{k}, \mathbbm{k}'), \quad \mathbbm{k} \in \Delta[-\infty], \, \mathbbm{k}' \in \Delta^{\sharp}[-\infty]$.

In this case,  we have 
\begin{align*}
 \nu_{\mathbbm{h}} = \mathcal{L}(  W_{\mathbbm{k}} +  \varepsilon W_{\mathbbm{k}'} ) = \mathcal{L}(  W_{\mathbbm{k}})*\mathcal{L}(  \varepsilon W_{\mathbbm{k}'} ).
 \end{align*}
Thus for proving \eqref{ex-ch-sym}, it suffices to prove it for the probability  measures $\mathcal{L}(  W_{\mathbbm{k}})$ and  $\mathcal{L}(  \varepsilon W_{\mathbbm{k}'} )$. For instance, we have
\begin{align*}
& \widehat{\mathcal{L}(  W_{\mathbbm{k}})} 􏰁( x e_{11}) 
= \E\Big[\chi\Big( \tr \Big(      \Big [ \sum\limits_{n=1}^\infty \varpi^{- k_n}X_i^{(n)} X_j^{(n)}\Big]_{i, j\in\N} x e_{11}    \Big) \Big)\Big]
\\
=&  \E\Big[\chi\Big(       \sum\limits_{n=1}^\infty \varpi^{- k_n}(X_1^{(n)})^2 x  \Big) \Big]
= \E\Big[    \prod_{n=1}^\infty \chi\Big(    \varpi^{- k_n}  (X_1^{(n)})^2 x  \Big) \Big].
 \end{align*}
Then  by dominated convergence theorem and the independence between all $X_1^{(n)}, n \in\N$, we get 
 \begin{align*}
 & \widehat{\mathcal{L}(  W_{\mathbbm{k}})} 􏰁( x e_{11}) 
=     \prod_{n=1}^\infty  \E\Big[  \chi\Big(   \varpi^{- k_n} (X_1^{(n)})^2 x  \Big) \Big] =      \prod_{n=1}^\infty  \theta( \varpi^{- k_n} x ). 
 \end{align*}
 Similar computation works for $\mathcal{L}(  \varepsilon W_{\mathbbm{k}'})$.

{\flushleft \bf Case 2:} There exists $k\in\Z$ and $\mathbbm{h} =(k; \mathbbm{k}, \mathbbm{k}'), \,\mathbbm{k} \in \Delta[k], \, \mathbbm{k}' \in \Delta^{\sharp}[k]$.

In this case, we have
 \begin{align*}
 \nu_{\mathbbm{h}} &= \mathcal{L}(  W_{\mathbbm{k}} +  \varepsilon W_{\mathbbm{k}'}    +\varpi^{-k} H  ) 
 =  \mathcal{L}(  W_{\mathbbm{k}})  *\mathcal{L}(  \varepsilon W_{\mathbbm{k}'} ) * \mathcal{L}(\varpi^{-k}H). 
\end{align*}
By \eqref{ch-conv},  for proving \eqref{ex-ch-sym}, it suffices to prove it for the probability  measures 
\begin{align*}
\mathcal{L}(  W_{\mathbbm{k}}), \quad \mathcal{L}(  \varepsilon W_{\mathbbm{k}'} ) \an \mathcal{L}(\varpi^{-k}H).
\end{align*}
By the computation in Case 1, we only need to verify  \eqref{ex-ch-sym} for $\mathcal{L}(\varpi^{-k}H)$. A simple computation yields the desired identity
\begin{align*}
& \widehat{\mathcal{L}(\varpi^{-k}H)} 􏰁( x e_{11}) 
=\E\Big[\chi\Big( \tr \Big( \varpi^{-k} H x e_{11}    \Big) \Big)\Big]
\\
&=\E\Big[\chi\Big(    \varpi^{-k} H_{11}    x  \Big) \Big]
=  \1_{\O_F} (\varpi^{-k} x).
 \end{align*}
\end{proof}

\begin{rem}\label{rem-id-2-pf}
Let us prove the identity \eqref{id-2} mentioned in  Remark \ref{rem-id-2}.  Denote
\begin{align*}
\sigma_1: = \mathcal{L} \Big(\Big [ \sum\limits_{n=1}^2 X_i^{(n)} X_j^{(n)}\Big]_{i, j\in\N}\Big), \quad\sigma_2 :=  \mathcal{L} \Big(  \varepsilon \Big [ \sum\limits_{n=1}^2 X_i^{(n)} X_j^{(n)}\Big]_{i, j\in\N}\Big). 
\end{align*}
Note that $\sigma_1, \sigma_2$ are both $\GL(\infty, \O_F)$-invariant. Since for any $x\in F$, we have $\theta(x)^2 = \theta(\varepsilon x)^2$. Consequently
\begin{align*}
& \widehat{\sigma_1} (\diag(x_1 ,\cdots ,x_r, 0,\cdots)) 
= \prod_{i = 1}^r \theta(x_i )^2 =  \prod_{i = 1}^r   \theta(\varepsilon x_i )^2 
\\
 &=\widehat{\sigma_2} (\diag(x_1 ,\cdots ,x_r, 0,\cdots)) .
\end{align*}
It follows that $\sigma_1 = \sigma_2$.
\end{rem}

%%%%%%%%%
%%%%%%%%%
%%%%%%%%%
%%%%%%%%%
%%%%%%%%%
%%%%%%%%%
%%%%%%%%%
%%%%%%%%%
%%%%%%%%%
%%%%%%%%%
%%%%%%%%%
%%%%%%%%%
%%%%%%%%%
%%%%%%%%%
\section{Uniqueness of parametrization}\label{sec-unique}
In this section, we will prove two uniqueness results, Proposition \ref{prop-unique} and Proposition \ref{prop-unique-sym}.

\subsection{Measures on $\Mat(\N, F)$}
\begin{prop}\label{prop-unique}
Let $\mathbbm{k}, \widetilde{\mathbbm{k}} \in \Delta$. Then   $\mu_{\mathbbm{k}} = \mu_{\widetilde{\mathbbm{k}}}$ if and only if $\mathbbm{k} = \widetilde{\mathbbm{k}}$.
\end{prop}

\begin{lem}\label{lem-inj}
The map
\begin{align*}
\mathbbm{k}  = (k_j)_{j\in\N} \mapsto \Big(\sum\limits_{j=1}^\infty (k_j+\ell)\mathbbm{1}_{\{k_j+\ell\ge 1\}}\Big)_{\ell\in\Z}
\end{align*}
from $\Delta$ to $(\Z\cup\{+\infty\})^\Z$ is injective. 
\end{lem}

\begin{proof}
We need to show that if  $\mathbbm{k} = (k_j)_{j \in \N}$ and  $\widetilde{\mathbbm{k}} = (\widetilde{k}_j)_{j\in\N}$ are two  distinct elements in $\Delta$,  then there exists $\ell \in\Z$, such that 
\begin{align}\label{an-l}
\sum\limits_{j=1}^\infty (k_j+\ell)\mathbbm{1}_{\{k_j+\ell\ge 1\}} \ne  \sum\limits_{j=1}^\infty (\widetilde{k}_j+\ell)\mathbbm{1}_{\{\widetilde{k}_j+\ell\ge 1\}}. 
\end{align}
By assumption, there exists $j_0\in\N$, such that
\begin{align}\label{small-j-0}
\text{$k_j = \widetilde{k}_j$ for any $1 \le j < j_0$ and $k_{j_0} \ne \widetilde{k}_{j_0}$.} 
\end{align}
By symmetry, let us assume that $k_{j_0} >\widetilde{ k}_{j_0}$. Under this assumption (whether $\widetilde{k}_{j_0}$ equals to $-\infty$ or not), we will have $k_{j_0}\in\Z$.  Now by taking $\ell = 1-k_{j_0} \in\Z$, we have 
\begin{align*}
\left\{\begin{array}{l}
\text{$k_j +\ell = \widetilde{k}_j  +\ell \ge 1$  for any $1 \le j < j_0$.}
\vspace{2mm}
\\
\text{$k_{j_0} + \ell =1$ and $\widetilde{k}_{j_0} + \ell \le 0$.}
\end{array}
\right.
\end{align*}
Consequently, 
$$
\sum\limits_{j=1}^\infty (k_j+\ell )\mathbbm{1}_{\{k_j+\ell \ge 1\}} \ge  \sum_{j=1}^{j_0} (k_j+\ell) =1+  \sum_{j=1}^{j_0-1} (k_j+\ell)  ,
$$
while 
$$
\sum\limits_{j=1}^\infty (\widetilde{k}_j+\ell )\mathbbm{1}_{\{\widetilde{k}_j+\ell \ge 1\}} =  \sum_{j=1}^{j_0-1} (\widetilde{k}_j+\ell)= \sum_{j=1}^{j_0-1} (k_j+\ell).
$$
Thus we prove that the  inequality \eqref{an-l} holds for $\ell = 1-k_{j_0}$.  
\end{proof}

\begin{proof}[Proof of Proposition \ref{prop-unique}]
Proposition \ref{prop-unique} follows from Proposition \ref{prop-explicit},  Lemma \ref{lem-inj} and the fact any probability measure on $\Mat(\N, F)$ is uniquely determined by its characteristic function.
\end{proof}

\subsection{Measures on $\Sym(\N, F)$}

\begin{prop}\label{prop-unique-sym}
Let $\mathbbm{h}, \widetilde{\mathbbm{h}} \in \Omega$. Then   $\nu_{\mathbbm{h}} = \nu_{\widetilde{\mathbbm{h}}}$ if and only if $\mathbbm{h} = \widetilde{\mathbbm{h}}$.
\end{prop}

\begin{rem}\label{rem-multi} 
Any element $\mathbbm{k} =(k_j)_{j\in\N} \in \Delta$ is uniquely determined by the bi-infinite sequence in $\N\cup\{\infty\}$: 
\begin{align*}
\Big(   \# \{j \in \N | k_j = \ell \}  \Big)_{\ell\in\Z}. 
\end{align*}
\end{rem}

\begin{proof}[Proof of Proposition \ref{prop-unique-sym}]
Let $\mathbbm{h} =(k; \mathbbm{k}, \mathbbm{k}')$ and $\widetilde{\mathbbm{h}} =(\widetilde{k}; \widetilde{\mathbbm{k}}, \widetilde{\mathbbm{k}}')$ be  two elements  in $\Omega$ such that  $\nu_{\mathbbm{h}} = \nu_{\widetilde{\mathbbm{h}}}.$
By Proposition \ref{prop-ch-sym}, this is {\it equivalent} to the following identity:  for any $x\in F$,  
\begin{align}\label{equ-1}
\begin{split}
 &  \1_{\O_F}(\varpi^{-k} x) \cdot  \prod_{n=1}^\infty  \theta(  \varpi^{- k_n} x) \prod_{n=1}^\infty  \theta( \varepsilon \varpi^{- k_n'} x)
 \\
 =&  \1_{\O_F}(\varpi^{-\widetilde{k}} x) \cdot  \prod_{n=1}^\infty  \theta(  \varpi^{- \widetilde{k}_n} x) \prod_{n=1}^\infty  \theta( \varepsilon \varpi^{- \widetilde{k}_n'} x). 
 \end{split}
\end{align}
Recall the identity \eqref{sup-1}. By  taking the modulus and square of both sides of \eqref{equ-1} and substituting  $x = \varpi^{-\ell} u $ with $\ell \in \Z$ and $u\in \O_F^{\times}$, we  obtain   
\begin{align}\label{equ-2}
\begin{split}
 &  \1_{\{k+\ell \le 0\}} \cdot   q^{-  \sum_{n=1}^\infty (k_n+\ell) \1_{\{k_n+\ell \ge 1\}} -  \sum_{n=1}^\infty (k_n'+\ell) \1_{\{k_n'+\ell \ge 1\}}}  
 \\
 =&   \1_{\{\widetilde{k}+\ell \le 0\}} \cdot   q^{-  \sum_{n=1}^\infty (\widetilde{k}_n+\ell) \1_{\{\widetilde{k}_n+\ell \ge 1\}} -  \sum_{n=1}^\infty (\widetilde{k}_n'+\ell) \1_{\{\widetilde{k}_n'+\ell \ge 1\}}} . 
 \end{split}
\end{align}

{\flushleft \bf Claim 1:}   $k = \widetilde{k}$. 

Indeed, if  $k = -\infty$, then the left hand side of the identity \eqref{equ-2} never vanishes. Consequently, so does the right hand side. It follows that $\widetilde{k} = -\infty$.  If $k \in\Z$.  Then the left hand side of the identity \eqref{equ-2} vanishes at $\ell = 1- k$. Consequently, so does the right hand side of  \eqref{equ-2} vanishes both at $\ell = 1- k$. It follows that $\widetilde{k} + 1 - k >0$ or equivalently $\widetilde{k}\ge k$. By symmetry, we  have $k = \widetilde{k}$.  

{\flushleft \bf Claim 2:}    $ (\mathbbm{k}, \mathbbm{k}')  = (\widetilde{\mathbbm{k}}, \widetilde{\mathbbm{k}}')$. 

For simplifying notation, let us define $\mathbbm{k}^*, \widetilde{\mathbbm{k}}^* \in \Delta$ as follows:  if $k = \widetilde{k} = -\infty$, then set $\mathbbm{k}^* : = \mathbbm{k},  \widetilde{\mathbbm{k}}^*: = \widetilde{\mathbbm{k}}$; if  $k = \widetilde{k}\in\Z$, then both $\mathbbm{k}$ and $\widetilde{\mathbbm{k}}$ are finite sequences in $\Z_{> k}$, set  $\mathbbm{k}^*$ and $\widetilde{\mathbbm{k}}^*$ by adding  infinitely many $k$. Clearly, for proving  $ (\mathbbm{k}, \mathbbm{k}')  = (\widetilde{\mathbbm{k}}, \widetilde{\mathbbm{k}}')$, it suffices to prove that $(\mathbbm{k}^*, \mathbbm{k}')  = (\widetilde{\mathbbm{k}}^*, \widetilde{\mathbbm{k}}')$. By Remark \ref{rem-multi},  it suffices to prove that for any $l \in\Z$, 
\begin{align}\label{sum-multi-1}
\begin{split}
& \# \{n \in \N | k_n^* = l \} =  \# \{n \in \N | \widetilde{k}^*_n = l \};
\\  
&\#\{n\in\N | k_n' = l \} =  \#\{n\in\N | \widetilde{k}_n' = l \}.
\end{split} 
\end{align}
Applying \eqref{indicator-tran} to    $\1_{\{k+\ell \le 0\}}$ and $\1_{\{\widetilde{k}+\ell \le 0\}}$, we may write \eqref{equ-1} as 
\begin{align}\label{q-q-powers}
\begin{split}
 &    q^{-  \sum_{n=1}^\infty (k_n^*+\ell) \1_{\{k_n^*+\ell \ge 1\}} -  \sum_{n=1}^\infty (k_n'+\ell) \1_{\{k_n'+\ell \ge 1\}}}  
 \\
 =&    q^{-  \sum_{n=1}^\infty (\widetilde{k}_n^*+\ell) \1_{\{\widetilde{k}_n^*+\ell \ge 1\}} -  \sum_{n=1}^\infty (\widetilde{k}_n'+\ell) \1_{\{\widetilde{k}_n'+\ell \ge 1\}}}. 
 \end{split}
\end{align}
By Lemma \ref{lem-inj} and Remark \ref{rem-multi}, the equality \eqref{q-q-powers} implies  that for any $l \in\Z$, we have 
\begin{align}\label{sum-multi}
\begin{split}
& \# \{n \in \N | k_n^* = l  \} + \#\{n\in\N | k_n' = l \} 
\\
=&  \# \{n \in \N | \widetilde{k}_n^* = l \} + \#\{n\in\N | \widetilde{k}_n' = l\}.
\end{split} 
\end{align}
The identity \eqref{sum-multi} implies in particular that the two identities in \eqref{sum-multi-1} hold or are violated simultaneously.  Now assume by contradiction that there exists $l_0 \in\Z$, such that the identities in \eqref{sum-multi-1} are violated. Obviously, such $l_0$ verifies 
\begin{align*}
k < l_0 \le \max\{k_n, \widetilde{k}_n, k_n', \widetilde{k}'_n \} < + \infty.
\end{align*}
Now let $l_{max} \in \Z$ be the largest  $l_0$ such that the identities in \eqref{sum-multi-1} are violated. Substituting  $x = \varpi^{l_{max} -1} u$ with $u\in \O_F^{\times}$ into the identity \eqref{equ-1},  we obtain 
\begin{align}\label{equ-1-tran}
\begin{split}
 &  \prod_{n_1=1}^\infty  \theta(  \varpi^{- k_{n_1}}   \varpi^{l_{max} -1} u ) \prod_{n_2=1}^\infty  \theta( \varepsilon \varpi^{- k_{n_2}'}   \varpi^{l_{max} -1} u )
 \\
 =&   \prod_{m_1=1}^\infty  \theta(  \varpi^{- \widetilde{k}_{m_1}}   \varpi^{l_{max} -1} u ) \prod_{m_2=1}^\infty  \theta( \varepsilon \varpi^{- \widetilde{k}_{m_2}'}   \varpi^{l_{max} -1} u ). 
 \end{split}
\end{align}
By the assumption of $l_{max}$, we know that for any $l > l_{max}$, 
\begin{align}\label{remove-cond-1}
\begin{split}
& \# \{n \in \N | k_n^* = l \} =  \# \{n \in \N | \widetilde{k}^*_n = l \};
\\  
&\#\{n\in\N | k_n' = l \} =  \#\{n\in\N | \widetilde{k}_n' = l \}.
\end{split}
\end{align}
Since $l_{max} > k$, by definition of $\mathbbm{k}^*$ and $\widetilde{\mathbbm{k}}^*$ we also have for any $l > l_{max}$: 
\begin{align}\label{remove-cond-2}
\begin{split}
 &\# \{n \in \N | k_n^* = l \}   =  \# \{n \in \N | k_n = l \};
 \\ 
 & \# \{n \in \N | \widetilde{k}^*_n = l \} =  \# \{n \in \N | \widetilde{k}_n = l \}.
 \end{split}
\end{align}
By Proposition \ref{prop-theta-detail}, the function $\theta$ never vanishes. Hence by the identities \eqref{remove-cond-1} and \eqref{remove-cond-2}, we can remove simultaneously all those terms concerning $k_{n_1} > l_{max}, k_{n_2}'> l_{max}$ and $\widetilde{k}_{m_1} > l_{max}, \widetilde{k}_{m_2}'> l_{max}$ from both sides of identity \eqref{equ-1-tran}.  Again by Proposition \ref{prop-theta-detail}, for any  $k_{n_1} <  l_{max}, k_{n_2}'< l_{max}$ and $\widetilde{k}_{m_1} < l_{max}, \widetilde{k}_{m_2}'< l_{max}$, we have  
\begin{align*}
 & \theta(  \varpi^{- k_{n_1}}   \varpi^{l_{max} -1} u )  =  \theta( \varepsilon \varpi^{- k_{n_2}'}   \varpi^{l_{max} -1} u )
\\
=&     \theta(  \varpi^{- \widetilde{k}_{m_1}}   \varpi^{l_{max} -1} u ) 
 = \theta( \varepsilon \varpi^{- \widetilde{k}_{m_2}'}   \varpi^{l_{max} -1} u )  =1.  
\end{align*}
Consequently, we may remove simultaneously all those terms concerning $k_{n_1} <  l_{max}, k_{n_2}'< l_{max}$ and $\widetilde{k}_{m_1} < l_{max}, \widetilde{k}_{m_2}'< l_{max}$ from both sides of identity \eqref{equ-1-tran} as well.  Then we arrive at the identity
\begin{align}\label{DDDD-id}
 \theta(  \varpi^{- 1} u )^D \cdot  \theta(   \varepsilon \varpi^{-1} u )^{D'} =  \theta(  \varpi^{- 1} u )^{\widetilde{D}} \cdot  \theta(   \varepsilon  \varpi^{-1} u )^{\widetilde{D'}},
\end{align}
where 
\begin{align*}
D:& = \#\{n\in\N| k_n= l_{max}\} \an D': = \#\{n\in\N| k_n'= l_{max}\};
\\
\widetilde{D}: & = \#\{n\in\N| \widetilde{k}_n= l_{max} \} \an \widetilde{D}': = \#\{n\in\N| \widetilde{k}_n'= l_{max}\}.
\end{align*}
By definitions for $\Delta[k]$ and $\Delta^{\sharp}[k]$, we must have 
\begin{align}\label{D-ranges}
D, \widetilde{D} \in \N \cup\{0\} \an D', \widetilde{D'}\in\{0, 1\}. 
\end{align}
The identity \eqref{sum-multi} now implies that 
\begin{align}\label{D-sums}
D + D' = \widetilde{D} + \widetilde{D'}.
\end{align}
By definition of $l_{max}$, we have $D \ne D'$. Without loss of generality, we may assume that $D' = 0$ and $\widetilde{D'}=1$.  Then  the identity \eqref{DDDD-id}  becomes 
\begin{align}\label{4D-bis}
 \theta(  \varpi^{- 1} u )^D =  \theta(  \varpi^{- 1} u )^{\widetilde{D}} \cdot  \theta(   \varepsilon  \varpi^{-1} u ).
\end{align}
But now $D = \widetilde{D}+1$ and since $\theta$ never vanishes, the identity \eqref{4D-bis} is equivalent to 
\begin{align*}
 \theta(  \varpi^{- 1} u ) =  \theta(   \varepsilon  \varpi^{-1} u ).
\end{align*}
 Since $|\varpi^{- 1} u |>1, \mathrm{ord}_F(\varpi^{-1} u) =-1 \equiv 1 (\modulo 2)$ and $L_2(\varepsilon)=-1$, by Proposition \ref{prop-theta-detail},  we have  $\theta(   \varepsilon \varpi^{- 1} u ) =  - \theta(    \varpi^{-1} u )$. Consequently,  we would have $ \theta(  \varpi^{- 1} u ) =  \theta(   \varepsilon  \varpi^{-1} u )=0$. This contradicts to the non-vanishing property of $\theta$. Hence we complete the proof of Claim 2. 

Combining Claim 1 and Claim 2, we complete the proof of Proposition \ref{prop-unique-sym}.
\end{proof}

%%%%%%%%%
%%%%%%%%%
%%%%%%%%%
%%%%%%%%%
%%%%%%%%%
%%%%%%%%%
%%%%%%%%%
%%%%%%%%%
%%%%%%%%%
%%%%%%%%%
%%%%%%%%%
%%%%%%%%%
%%%%%%%%%
%%%%%%%%%

\section{Ergodic measures as limits of orbital measures: the Vershik-Kerov ergodic method}\label{sec-erg}

In this section, we recall the Vershik-Kerov ergodic method for dealing with ergodic measures for inductively compact groups. The general setting is as follows. Let 
\begin{align*}
K(1) \subset K(2) \subset \cdots \subset K(n) \subset \cdots \subset K(\infty),
\end{align*}
be an increasing chain of topological groups such that for any $n\in\N$, the group $K(n)$ is compact and $K(\infty)$ is the inductive limit: 
\begin{align*}
K(\infty) =\lim_{\longrightarrow} K(n). 
\end{align*}
For any $n\in \N$,  let $m_{K(n)}$ denote the normalized Haar measure of $K(n)$.  Fix a group action of $K(\infty)$  on a Polish space $\X$.

\begin{defn}[Orbital measures]
 For any $x\in\X$ and any $n\in\N$, we define the $K(n)$-orbital measure generated by $x$,  denoted by $m_{K(n)}(x)$, as  the unique $K(n)$-invariant probability measure on $\X$ supported on the $K(n)$-orbit $K(n) \cdot x   : = \{g\cdot x| g \in K(n)\}$.
In other words,  $m_{K(n)}(x)$ is  the image of $m_{K(n)}$ under the map  $g \mapsto g \cdot x$ from  $K(n)$ to $\X$.
\end{defn}

\begin{defn}
Let $\mathscr{L}^{K(\infty)} (\X) \subset \mathcal{P}(\X)$ be the set of probability measures $\mu$ on $\X$ such that there exists $x \in \X$ verifying $
m_{K(n)}(x) \Longrightarrow \mu$.
\end{defn}

\begin{thm}[{Vershik \cite[Theorem 1]{Vershik-inf-group}}]\label{Vershik-thm}
The following inclusion holds:
\begin{align}\label{vershik-inclusion}
\mathcal{P}_{\mathrm{erg}}^{K(\infty)}(\X)\subset \mathscr{L}^{K(\infty)}(\X).
\end{align} 
More precisely, if $\mu$ is an ergodic  $K(\infty)$-invariant Borel probability measure on $\X$, then for $\mu$-almost every point $x\in\X$,  the weak convergence $m_{K(n)} (x) \Longrightarrow \mu$ holds.
\end{thm}

Vershik's method  in \cite{Vershik-inf-group}  was further developed in  a series of papers \cite{VK81, VK-char, KV86, OV-ams96} by Kerov,  Olshanski and Vershik.

\begin{rem}
In general, the converse inclusion $\mathscr{L}^{K(\infty)}(\X) \subset \mathcal{P}_{\mathrm{erg}}^{K(\infty)}(\X)$ does not hold. There is however a simple situation, see the note \cite{Q-vershik},  where $\mathcal{P}_{\mathrm{erg}}^{K(\infty)}(\X)$ always coincides with $\mathscr{L}^{K(\infty)}(\X)$.  
\end{rem}

For simplifying notation, in what follows, we denote
\begin{align*}
\mathscr{L}(\Mat(\N, F)) : &= \mathscr{L}^{\GL(\infty, \O_F) \times \GL(\infty, \O_F)}(\Mat(\N, F));
\\
\mathscr{L}(\Sym(\N, F)) : & = \mathscr{L}^{\GL(\infty, \O_F)}(\Sym(\N, F)).
\end{align*}

For any $n\in\N$, we set
\begin{align*}
\ORB_n(\Mat(n, F)):  = \{m_{\GL(n, \O_F) \times \GL(n, \O_F)}(M)| M \in \Mat(n, F)\}.
\end{align*}
By identifying $\Mat(n, F)$ in a natural way with the subset of $\Mat(\N, \FF)$, we have $\mathcal{P} (\Mat(n, F)) \subset \mathcal{P} (\Mat(\N, F))$. In particular, 
\begin{align*}
\ORB_n(\Mat(n, F)) \subset \mathcal{P} (\Mat(\N, F))
\end{align*}

\begin{defn}\label{defn-orb-limit}
Let $\ORB_\infty(\Mat(\N, F))$ denote the set of probability measures $\mu$ on  $\Mat(\N, F)$ such that there exists a subsequence  of positive integers $n_1 < n_2 < \cdots$ and  a sequence $(\mu_{n_k})_{k\in\N}$  of orbital measures with  $\mu_{n_k} \in \ORB_{n_k} (\Mat(n_k, F)) $, so that  
\begin{align*}
\mu_{n_k} \Longrightarrow \mu.
\end{align*} 
\end{defn}

Similarly, in the symmetric case, for any $n\in\N$, we set
\begin{align*}
\ORB_n(\Sym(n, F)):  = \{m_{\GL(n, \O_F)}(S)| S \in \Sym(n, F)\}.
\end{align*}
 By identifying $\Sym(n, F)$ in a natural way with a subspace of $\Sym(\N, F)$, we have 
 \begin{align*}
\ORB_n(\Sym(n, F)) \subset \mathcal{P} (\Sym(\N, F)).
\end{align*}
 \begin{defn}\label{defn-ol-sym}
 Let $\ORB_\infty(\Sym(\N, F))$ denote the set of probability measures $\nu$ on $\Sym(\N, F)$ such that there exists a subsequence  of positive integers $n_1 < n_2 < \cdots$  and  a sequence $(\nu_{n_k})_{k\in\N}$  of orbital measures with  $\nu_{n_k} \in \ORB_{n_k} (\Sym(n_k, F)) $, so that  
\begin{align*}
\nu_{n_k} \Longrightarrow \nu.
\end{align*} 
\end{defn}

\begin{rem}
It is easy to see that we have 
\begin{align*}
 \ORB_{\infty}(\Mat(\N, F)) \subset  \mathcal{P}_{\mathrm{inv}} (\Mat(\N, F));
 \\
  \ORB_{\infty}(\Sym(\N, F)) \subset  \mathcal{P}_{\mathrm{inv}} (\Sym(\N, F)). 
 \end{align*} 
\end{rem}

\begin{prop}\label{prop-limit-erg}
The following two inclusions hold:
\begin{align}\label{incl-ns}
\mathscr{L}(\Mat(\N, F)) \subset \ORB_{\infty}(\Mat(\N, F));
\end{align}
\begin{align}\label{incl-sym}
\mathscr{L}(\Sym(\N, F)) \subset \ORB_{\infty}(\Sym(\N, F)). 
\end{align} 
\end{prop}

\begin{proof}
For any $n, m \in \N$ such that $m \ge n$, let 
\begin{align*}
&\Cut_n^\infty: \Mat(\N, \FF) \rightarrow \Mat(n, \FF) 
\\ 
&\Cut_n^m: \Mat(m, \FF) \rightarrow \Mat(n, \FF)
\end{align*}
be the maps of cutting the $n \times n$ left-upper corner. 

For simplifying notation, denote 
\begin{align}\label{notation-kn}
\mathcal{K}(n) := \GL(n, \O_F) \times \GL(n, \O_F).
\end{align} 
Let $\mu \in \mathscr{L}(\Mat(\N, F))$. By definition, there exists an infinite matrix $X_0\in\Mat(\N, F)$ and a subsequence $(n_k)_{k\in\N}$  of  positive integers such that 
$$
m_{\K(n_k)}(X_0) \Longrightarrow \mu. 
$$
This is equivalent to say that for any $N \in\N$, we have
  \begin{align}\label{cut-1}
(\Cut^\infty_N)_{*} [m_{\K(n_k)}(X_0)] \Longrightarrow (\Cut^\infty_N)_{*}  (\mu).
 \end{align}
Take 
$$
X_{k}  = \Cut^\infty_{n_k}(X_0) \in \Mat(n_k, F).
$$
Then we have
$$
m_{\K(n_k)}(X_{k}) \Longrightarrow \mu. 
$$
Indeed, it suffices to prove that for any $N \in\N$, we have 
  \begin{align}\label{cut-2}
  (\Cut^\infty_N)_{*} [m_{\K(n_k)}(X_k)] \Longrightarrow (\Cut^\infty_N)_{*}  (\mu).
 \end{align}
 For any $k\in\N$, we clearly have 
  \begin{align}\label{cut-3}
 (\Cut^\infty_{n_k})_{*} [m_{\K(n_k)}(X_k)] =  (\Cut^\infty_{n_k})_{*} [m_{\K(n_k)}(X_0)]
  \end{align}
But if $n_k \ge N$, we have $ \Cut^\infty_N \circ \Cut^\infty_{n_k}= \Cut^\infty_N$. Combining with \eqref{cut-3}, we see that, once $n_k \ge N$, we have
$$
 (\Cut^\infty_{N})_{*} [m_{\K(n_k)}(X_k)] =  (\Cut^\infty_{N})_{*} [m_{\K(n_k)}(X_0)].
$$ 
Now it is clear that \eqref{cut-1} implies \eqref{cut-2}. The first inclusion \eqref{incl-ns} is proved. The proof of the second inclusion \eqref{incl-sym} is the same. 
\end{proof}

%%%%%%%%%
%%%%%%%%%
%%%%%%%%%
%%%%%%%%%
%%%%%%%%%
%%%%%%%%%
%%%%%%%%%
%%%%%%%%%
%%%%%%%%%
%%%%%%%%%
%%%%%%%%%
%%%%%%%%%
%%%%%%%%%
%%%%%%%%%

\section{Asymptotic multiplicativity for orbital integrals}\label{sec-ch}

\subsection{$\GL(n, \O_F)\times \GL(n, \O_F)$-orbital integrals}

Recall Definition \ref{defn-Theta} and Lemma \ref{lem-Theta}: 
\begin{align*}
\Theta (x) = \int_{\O_F\times \O_F} \chi(z_1 z_2 \cdot x ) dz_1 dz_2 = q^{-\ell \1_{\{\ell\ge 1\}}}, \quad |x| = q^{\ell}.
\end{align*}

In what follows, we use the notation \eqref{notation-kn}. 

\begin{thm}\label{thm-asy-mul}
Let $n, r \in \N$ be such that $r \le n $.  Suppose that $D$ and $A$ are two diagonal matrices given by: 
\begin{align*}
D = \diag (x_1, \cdots, x_n), \quad A = \diag(a_1, \cdots, a_r, 0, \cdots, 0)
\end{align*}
where $x_1, \cdots, x_n, a_1, \cdots, a_r \in F$. Then
\begin{align}\label{nsym-re}
&  \Big|   \int\limits_{\K(n)}    \chi   (\tr ( g_1D g_2  A  )) d g_1 d g_2 -  \prod_{i = 1}^r \prod_{j=1}^n \Theta (a_i x_j)   \Big| \le 2 (1  - \prod_{w = 0}^{r-1} (1 - q^{w-n}))^2, 
 \end{align}
where $dg_1dg_2$ is the normalized Haar measure on  $\K(n)$. 

In particular, for any $a\in F$, we have
\begin{align}\label{nsym-r1}
&  \Big|   \int\limits_{\K(n)}    \chi   (  a \cdot \tr ( g_1D g_2  e_{11}  )) d g_1 d g_2 -  \prod_{j=1}^n \Theta (a x_j)   \Big| \le 2 q^{-2n}.
 \end{align}
\end{thm}

\begin{rem}
Obviously, we have 
\begin{align*}
\int\limits_{\K(n)}    \chi   (\tr ( g_1D g_2  A  )) d g_1 d g_2 = \int\limits_{\K(n)}    \chi   (\tr ( g_1D g_2^{-1}  A  )) d g_1 d g_2.
\end{align*}
\end{rem}

The following elementary  lemma will be useful.   
\begin{lem}\label{lem-Srn}
Let $r \le n$ be two positive integers.  Define
  \begin{align*}
  S(r \times n) : = \Big\{  M \in \Mat(r \times n, \mathcal{C}_q) \Big| \rank_{\F_q}(M (\modulo \varpi \O_F))= r \Big\}.
  \end{align*}
 Then  we have
  \begin{align}\label{card-srn}
  \# S(r \times n) = \prod_{w = 0}^{r-1} (q^n - q^w). 
 \end{align}
For any rectangular matrix $M \in S(r\times n)$, we have 
\begin{align*}
\#\{t \in \GL(n, \mathcal{C}_q)| t_{ij} = M_{ij}, \forall 1\le i \le r, 1\le j \le n\} = \prod_{w  =0}^{n-w-1} (q^n - q^{r+w} ).
\end{align*}
In particular, the above cardinality does not depend on the choice of $M\in S(r\times n)$. 
\end{lem}

\begin{proof}[Proof of Theorem \ref{thm-asy-mul}]
Fix $n, r \in\N$ and fix the two diagonal matrices $D$ and $A$.  Let $ T = T(n), T' =  T'(n)$ be two independent copies of random matrices  sampled uniformly from the finite set $\GL(n,\mathcal{C}_q)$, and let   $V= V(n), V'  = V'(n)$ be two independent random matrices  sampled uniformly from $\Mat(n, \varpi\O_F)$ and independent of $T, T'$.   By Proposition \ref{prop-haar-rm}, we have 
\begin{align*}
\int\limits_{\K(n)}    \chi   (\tr ( g_1D g_2  A  )) d g_1 d g_2 
  = \E   [ \chi   (\tr ( (T+ V )D  (T' + V') A  ))  ]. 
\end{align*}
Since the   transposed random matrix $(T'+ V')^t$ and the original random matrix $T' + V'$ have the same distribution, we have
\begin{align*}
& \E   [ \chi   (\tr ( (T+ V )D  (T' + V') A  ))  ]  = \E   [ \chi   (\tr ( (T+ V )D  (T' + V')^t A  ))  ] 
\\
=&     \E   \Big[  \chi (\sum_{i = 1}^r \sum_{j=1}^n (T_{ij} + V_{ij} ) (T'_{ij} + V'_{ij}) a_i x_j    ) \Big]
\\
=&   \frac{1}{[\# \GL(n, \mathcal{C}_q)]^2 }  \sum_{t, t' \in \GL(n, \mathcal{C}_q)}  \E   \Big[  \chi (\sum_{i = 1}^r \sum_{j=1}^n (t_{ij} + V_{ij} )  (t'_{ij} + V'_{ij}) a_i x_j   ) \Big]. 
\end{align*}
By Lemma \ref{lem-Srn},  we obtain 
\begin{align*}
 & \E   [ \chi   (\tr ( (T+ V )D  (T' + V') A  ))  ]  
 \\
 =& \frac{1}{[\# S(r \times n)]^2}\sum_{M, M' \in  S(r \times n) }  \E   \Big[  \chi \Big(\sum_{i = 1}^r\sum_{j=1}^n (M_{ij} + V_{ij} ) (M'_{ij} + V'_{ji}) a_i x_j    \Big) \Big]. 
\end{align*}
For simplifying notation, for any pair of rectangular matrices  $M, M' \in \Mat(r \times n, \mathcal{C}_q)$, we denote
\begin{align*}
F(M, M') : =  \E   \Big[  \chi \Big(\sum_{i = 1}^r \sum_{j=1}^n (M_{ij} + V_{ij} ) (M'_{ij} + V'_{ji}) a_i x_j    \Big) \Big].
\end{align*}
Then we may write
\begin{align*}
 & \E   [ \chi   (\tr ( (T+ V )D  (T' + V') A  ))  ] 
 \\
  =&  \frac{1}{[\# S(r \times n)]^2}\underbrace{
\Big[ \sum_{M, M' \in  \Mat( r \times n, \mathcal{C}_q ) }  F(M, M') +  \mathcal{E}_1\Big]}_{\text{denoted by $I$}},
\end{align*}
where
\begin{align*}
\mathcal{E}_1 := -   \sum_{M, M' \in  \Mat(r \times n, \mathcal{C}_q) \setminus S( r \times n) }  F(M, M'). 
\end{align*} 
Note that if $M_{ij}$ and $V_{ij}$ are sampled independently and uniformly from the finite set $\mathcal{C}_q$ and from the compact additive group $\varpi \O_F$ respectively, then   $M_{ij}+V_{ij}$ is uniformly distributed on $\O_F$. It follows that 
\begin{align*}
\prod_{i = 1}^r \prod_{j=1}^n \Theta (a_i x_j) =  \frac{1}{[\# \Mat(r \times n, \mathcal{C}_q )]^2 }   \sum_{M, M' \in  \Mat(r \times n, \mathcal{C}_q ) }  F(M, M').
\end{align*}
Consequently, 
\begin{align*}
& \E   [ \chi   (\tr ( (T+ V )D  (T' + V') A  ))  ] 
  =  \frac{I}{[\# \Mat(r \times n, \mathcal{C}_q )]^2 }+\mathcal{E}_2
    \\
  & =    \prod_{i = 1}^r \prod_{j=1}^n \Theta (a_i x_j) +  \frac{ \mathcal{E}_1}{[\# \Mat(r \times n, \mathcal{C}_q )]^2 }   +\mathcal{E}_2,
 \end{align*}
where
\begin{align*}
\mathcal{E}_2 : =  \frac{I}{[\# S(r \times n)]^2}  -  \frac{I}{[\# \Mat(r \times n, \mathcal{C}_q )]^2 }.
\end{align*}

Now let us estimate these error terms.  By the obvious estimate $| F(M, M')| \le 1$, we have
\begin{align*}
| \mathcal{E}_1 | \le  [\# ( \Mat( r \times n, \mathcal{C}_q) \setminus S( r \times n))]^2.
\end{align*}
Note that $| I | \le  [\# S(r \times n) ]^2$.
Hence 
\begin{align*}
 |\mathcal{E}_2 | &=  |I| \Big( \frac{1}{[\# S(r \times n)]^2}  -  \frac{1}{[\# \Mat(r \times n, \mathcal{C}_q )]^2 }\Big) 
\\
&\le   \left(\frac{\#  \Mat(r \times n, \mathcal{C}_q)- \# S(r \times n)}{\# \Mat( r \times n, \mathcal{C}_q )}\right)^2.
\end{align*}
Taking \eqref{card-srn} into account, we get
 \begin{align*}
&  \left|   \E   [ \chi   (\tr ( (T+ V )D  (T' + V') A  ))  ] -  \prod_{i = 1}^r \prod_{j=1}^n \Theta (a_i x_j)   \right|
\\
& \le \frac{ |\mathcal{E}_1|}{[\# \Mat(r \times n, \mathcal{C}_q )]^2 }   +|\mathcal{E}_2|
\\
& \le 2 \left(\frac{\#  \Mat(r \times n, \mathcal{C}_q)- \# S(r \times n)}{\# \Mat(r \times n, \mathcal{C}_q )}\right)^2
\\
& \le 2 (1  - \prod_{w = 0}^{r-1} (1 - q^{w-n}))^2.   
 \end{align*}
\end{proof}

\begin{thm}[Uniform Asymptotic Multiplicativity]\label{thm-uam}
Let $n, r \in \N$ be such that $r \le n $.  Suppose that $D$ and $A$ are two diagonal matrices given by: 
 $$
D = \diag (x_1, \cdots, x_n), \quad A = \diag(a_1, \cdots, a_r, 0, \cdots, 0)
$$
where $x_1, \cdots, x_n, a_1, \cdots, a_r \in F$. Then
\begin{align}\label{uam-ns-formula}
\begin{split}
&  \Big|   \int\limits_{\K(n)}    \chi   (\tr ( g_1D g_2  A  )) d g_1 d g_2 -  \prod_{i = 1}^r   \int\limits_{\K(n)}    \chi   ( a_i \cdot \tr ( g_1D g_2  e_{11}  )) d g_1 d g_2   \Big|
\\
&  \le 2 (1  - \prod_{w = 0}^{r-1} (1 - q^{w-n}))^2 + 2r q^{-2n}. 
\end{split}
 \end{align}
\end{thm}

\begin{proof}
By inequalities \eqref{nsym-re} and \eqref{nsym-r1}, we have 
\begin{align*}
&    \int\limits_{\K(n)}    \chi   (\tr ( g_1D g_2  A  )) d g_1 d g_2 =  \prod_{i = 1}^r \prod_{j=1}^n \Theta (a_i x_j) +  \varepsilon_0
\\
   =&   \prod_{i = 1}^r  \Big( \underbrace{\int\limits_{\K(n)}    \chi   (a_i \cdot \tr ( g_1D g_2  e_{11}  )) d g_1 d g_2 + \varepsilon_i}_{= \prod_{j=1}^n \Theta (a_i x_j)  } \Big) +  \varepsilon_0, 
\end{align*}
with the errors $\varepsilon_0, \varepsilon_1, \cdots, \varepsilon_r$ controlled by 
\begin{align*}
|  \varepsilon_0 | \le 2 (1  - \prod_{w = 0}^{r-1} (1 - q^{w-n}))^2 \an | \varepsilon_i | \le 2 q^{-2n}, \quad i = 1, \cdots, r. 
\end{align*}
Using the elementary inequalities $| \prod_{j=1}^n \Theta (a_i x_j) | \le 1$
and by  a simple computation,  we get
\begin{align*}
&  \Big|   \int\limits_{\K(n)}    \chi   (\tr ( g_1D g_2  A  )) d g_1 d g_2 -  \prod_{i = 1}^r   \int\limits_{\K(n)}    \chi   ( a_i \cdot \tr ( g_1D g_2  e_{11}  )) d g_1 d g_2   \Big|
\\
&  \le  \varepsilon_0+  \varepsilon_1 + \cdots +  \varepsilon_r
\\
& \le  2 (1  - \prod_{w = 0}^{r-1} (1 - q^{w-n}))^2 + 2r q^{-2n}. 
 \end{align*}
\end{proof}

\subsection{$\GL(n, \O_F)$-orbital integrals}

Recall the definition \eqref{defn-theta} for the function  $\theta: F \rightarrow \C$.

\begin{thm}\label{thm-sym-asy}
Let $n, r \in \N$ be such that $r \le n$. Given two diagonal matrices $D$ and $A$:  
$$
D = \diag (x_1, \cdots, x_n), \quad A = \diag(a_1, \cdots, a_r, 0, \cdots, 0)
$$
where $x_1, \cdots, x_n, a_1, \cdots, a_r \in F$, we have
\begin{align}\label{sym-re}
\Big|   \int\limits_{\GL(n, \O_F)}    \chi   (\tr ( gD g^t  A  )) d g  - \prod_{i = 1}^r \prod_{j = 1}^n \theta (a_i x_j) \Big| \le 2\cdot (1 - \prod_{w = 0}^{r-1} (1 - q^{w-n})).
\end{align}

In particular, for any $a\in F$, we have
\begin{align}\label{sym-r1}
\Big|   \int\limits_{\GL(n, \O_F)}    \chi   ( a  \cdot \tr ( gD g^t  e_{11}  )) d g  -  \prod_{j = 1}^n \theta (a \cdot  x_j) \Big| \le 2q^{-n}.
\end{align}
\end{thm}

\begin{proof}
Fix $n, r \in\N$ and fix the two diagonal matrices $D$ and $A$.  Let $ T = T(n)$ be a random matrix uniformly distributed on the finite set $\GL(n,\F_q)$, and let   $V= V(n)$ be a random matrix  uniformly distributed on $\Mat(n, \varpi\O_F)$ and independent of $T$.   By Proposition \ref{prop-haar-rm}, we have 
\begin{align*}
\int\limits_{\GL(n, \O_F)}    \chi   (\tr ( gD g^t  A  )) d g  = \E \left [ \chi\left(\sum_{i = 1}^r \sum_{j = 1}^n   a_i x_j  (T_{ij} +  V_{ij})^2 \right)\right].
\end{align*}
By similar arguments used in the proof of Theorem \ref{thm-asy-mul},  we may get the desired  inequality \eqref{sym-re}. The second inequality \eqref{sym-r1} follows immediately by taking $e=1$ and  $a_1 = a$.
\end{proof}

\begin{thm}[Uniform Asymptotic Multiplicativity]\label{thm-sym-uam}
Let $n, r \in \N$ be such that $r \le n$. Given two diagonal matrices $D$ and $A$:  
$$
D = \diag (x_1, \cdots, x_n), \quad A = \diag(a_1, \cdots, a_r, 0, \cdots, 0)
$$
where $x_1, \cdots, x_n, a_1, \cdots, a_e \in F$, we have
\begin{align*}
& \Big|   \int\limits_{\GL(n, \O_F)}    \chi   (\tr ( gD g^t  A  )) d g  - \prod_{i = 1}^r \int\limits_{\GL(n, \O_F)}    \chi   (a_i \tr ( gD g^t   e_{11}  )) d g \Big|
\\
 & \le 2 (1  - \prod_{w = 0}^{r-1} (1 - q^{w-n})) + 2r q^{-n}. 
\end{align*}
\end{thm}

\begin{proof}
The proof is similar to that of Theorem \ref{thm-uam}.
\end{proof}

%%%%%%%%%
%%%%%%%%%
%%%%%%%%%
%%%%%%%%%
%%%%%%%%%
%%%%%%%%%
%%%%%%%%%
%%%%%%%%%
%%%%%%%%%
%%%%%%%%%
%%%%%%%%%
%%%%%%%%%
%%%%%%%%%
%%%%%%%%%

\section{The completion of  the classification of ergodic measures}\label{sec-cl-F}

\subsection{The case of $\mathcal{P}_{\mathrm{erg}}(\Mat(\N, F))$}

\begin{thm}[Multiplicativity Theorem for Orbital Limit Measures]\label{thm-multiplicativity-ns}
Let $\mu\in\ORB_\infty(\Mat(\N, F))$.  Then for any $r \in\N$ and for any finite sequence $x_1, \cdots, x_r$ in $F$, we have 
  \begin{align}\label{multiplicativity-ns}
  \widehat{\mu} (\diag (x_1, \cdots, x_r, 0, 0, \cdots)) = \prod_{j=1}^r \widehat{\mu}  (x_j e_{11}). 
  \end{align}
     In particular, we have 
  \begin{align*}
  \ORB_\infty(\Mat(\N, F)) = \mathcal{P}_{\mathrm{erg}}(\Mat(\N, F)).
  \end{align*} 
\end{thm}

\begin{proof}
Let $\mu \in \ORB_\infty(\Mat(\N, F))$.  Then by definition, there exists an increasing sequence $(n_k)_{k\in\N}$ of positive integers and a sequence $(\mu_{n_k})_{k\in\N}$ of orbital measures with  $\mu_{n_k}\in \ORB_{n_k}(\Mat(n_k, F))$, such that 
\begin{align}\label{ass-weak-cv}
\mu_{n_k} \Longrightarrow \mu \text{ as } k \to\infty.
\end{align}
Take any $x_1, \cdots, x_r \in F$. By the inequality \eqref{uam-ns-formula},  we have
\begin{align}\label{asy-mul-proof}
\lim_{k\to\infty }\Big|\widehat{\mu_{n_k}} (\diag (x_1, \cdots, x_r, 0, \cdots)) - \prod_{j = 1}^r  \widehat{\mu_{n_k}} (x_j e_{11})\Big| =0.
\end{align} 
Combining \eqref{ass-weak-cv} and \eqref{asy-mul-proof}, we get the desired identity \eqref{multiplicativity-ns}. 

By Vershik's  Theorem \ref{Vershik-thm} and Proposition \ref{prop-limit-erg}, to obtain  \eqref{asy-mul-proof}, we only need to prove the inclusion 
\begin{align}\label{proof-des-inc}
  \ORB_\infty(\Mat(\N, F)) \subset \mathcal{P}_{\mathrm{erg}}(\Mat(\N, F)).
\end{align}
Recall the definition \eqref{diag-inj1} of the affine map $\Psi$: 
\begin{align*}
\Psi: \mathcal{P}_{\mathrm{inv}}(\Mat(\N, F)) \rightarrow \mathcal{P}_{\mathrm{inv}}^{S(\infty)}(F^\N).
\end{align*}
The identity \eqref{multiplicativity-ns} implies that for any $\mu \in   \ORB_\infty(\Mat(\N, F))$, the marginal measure on the diagonal matrices $\Psi(\mu)$  is a Bernoulli measure on $F^\N$. Consequently, by exactly the same argument as in the proof of Theorem \ref{thm-erg-1}, we can prove the desired inclusion \eqref{proof-des-inc}. 
\end{proof}

An immediate consequence of Theorem \ref{thm-multiplicativity-ns} and the argument used in the proof of Theorem \ref{thm-multiplicativity-ns} is the following Ismagilov-Olshanski multiplicativity in our setting.
\begin{cor}[Ismagilov-Olshanski multiplicativity]\label{cor-IO-ns}
An invariant probability measure  $\mu\in \mathcal{P}_{\mathrm{inv}}(\Mat(\N, F))$ is ergodic if and only if  for any $r \in\N$ and for any finite sequence $x_1, \cdots, x_r$ in $F$, we have 
  \begin{align*}
  \widehat{\mu} (\diag (x_1, \cdots, x_r, 0, 0, \cdots)) = \prod_{j=1}^r \widehat{\mu}  (x_j e_{11}). 
  \end{align*}
\end{cor}

\begin{rem}
The reader may compare our method with the different methods used in, for instance, Olshansk-Vershik \cite{OV-ams96}. The Olshanski-Vershik argument relies on the Ismagilov-Olshanski multiplicativity: the multiplicativity of the characteristic function is  equivalent to the ergodicity of the corresponding probability measure. In different contexts, this multiplicativity is established  e.g.  by Ismagilov \cite{Ismagilov69, Ismagilov70}, Nessonov \cite{Nessonov}, Voiculescu \cite{Voiculescu76}, Olshanski \cite{Ol78}, Stratila-Voiculescu \cite{SV82}, Pickrell  \cite{Pickrell-jfa90}, Vershik-Kerov \cite{VK90},  Olshanski \cite{Olshansk-Howe}. 

In our situation, the Ismagilov-Olshanski multiplicativity follows as a  corollary.
\end{rem}

Now we may concentrate on the classification of $\ORB_\infty(\Mat(\N, F))$. For this purpose, we need to study the weak convergence of probability measures on $\Mat(\N, F)$.  The following standard proposition implies that the  weak convergence of probability measures on $\Mat(\N, F)$ is equivalent to the locally uniform convergence of corresponding characteristic functions.  For completeness, we include its proof in the Appendix. 

\begin{prop}\label{prop-uniform-convergence}
A sequence  of invariant probability measures $(\mu_n)_{n\in\N}$ in $\mathcal{P}_{\mathrm{inv}}(\Mat(\N, F))$ converges weakly to an invariant probability measure $\mu\in\mathcal{P}_{\mathrm{inv}}(\Mat(\N, F))$ if and only if for any $r \in\N$ and any $\ell_1, \cdots, \ell_r \in\Z$, we have  
$$
 \widehat{\mu} (\diag (\varpi^{-\ell_1}, \cdots, \varpi^{-\ell_r}, 0, \cdots))  = \lim_{n\to\infty}\widehat{\mu}_n (\diag (\varpi^{-\ell_1}, \cdots, \varpi^{-\ell_r}, 0, \cdots))
$$
and the convergence is uniform on any subset of type 
$$
\{(\ell_1, \cdots, \ell_r)| \ell_1, \cdots, \ell_r\in\Z_{\le C} \}.
$$
\end{prop}

\begin{lem}\label{lem-uniform-c}
Let $\mu$  be a Borel probability measure on $\Mat(\N, F)$. The  function $ x \in F \mapsto \widehat{\mu}(x e_{11}) \in \C$
is uniformly continuous. In particular, we have 
\begin{align*}
\lim_{\ell\to-\infty}  \widehat{ \mu} (\varpi^{-\ell} e_{11})  =1.
\end{align*}
\end{lem}

\begin{proof}
Note that the function $x \mapsto \widehat{\mu}(x e_{11})$ is the characteristic function of the marginal probability measure $(\Cut_{1}^\infty)_{*}\mu$ on $\FF$. The uniform continuity of this function then follows immediately, see, e.g., Hewitt and Ross \cite[Theorem 31.5,  p.212]{Hewitt-Ross}.  
\end{proof}

\begin{lem}\label{lem-compact}
Assume that we are given a sequence of probability measures $(\mu_n)_{n\in\N}$, such that $\mu_n\in\ORB_n(\Mat(n, F))$. A necessary and sufficient condition for  this sequence to be  tight is the following:  
\begin{itemize}
\item[($C_1$)] There exists $\gamma \in\Z$, such that the supports $\supp(\mu_n)$ are all contained in the following compact subset of $\Mat(\N, F)$:   
\begin{align*}
\left\{X  \in \Mat(\N, \FF) \Big |   \text{$ | X_{ij}| \le q^\gamma, \forall i, j \in \N$}  \right\}.
\end{align*}
\end{itemize}
\end{lem}

\begin{proof}
The above condition ($C_1$) is clearly sufficient for the sequence to be tight. Now suppose that the sequence $(\mu_n)_{n\in\N}$ is tight. By assumption, suppose that $\mu_n$ is the $\GL(n, \O_F) \times \GL(n, \O_F)$-orbital measure supported on an orbit generated by $X_n \in\Mat(n, \O_F)$. By Lemma \ref{lem-s-num}, we may assume that 
$$
X_n = \diag(x_1^{(n)}, \cdots, x_n^{(n)}), \quad | x_1^{(n)}| \ge \cdots \ge | x_n^{(n)}|.
$$
Assume by contradiction that the condition ($C_1$) is not satisfied.  Then there exists a subsequence $(n_k)_{k\in\N}$ of positive integers, such that 
\begin{align}\label{to-inf-ns}
\lim_{k\to\infty} | x_{1}^{(n_k)}| = \infty.
\end{align}
Passing to a subsequence if necessary, we may assume that there exists a probability measure $\mu$ on $\Mat(\N, F)$, such that $\mu_{n_k} \Longrightarrow \mu$. By Lemma \ref{lem-uniform-c}, for any $a \in F$, we have 
\begin{align*}
\lim_{|a|\to 0}  \lim_{k\to\infty}\widehat{\mu}_{n_k} (a e_{11}) = \lim_{|a|\to 0}  \widehat{\mu} (a e_{11}) = 1.
\end{align*}
That is,
\begin{align}\label{nec-cond-ns}
 \lim_{|a|\to 0}  \lim_{k\to\infty} \int\limits_{\K(n_k)}    \chi   ( a  \cdot \tr ( g_1 X_{n_k} g_2  e_{11}  )) d g_1dg_2 =1. 
\end{align}
By \eqref{nsym-r1}, the relation \eqref{nec-cond-ns} implies that
\begin{align}\label{nec-cond-2ns}
 \lim_{|a|\to 0}  \lim_{k\to\infty} \prod_{j = 1}^{n_k} \Theta (a \cdot  x_j^{(n_k)}) =1. 
\end{align}
Since $| \Theta(x)| \le 1$,  for any $a \in F^{\times}$, we have 
\begin{align*}
  \lim_{k\to\infty} | \prod_{j = 1}^{n_k} \Theta (a \cdot  x_j^{(n_k)}) |  \le   \lim_{k\to\infty} |  \Theta (a \cdot  x_1^{(n_k)}) | =0. 
\end{align*}
This contradicts to \eqref{nec-cond-2ns}. Thus the condition ($C_1$) is  necessary for the sequence $(\mu_n)_{n\in\N}$ to be tight. 
\end{proof}

Recall that
\begin{align*}
\Delta=\left\{ \mathbbm{k}= (k_j)_{j = 1}^\infty\Big|  k_j \in \Z \cup\{-\infty\}; k_1 \ge k_2 \ge \cdots \right\}
\end{align*}
as a subset of $(\Z \cup\{-\infty\})^\N$, is assumed to be equipped with the subspace topology of Tychonoff's product topology  on $(\Z \cup\{-\infty\})^\N$. 

\begin{lem}\label{lem-inj-con}
Let  $\ell \in \Z$. Then  
\begin{align*}
\mathbbm{k}  \mapsto   f(\mathbbm{k}) : = \sum\limits_{j=1}^\infty (k_j+\ell)\mathbbm{1}_{\{k_j+\ell\ge 1\}}
\end{align*}
defines a continuous map from $\Delta$ to  $\Z\cup\{+\infty\}$. 
\end{lem}

\begin{proof}
It suffices to prove Lemma \eqref{lem-inj-con} for $\ell = 0$. We want to prove that $\mathbbm{k} \mapsto f(\mathbbm{k})$ is continuous at some point $\mathbbm{k}^{(0)} = (k^{(0)}_j)_{j\in\N} \in \Delta$. 

{\bf \flushleft Case 1:} $f(\mathbbm{k}^{(0)})  =+ \infty$.  This means that  $k_j^{(0)} \ge 1$ for any $j \in \N$.  Consequently, for any $A \in\R$, we may take $n$ large enough so that $\sum_{j=1}^n k_j^{(0)}  > A$. Then for any $\mathbbm{k}\in\Delta$ sufficiently close to $\mathbbm{k}^{(0)}$, we have $k_j = k_j^{(0)}$ for $j = 1, \cdots, n$. For such $\mathbbm{k}$, we have  $f(\mathbbm{k}) \ge \sum_{j=1}^n k_j > A$.

{\bf \flushleft Case 2:}  $f(\mathbbm{k}^{(0)})  < + \infty$. Then choose $n$ so that 
\begin{align*}
k_1^{(0)} \ge \cdots\ge k_{n-1}^{(0)}> 0 \ge k_n^{(0)} \ge \cdots.
\end{align*}
For any $\mathbbm{k}\in\Delta$ sufficiently close to $\mathbbm{k}^{(0)}$, we have $k_j = k_j^{(0)}$ for $j = 1, \cdots, n-1$  and $k_n \le 0$. For such $\mathbbm{k}$, we have $f(\mathbbm{k}) =\sum_{j=1}^{n-1} k_j  =f(\mathbbm{k}^{(0)})$.

The proof of Lemma \ref{lem-inj-con} is completed. 
\end{proof}

\begin{thm}\label{ct-ns}
The  map $\mathbbm{k} \mapsto \mu_{\mathbbm{k}}$ induces a bijection between  $\Delta$ and $\mathcal{P}_{\mathrm{erg}}(\Mat(\N, F))$. 
\end{thm}

\begin{proof}
The injectivity of the  map $\mathbbm{k} \mapsto \mu_{\mathbbm{k}}$ from   $\Delta$ to $\mathcal{P}_{\mathrm{erg}}(\Mat(\N, F))$ has already been proved in Proposition \ref{prop-unique}. We only need to prove that the map is also surjective. 
 
Let  $\mu\in \mathcal{P}_{\mathrm{erg}}(\Mat(\N, F))$. By Theorem \ref{thm-multiplicativity-ns}, $\mu \in \ORB_\infty(\Mat(\N, F))$. Consequently, there exists a sequence $(n_l)_{l\in\N}$ of positive integers and  a sequence $(\mu_{n_l})_{l\in\N}$  such that  $\mu_{n_l}\in\ORB_{n_l}(\Mat(n_l, F))$ and 
\begin{align}\label{conv-1-ns}
\mu_{n_l}\Longrightarrow \mu \text{ as }  l \to \infty.
\end{align}
 By Lemma \ref{lem-s-num}, we may assume that $\mu_{n_l}$ is the $\K(n_l)$-orbital measure supported on the orbit $\K(n_l) \cdot X_{n_l}$ with
\begin{align*}
X_{n_l} =  \diag (\varpi^{-k^{(n_l)}_1}, \cdots, \varpi^{-k^{(n_l)}_{n_l}}), \quad k^{(n_l)}_1 \ge \cdots \ge k^{(n_l)}_{n_l} \ge -\infty.
\end{align*}
By Lemma \ref{lem-compact},  the convergence \eqref{conv-1-ns} implies that
$\sup_{l\in\N} k_1^{(n_l)}<\infty$. 
Consequently, passing to a subsequence of $(n_l)_{l\in\N}$ if necessary, we may assume that for any $j\in\N$, there exists $k_j\in \Z\cup\{-\infty\}$ such that
\begin{align}\label{coor-cv-ns}
\lim_{l\to\infty} k_j^{(n_l)}  =  k_j. 
\end{align}
The convergence \eqref{conv-1-ns} and the relation \eqref{nsym-r1} now imply that,  for any $\ell \in \Z$, we have 
\begin{align*}
& \widehat{\mu}(\varpi^{-\ell} e_{11})  =\lim_{l\to\infty}  \prod_{j=1}^\infty \Theta(\varpi^{-\ell}   \varpi^{-k_j^{(n_l)}})
\\
  &= \lim_{l \to \infty} \exp\Big(- \log q \cdot \sum_{j=1}^\infty ( k_j^{(n_l)} + \ell )\1_{\{k_j^{(n_l)} + \ell \ge 1\}} \Big)  .
\end{align*}
By  Lemma \ref{lem-inj-con} and  \eqref{coor-cv-ns}, we get 
\begin{align}\label{nu-x-ns}
\widehat{\mu}(\varpi^{-\ell} e_{11}) = \exp\Big(- \log q \cdot \sum_{j=1}^\infty ( k_j + \ell )\1_{\{k_j + \ell  \ge 1\}} \Big).
\end{align}
Let us define $\mathbbm{k}: = (k_j)_{j\in\N} \in \Delta$.  Comparing \eqref{nu-x-inter} with the formula \eqref{explicit-ns-1} in Proposition \ref{prop-explicit},  we get  
\begin{align}\label{rank1-concide-ns}
\text{$\widehat{\mu}(\varpi^{-\ell } e_{11})  = \widehat{\mu_{\mathbbm{k}}} (\varpi^{-\ell } e_{11})$ for all $\ell \in \Z$.}
\end{align} 
But by the multiplicativity of $\widehat{\mu}$ established in Theorem \eqref{thm-m-sym} and the multiplicativity of $\widehat{\mu}_{\mathbbm{k}}$ established in Proposition \ref{prop-explicit}, the above identity \eqref{rank1-concide-ns} implies $\mu= \mu_{\mathbbm{k}}$.  

The proof of Theorem \ref{ct-ns} is completed.
\end{proof}

\subsection{The case of $\mathcal{P}_{\mathrm{erg}}(\Sym(\N, F))$}

\begin{thm}[Multiplicativity Theorem for Orbital Limit Measures]\label{thm-m-sym}
Let $\nu\in\ORB_\infty(\Sym(\N, F))$.  Then for any $r \in\N$ and for any finite sequence $(x_1, \cdots, x_r)$ in $F$, we have 
  \begin{align*}
  \widehat{\nu} (\diag (x_1, \cdots, x_r, 0, 0, \cdots)) = \prod_{j=1}^r  \widehat{\nu}  (x_i e_{11}). 
  \end{align*}
  In particular, we have 
  \begin{align*}
  \ORB_\infty(\Sym(\N, F)) = \mathcal{P}_{\mathrm{erg}}(\Sym(\N, F)).
  \end{align*} 
\end{thm}

\begin{proof}
The proof is similar to that of Theorem \ref{thm-multiplicativity-ns}.
\end{proof}

\begin{cor}[Ismagilov-Olshanski multiplicativity]\label{cor-IO-sym}
An invariant probability measure  $\mu\in \mathcal{P}_{\mathrm{inv}}(\Sym(\N, F))$ is ergodic if and only if  for any $r \in\N$ and for any finite sequence $x_1, \cdots, x_r$ in $F$, we have 
  \begin{align*}
  \widehat{\nu} (\diag (x_1, \cdots, x_r, 0, 0, \cdots)) = \prod_{j=1}^r \widehat{\nu}  (x_j e_{11}). 
  \end{align*}
\end{cor}

\begin{lem}\label{lem-uc-sym}
Let $\nu$  be a Borel probability measure on $\Sym(\N, F)$. Then we have 
\begin{align*}
\lim_{|x| \to 0}  \widehat{ \nu} (x e_{11})  =1.
\end{align*}
\end{lem}

\begin{proof}
The proof is similar to that of Lemma \ref{lem-uniform-c}. 
\end{proof}

\begin{lem}\label{lem-comp-sym}
Assume that we are given a sequence of probability measures $(\nu_n)_{n\in\N}$, such that $\nu_n\in\ORB_n(\Sym(\N, F))$. The necessary and sufficient condition for  this sequence to be  tight is the following:  
\begin{itemize}
\item[($C_2$)] There exists $\gamma \in\Z$, such that the supports $\supp(\nu_n)$ are all contained in the following compact subset of $\Sym(\N, F)$:   
\begin{align*}
\left\{X  \in \Sym(\N, \FF) \Big |   \text{$ | X_{ij}| \le q^\gamma, \forall i, j \in \N$}  \right\}.
\end{align*}
\end{itemize}
\end{lem}
\begin{proof}
The above condition ($C_2$) is clearly sufficient for the sequence to be tight. Now suppose that the sequence is tight, we shall prove that ($C_2$) is satisfied. By assumption, suppose that $\nu_n$ is the $\GL(n, \O_F)$-orbital measure supported on the orbit $\GL(n, \O_F) \cdot X_n $. By Lemma \ref{lem-q-diag}, we may assume that 
$$
X_n = \diag(x_1^{(n)}, \cdots, x_n^{(n)}), \quad | x_1^{(n)}| \ge \cdots \ge | x_n^{(n)}|.
$$
Now we argue by contradiction. If the condition is not satisfied, then there exists a subsequence $(n_k)_{k\in\N}$ of positive integers such that 
\begin{align}\label{to-inf}
\lim_{k\to\infty} | x_{1}^{(n_k)}| = \infty.
\end{align}
Passing to a subsequence if necessary, we may assume that there exists a probability measure $\nu$ on $\Sym(\N, \FF)$, such that $\nu_{n_k} \Longrightarrow \nu$. By Lemma \ref{lem-uc-sym}, for any $a \in F$, we have 
\begin{align*}
\lim_{|a|\to 0}  \lim_{k\to\infty}\widehat{\nu}_{n_k} (a e_{11}) = \lim_{|a|\to 0} \widehat{\nu} (a e_{11}) =1.
\end{align*}
That is 
\begin{align}\label{nec-cond}
 \lim_{|a|\to 0}  \lim_{k\to\infty} \int\limits_{\GL(n_k, \O_F)}    \chi   ( a  \cdot \tr ( g X_{n_k} g^t  e_{11}  )) d g =1. 
\end{align}
By \eqref{sym-r1} in Theorem \ref{thm-sym-asy}, the relation \eqref{nec-cond} implies that
\begin{align}\label{nec-cond-2}
 \lim_{|a|\to 0}  \lim_{k\to\infty} \prod_{j = 1}^{n_k} \theta (a \cdot  x_j^{(n_k)}) =1. 
\end{align}
Using the elementary inequality $| \theta(x)| \le 1$  and  \eqref{to-inf} and Proposition \ref{prop-theta-detail} , for any $a \in F^{\times}$, we have 
\begin{align*}
  \lim_{k\to\infty} | \prod_{j = 1}^{n_k} \theta (a \cdot  x_j^{(n_k)}) |  \le   \lim_{k\to\infty} |  \theta (a \cdot  x_1^{(n_k)}) | =0. 
\end{align*}
This contradicts to \eqref{nec-cond-2}. Thus the condition ($C_2$) is  necessary for the sequence $(\nu_n)_{n\in\N}$ to be tight. 
\end{proof}

\subsubsection{Classification of $\mathcal{P}_{\mathrm{erg}}(\Sym(\N, F))$}

Recall that by Lemma \ref{lem-inf-prod}, for any $a\in F$ and any $\mathbbm{k} = (k_j)_{j\in\N}\in\Delta$, we may  define an infinite product 
$
\prod_{j=1}^\infty \theta(a\cdot \varpi^{-k_j}). 
$
\begin{lem}\label{lem-cont-sym}
Let  $a\in F$ be a fixed element. Then 
\begin{align*}
\mathbbm{k}  = (k_j)_{j\in\N} \mapsto  \prod\limits_{j=1}^\infty \theta(a\cdot \varpi^{-k_j})
\end{align*}
defines a continuous map from $\Delta$ to $\C$.
\end{lem}

\begin{proof}
If $a=0$, the assertion is obvious. Now assume   that $| a | = q^{\gamma}$ with $\gamma\in\Z$. Suppose that $\mathbbm{k}^{(n)} = (k_j^{(n)})_{j\in\N}$ converges to $\mathbbm{k} = (k_j)_{j\in\N}$. That is, for any $j\in
N$, 
\begin{align}\label{k-j}
\lim_{n\to\infty}k_j^{(n)} = k_j.
\end{align} 
We need to show that 
\begin{align}\label{need-cont}
& \lim_{n\to\infty}\prod_{j=1}^\infty \theta(a\cdot \varpi^{-k_j^{(n)}}) = \prod_{j=1}^{\infty} \theta(a\cdot \varpi^{-k_j}). 
\end{align}

First denote  $k:  = \lim\limits_{j\to\infty} k_j \in \Z \cup\{-\infty\}.$
{\flushleft \bf Case 1:}  $k+\gamma \le 0$. 

In this case,  there exists $j_0 \in\N$, such that $k_{j_0}+ \gamma \le 0.$ Since $\Z\cup\{-\infty\}$ is a discrete space, by \eqref{k-j}, there exists $n_0\in\N$, such that for any $n\ge n_0$, 
\begin{align}\label{}
 \sup_{j\ge j_0}( k_{j}^{(n)} + \gamma )=  k_{j_0}^{(n)} + \gamma \le 0.
\end{align}
Hence by property (i) in Proposition \ref{prop-theta-detail},  for any $j\ge j_0$ and $n\ge n_0$, we have 
$$
 \theta(a\cdot \varpi^{-k_j^{(n)}}) =\theta(a\cdot \varpi^{-k_j}) = 1. 
$$ 
Consequently, we have
\begin{align*}
& \lim_{n\to\infty}\prod_{j=1}^\infty \theta(a\cdot \varpi^{-k_j^{(n)}})=  \lim_{n\to\infty}\prod_{j=1}^{j_0-1} \theta(a\cdot \varpi^{-k_j^{(n)}})
\\
 & = \prod_{j=1}^{j_0-1} \theta(a\cdot \varpi^{-k_j}) = \prod_{j=1}^{\infty} \theta(a\cdot \varpi^{-k_j}). 
\end{align*}

{\flushleft \bf Case 2:}  $k+\gamma > 0$.  

In this case,  by the previous argument, 
we have 
\begin{align*}
\prod_{j=1}^\infty \theta(a\cdot \varpi^{-k_j}) = 0.
\end{align*} 
By \eqref{k-j}, for any $N\in\N$, there exists $n_0\in\N$, such that for any $n\ge n_0$ and any $1 \le j \le j_0+N$, we have $k_j^{(n)} = k_j$. This implies
\begin{align*}
 & \limsup_{n\to\infty}\Big|\prod_{j=1}^\infty \theta(a\cdot \varpi^{-k_j^{(n)}}) \Big| \le  \limsup_{N\to\infty} \Big(\limsup_{n\to\infty}\Big|\prod_{j=1}^{j_0+N} \theta(a\cdot \varpi^{-k_j^{(n)}}) \Big|  \Big)
 \\
 & \le \limsup_{N\to\infty}\Big|\prod_{j=1}^{j_0+N} \theta(a\cdot \varpi^{-k_j}) \Big| = \Big|\prod_{j=1}^\infty \theta(a\cdot \varpi^{-k_j})\Big|=0.
\end{align*}
Hence the desired relation \eqref{need-cont} holds.
\end{proof}

\begin{thm}\label{ct-sym}
Assume that $F$ is non-dyadic. Then the  map $\mathbbm{h} \mapsto \nu_{\mathbbm{h}}$ induces a bijection between  $\Omega$ and $\mathcal{P}_{\mathrm{erg}}(\Sym(\N, F))$. 
\end{thm}

\begin{proof}
The injectivity of the  map $\mathbbm{h} \mapsto \nu_{\mathbbm{h}}$ from   $\Omega$ to $\mathcal{P}_{\mathrm{erg}}(\Sym(\N, F))$ has already been proved in Proposition \ref{prop-unique-sym}. We only need to prove that the map is also surjective. 
 
Assume that $\nu\in \mathcal{P}_{\mathrm{erg}}(\Sym(\N, F))$. Since 
\begin{align*}
 \mathcal{P}_{\mathrm{erg}}(\Sym(\N, F)) =  \ORB_\infty(\Sym(\N, F)),
 \end{align*}
there exists  a sequence $(\nu_{n_l})_{l\in\N}$  of orbital measures  satisfying  $\nu_{n_l}\in\ORB_{n_l}(\Sym(n_l, F))$ such that 
\begin{align}\label{conv-1-sym}
\nu_{n_l} \Longrightarrow \nu \text{ as } l \to \infty.
\end{align}
 By Lemma \ref{lem-q-diag}, we may assume that $\nu_{n_l}$ is the $\GL(n_l, \O_F)$-orbital measure supported on the orbit $\GL(n_l, \O_F) \cdot X_{n_l} $ with
\begin{align*}
X_{n_l} =  \diag (x_1^{(n_l)}, \cdots, x_{n_l}^{(n_l)}), \quad x_1^{(n_l)}, \cdots, x_{n_l}^{(n_l)} \in   \mathcal{T},
\end{align*}
where $\mathcal{T}$ is given in \eqref{R-T}. 
If for any multi-set \footnote{By multi-set, we mean that  the multiplicities of elements are respected. In particular, if $B$ is a multi-set, then $B \cup  B$ will be a multi-set, such that the multiplicities for each element is multiplied by $2$.} $B$ with elements in $F$, we denote by $B^*$ the multi-set of non-zero elements of $B$. Then there exist
\begin{align*}
\mathbbm{k}^{(n_l)}  = (k_j^{(n_l)})_{j\in\N}, \, \mathbbm{k}^{'(n_l)} = (k_j^{'(n_l)})_{j\in\N} \in \Delta 
\end{align*}
such that the following two multi-sets coincide: 
\begin{align*}
\{x_1^{(n_l)}, \cdots, x_{n_l}^{(n_l)}\}^{*} =  ( \{ \varpi^{-k_j^{(n_l)}}|  j \in \N\} \cup \{  \varepsilon \varpi^{-k_j^{'(n_l)}}|  j \in \N\})^{*}.
\end{align*}
By Lemma \ref{lem-comp-sym},  the weak convergence \eqref{conv-1-sym} implies that
$$
\sup_{l \in \N} k_1^{(n_l)}<\infty \an \sup_{l\in\N} k_1^{'(n_l)}<\infty. 
$$
Consequently,  passing to a subsequence if necessary, we may assume that for any $j\in\N$, there exist $k_j, k_j'\in \Z\cup\{-\infty\}$ and we have
\begin{align}\label{coor-cv-sym}
\lim_{l\to\infty} k_j^{(n_l)}  =  k_j , \, \lim_{l\to\infty} k_j^{'(n_l)}  =  k_j'. 
\end{align}
Now by the weak convergence \eqref{conv-1-sym} and the relation \eqref{sym-r1} ,  for any $x\in F$, we have 
\begin{align*}
\widehat{\nu}(x e_{11}) & = \lim_{n\to\infty} \prod_{j=1}^\infty \theta(x    \varpi^{-k_j^{(n)}})  \prod_{j=1}^\infty \theta(x    \varepsilon \varpi^{-k_j^{'(n)}}) 
 \\
 &=\lim_{l\to\infty}  \prod_{j=1}^\infty \theta(x    \varpi^{-k_j^{(n_l)}})  \prod_{j=1}^\infty \theta(x    \varepsilon \varpi^{-k_j^{'(n_l)}}) .
\end{align*}
By the continuity of the map  in Lemma \ref{lem-cont-sym} and  \eqref{coor-cv-sym}, we get 
\begin{align}\label{nu-x}
\widehat{\nu}(x e_{11}) = \prod_{j=1}^\infty \theta(x    \varpi^{-k_j})  \prod_{j=1}^\infty \theta(x    \varepsilon \varpi^{-k_j'}). 
\end{align}
Now by using the identity $\theta(x    \varpi^{-k_j})^2  = \theta(x    \varepsilon \varpi^{-k_j'})^2$ for any $x\in F$ and $j\in\N$ and by moving certain elements ( in $\Z$ and with multiplicities larger than $1$) from $(k_j')_{j\in\N}$ to the sequence $(k_j)_{j\in\N}$, we get a new non-increasing sequence $(\widetilde{k}_j)_{j\in\N}$ and a new strictly decreasing sequence $(\widetilde{k}_j')_{j\in\N}$ of finite or infinite lengths in $\Z\cup\{-\infty\}$ such that the identity \eqref{nu-x} is transformed to 
\begin{align}\label{nu-x-inter}
\widehat{\nu}(x e_{11}) = \prod_{j=1}^\infty \theta(x    \varpi^{-\widetilde{k}_j})  \prod_{j=1}^\infty \theta(x    \varepsilon \varpi^{-\widetilde{k}'_j}). 
\end{align}

Assume first that $\lim\limits_{j\to\infty} \widetilde{k}_j = -\infty$. Then  
\begin{align*}
\mathbbm{h} = (-\infty; (\widetilde{k}_j)_{j\in\N}, (\widetilde{k}_j')_{j\in\N}) 
\end{align*}
is an element in $\Omega$. Comparing \eqref{nu-x-inter} with the formula in Proposition \ref{prop-ch-sym},  we get  
\begin{align*}
\text{$\widehat{\nu}(x e_{11})  = \widehat{\nu_{\mathbbm{h}}} (x e_{11})$ for all $x\in F$.}
\end{align*} 
But by the multiplicativity of $\widehat{\nu}$ established in Theorem \eqref{thm-m-sym} and the multiplicativity of $\widehat{\nu}_{\mathbbm{h}}$ established in Proposition \ref{prop-ch-sym}, the above identity implies $\nu= \nu_{\mathbbm{h}}$. 

Assume now that $\lim_{j\to\infty} \widetilde{k}_j = k\in\Z$.  Then using the relation \eqref{inf-prod-red}, we have 
\begin{align}\label{too-many-k}
\prod_{j=1}^\infty \theta(x    \varpi^{-\widetilde{k}_j}) =  \1_{\O_F} (x   \varpi^{-k})  \prod_{j\in\{n |\widetilde{k}_n > k\} } \theta(x    \varpi^{-\widetilde{k}_j}).
\end{align}
Moreover, we also have
\begin{align}\label{small-thing}
 \1_{\O_F} (x   \varpi^{-k})  \prod_{j=1}^\infty \theta(x    \varepsilon \varpi^{-\widetilde{k}'_j}) =  \1_{\O_F} (x   \varpi^{-k})  \prod_{j\in\{n |\widetilde{k}_n' > k\} }  \theta(x    \varepsilon \varpi^{-\widetilde{k}'_j}).
 \end{align}
It suffices to check for $x\in F$ such that $|x\varpi^{-k}|\le 1$. For  $\widetilde{k}_j'$ such that $\widetilde{k}_j' \le k$ (if they exist), we have $|x    \varepsilon \varpi^{-\widetilde{k}'_j}| =|x \varpi^{-k}   \varepsilon \varpi^{k-\widetilde{k}'_j}| \le 1 $. Hence $\theta(x    \varepsilon \varpi^{-\widetilde{k}'_j}) = 1$ by the  property (i) in  Proposition \ref{prop-theta-detail}. The identity \eqref{small-thing} is proved. Denote 
\begin{align*}
\widehat{k}_j: = \left\{\begin{array}{cc} \widetilde{k}_j, & \text{if $\widetilde{k}_j > k$}\vspace{2mm}\\
-\infty & \text{if $\widetilde{k}_j = k$}\end{array}\right. \an \widehat{k}'_j: = \left\{\begin{array}{cc} \widetilde{k}'_j, & \text{if $\widetilde{k}'_j > k$}\vspace{2mm}\\
-\infty & \text{if $\widetilde{k}'_j \le k$}\end{array}\right..
\end{align*}
Now, $\mathbbm{h} = (k; (\widehat{k}_j)_{j\in\N}, (\widehat{k}'_j)_{j\in\N})$, being an element in $\{k\}\times \Delta[k]\times \Delta^{\sharp}[k]$, is an element of $\Omega$.   Combining \eqref{nu-x-inter}, \eqref{too-many-k} and \eqref{small-thing}, we obtain
\begin{align*}
&\widehat{\nu}(x e_{11}) = \prod_{j=1}^\infty \theta(x    \varpi^{-\widetilde{k}_j})  \prod_{j=1}^\infty \theta(x    \varepsilon \varpi^{-\widetilde{k}'_j})
\\
& =   \1_{\O_F} (x   \varpi^{-k}) \Big[ \prod_{j\in\{n |\widetilde{k}_n > k\} } \theta(x    \varpi^{-\widetilde{k}_j}) \Big]\Big[  \prod_{j\in\{n |\widetilde{k}'_n > k\} }  \theta(x    \varepsilon \varpi^{-\widetilde{k}'_j})\Big]. 
\end{align*}
In this case, we also have 
\begin{align*}
\widehat{\nu}(x e_{11})  = \widehat{\nu_{\mathbbm{h}}} (x e_{11}), \quad \text{ for all $x\in F$.}
\end{align*}
By the same argument as above in using the multiplicativities of $\widehat{\nu}$ and  $\widehat{\nu}_{\mathbbm{h}}$, we get $\nu= \nu_{\mathbbm{h}}$. 

The proof of Theorem \ref{ct-sym} is completed. 
 \end{proof}

%%%%%%%%%
%%%%%%%%%
%%%%%%%%%
%%%%%%%%%
%%%%%%%%%
%%%%%%%%%
%%%%%%%%%
%%%%%%%%%
%%%%%%%%%
%%%%%%%%%
%%%%%%%%%
%%%%%%%%%
%%%%%%%%%
%%%%%%%%%

\section{Properties of the parametrization}\label{sec-cont-p}
\subsection{The parametrizations are homeomorphisms}
\begin{proof}[Proof of Theorem \ref{CT-ns}]
By Theorem \ref{ct-ns}, we only need to prove that the map $\mathbbm{k} \mapsto \mu_{\mathbbm{k}}$ from $\Delta$ to $\mathcal{P}_{\mathrm{erg}} (\Mat(\N, F))$ and its inverse are both continuous.  Note that since $\Delta$ and $\mathcal{P}_{\mathrm{erg}}(\Mat(\N, F))$ are metrizable, their topologies are determined by convergence of sequences. 

If a sequence $(\mathbbm{k}^{(n)})_{n\in\N}$ converges in $\Delta$ to a point $\mathbbm{k}^{(0)} \in\Delta$, then 
\begin{align}\label{finite-first}
\sup_{n, j } k_j^{(n)} <\infty.
\end{align}
Consequently, the family of the measures $\mu_{\mathbbm{k}^{(n)}}$, all being supported on a common compact subset of $\Mat(\N, \FF)$, is  tight. Thus to prove  that $\mu_{\mathbbm{k}^{(n)}}$ converges weakly to $\mu_{\mathbbm{k}^{(0)}}$, it suffices to prove that the latter one is the unique accumulation point of the former family of measures. Now let $\mu$ be an accumulation point of the sequence $(\mu_{\mathbbm{k}^{(n)}})_{n\in \N}$. By definition, there exists a subsequence $(n_j)_{j \in\N}$ such that  
\begin{align*}
\mu = \lim_{j \to\infty} \mu_{\mathbbm{k}^{(n_j)}}.
\end{align*} 
Since $\mathbbm{k}^{(n)} \longrightarrow \mathbbm{k}^{(0)}$, by explicit formula \eqref{explicit-ns-1} in Proposition \ref{prop-explicit} and Lemma \ref{lem-inj-con}, the charateristic function of $\mu$ is given  by 
\begin{align*}
&\widehat{\mu} (\diag (\varpi^{-\ell_1}, \cdots, \varpi^{-\ell_r}, 0, 0, \cdots))
\\
=& \lim_{j\to\infty} \widehat{\mu}_{\mathbbm{k}^{(n_j)}} (\diag (\varpi^{-\ell_1}, \cdots, \varpi^{-\ell_r}, 0, 0, \cdots)) 
\\
  =  & \widehat{\mu}_{\mathbbm{k}^{(0)}} (\diag (\varpi^{-\ell_1}, \cdots, \varpi^{-\ell_r}, 0, 0, \cdots)), \quad (\ell_j \in\Z).
\end{align*}
This implies that we  have $\mu = \mu_{\mathbbm{k}^{(0)}}$. 

Conversely, if $ \mu_{\mathbbm{k}^{(n)}}$ converges to $\mu_{\mathbbm{k}^{(0)}}$. By using the same argument in the proof of Lemma \ref{lem-compact}, we can still get the relation \eqref{finite-first}.  Again by compactness argument, it suffices to show that $\mathbbm{k}^{(0)}$ is the unique accumulation point for the sequence $\mathbbm{k}^{(n)}$.  But if $\mathbbm{k}$ is an accumulation point of $\mathbbm{k}^{(n)}$, then $\mu_{\mathbbm{k}}$ is an accumulation point of $\mu_{\mathbbm{k}^{(n)}}$,  whence  $\mu_{\mathbbm{k}} = \mu_{\mathbbm{k}^{(0)}}$. Combining with Proposition \ref{prop-unique} we have $\mathbbm{k}= \mathbbm{k}^{(0)}$. 

The proof of Theorem  \ref{CT-ns} is completed. 
\end{proof}

\begin{proof}[Proof of Theorem \ref{CT-sym}]
By Theorem \ref{ct-sym}, we only need to prove that the map $\mathbbm{h} \mapsto \nu_{\mathbbm{h}}$ from $\Omega$ to $\mathcal{P}_{\mathrm{erg}} (\Sym(\N, F))$ and its inverse are both continuous. The proof of this part is similar to that of Theorem \ref{CT-ns} as above.
\end{proof}

\subsection{The parametrizations are semi-group homomorphisms}

On $\Delta$ is equipped with an Abelian semi-group structure. Given any two points $\mathbbm{k} = (k_n)_{n\in\N}$ and $\mathbbm{k}' = (k_n')_{n\in\N}$ in $\Delta$, we define  $\mathbbm{k} \oplus_\Delta \mathbbm{k}'$ as follows: 

(i) If $ \inf  k_n = \inf k_n' = -\infty$, then we define $\mathbbm{k} \oplus_\Delta \mathbbm{k}' \in \Delta$ to be the non-increasing rearrangement of the sequence $(\widetilde{k_n})_{n\in\N}$, where
$$
\widetilde{k_{2n-1}} :  =k_n \an  \widetilde{k_{2n}} :  =k_n' \quad (n = 1, 2, \cdots). 
$$

(ii) If $k = \max\{ \inf k_n , \inf k_n' \}\in\Z$, then we  define  $\mathbbm{k} \oplus \mathbbm{k}' \in \Delta$ to be the non-increasing rearrangement of the sequence
$(k_n^*)_{n\in\N}$, that is any sequence exhausting the integers larger than $k$ and from $\mathbbm{k}$ and $\mathbbm{k}'$,  repeated with corresponding multiplicity. For instance, if $\mathbbm{k}= ( 6,  2, 2, -3, -3, -3, -3, \cdots), \mathbbm{k}' =(4, 3, 0, -1, - \infty, 
\cdots) $, then we define 
$$
\mathbbm{k} \oplus_\Delta \mathbbm{k}' = (6, 4, 3, 2, 2, 0, -1, -3, -3, -3, -3, \cdots).
$$

Clearly, we have

\begin{prop}
The map $\mathbbm{k} \mapsto \mu_\mathbbm{k}$ defines a semi-group isomorphism between $(\Delta, \oplus_\Delta)$ and $(\mathcal{P}_{\mathrm{erg}}(\Mat(\N, F)), *)$. More precisely, we have 
\begin{align*}
\mu_{\mathbbm{k}} * \mu_{\mathbbm{k}'}  = \mu_{\mathbbm{k} \oplus_\Delta \mathbbm{k}'}. 
\end{align*}
\end{prop}

An Abelian semigroup structure  $\oplus_\Omega$  on $\Omega$ such that $\mathbbm{h}\to \nu_{\mathbbm{h}}$ defines semi-groups isomorphism between $(\Omega, \oplus_\Omega)$ and $(\mathcal{P}_{\mathrm{erg}}(\Sym(\N, F)), *)$ is introduced in the same way.

%%%%%%%%%
%%%%%%%%%
%%%%%%%%%
%%%%%%%%%
%%%%%%%%%
%%%%%%%%%
%%%%%%%%%
%%%%%%%%%
%%%%%%%%%
%%%%%%%%%
%%%%%%%%%
%%%%%%%%%
%%%%%%%%%
%%%%%%%%%
%%%%%%%%%

\section{Proof of Proposition \ref{prop-theta-detail}}\label{sec-theta-detail}

In this section,  we always assume that $F$ is non-dyadic.  We will use the following change of variables  in the integration over a local field.  To introduce the formula for change of variables, we need the notion of $F$-analytic functions.  A function $\varphi: U\rightarrow V$, with $U, V$ open subsets of $F$, is called  $F$-analytic, if in some neighbourhood of any point in $U$ it is given by a convergent power series, it is called $F$-bi-analytic, if $\varphi$ is invertible such that  both $\varphi: U\rightarrow V$  and $\varphi^{-1}: V \rightarrow U$ are $F$-analytic. 

\begin{thm}[Change of variables, see Schoissengeier \cite{CV-local}]\label{thm-cv}
Let $\varphi: U\rightarrow V$ be a $F$-bi-analytic function. Then for any  integrable function $f: U\rightarrow \C$, we have 
\begin{align*}
\int_U f(\varphi(x)) | \varphi'(x)| dx = \int_V f(y) dy,
\end{align*}
where  $\varphi'$ is the formal derivative of $\varphi$. 
\end{thm}

We will also need the following classical result from number theory concerning Gauss sums for finite field $\F_q$. For the reader's convenience, we include its standard proof in Appendix.

 Denote by $\lambda_2$ the unique multiplicative character for $\F_q^{\times}$ of order $2$, that is, 
\begin{align*}
\lambda_2(a)= \left\{\begin{array}{cl} 1, & \text{ if $a \in (\F_q^{\times})^2$} \vspace{2mm}\\ -1, & \text{ if $a \notin (\F_q^{\times})^2$} \end{array}\right. . 
\end{align*}
By convention, we extend the definition of $\lambda_2$ to the whole finite field $\F_q$ by setting $\lambda_2(0) =0$. 

Denote the set of additive characters of $\F_q$ by $\widehat{\F}_q$.  Given any $\tau \in \widehat{\F}_q \setminus \{1\}$ (that is, $\tau$ is  non-trivial character  of $\F_q$) and any $a\in \F_q$, denote by $\tau_a$ the character of $\F_q$ defined by $\tau_a (x ) = \tau(a\cdot x)$. It is a standard fact that the map $a \mapsto \tau_a$ is an group isomorphism between $\F_q$ and $\widehat{\F}_q$.

\begin{lem}[Gauss sums]\label{gauss-sum}
Fix an element $\tau \in \widehat{\F}_q \setminus \{1\}$. Then for any $a\in\F_q^{\times}$, we have
\begin{align}\label{c-sign}
\sum_{x \in \F_q} \tau_a(x^2) = \lambda_2(a)  \cdot \sum_{x \in \F_q} \tau(x^2).
\end{align}
Moreover,
\begin{align*}
\Big\{\sum_{x \in \F_q} \tau_a( x^2) |  a  \in \F_q^{\times}\Big\} =  \{\varrho_q \sqrt{q}, - \varrho_q\sqrt{q}\}. 
\end{align*}
\end{lem}

Denote by $\O_F^{*}$ the set of non-zero elements in $\O_F$ and denote by $(\O_F^{*})^2$ the square elements in $\O_F^{*}$, that is, 
$$
(\O_F^{*})^2: = \{ a \in \O_F^*:  \text{there exists $b \in F$ such that $a = b^2$} \}.
$$

Denote  the square function $x \mapsto x^2$ by $\psi(x) = x^2$. 
\begin{prop}\label{prop-partition}
There exists a partition 
$$
\O_F^* = U_1 \sqcup U_2,
$$
such that the square function $\psi(x) = x^2$ induces two $F$-bi-analytic functions: 
$$
\psi : U_i \rightarrow (\O_F^{*})^2, \quad i = 1, 2. 
$$ 
\end{prop}

Recall that by Lemma \ref{lem-clopen}, the group $(\O_F^{\times})^2$ is a disjoint union of $\frac{q-1}{2}$ balls of radius $q^{-1}$: 
\begin{align*}
(\O_F^{\times})^2 =  \bigsqcup_{a \in (\mathcal{C}_q^{\times})^2} (a + \varpi \O_F).
\end{align*}

\begin{lem}\label{lem-divide}
Any element $a \in (\O_F^{\times})^2$ has two square roots  $\alpha_1,   \alpha_2 \in \O_F^{\times}$ such that $(\alpha_1+ \varpi \O_F) \cap ( \alpha_2+ \varpi \O_F) = \emptyset$ and we have 
\begin{align}\label{inclusion-odd}
\psi(\alpha_i + \varpi \O_F ) \subset a + \varpi \O_F, \quad i= 1, 2.
\end{align}
Moreover,  we have two bijective maps : 
\begin{align}\label{odd}
\alpha_i + \varpi \O_F \xrightarrow{ x \mapsto \psi(x) = x^2}  a + \varpi \O_F, \quad i= 1, 2.
\end{align}
\end{lem}

\begin{proof}
For any $a \in (\O_F^{\times})^2$,  there exist exactly two elements $\alpha_1,   \alpha_2 \in \O_F^{\times}$, such that 
$$
\alpha_1 = -\alpha_2 \an \alpha_1^2 = \alpha_2^2 = a. 
$$
Hence $| \alpha_1  - \alpha_2| = | 2\alpha| = | \alpha| = 1$ and 
\begin{align*}
(\alpha_1+ \varpi \O_F) \cap ( \alpha_2+ \varpi \O_F) = \emptyset.   
\end{align*}
For any $z \in \O_F$ and $i = 1, 2$, we have 
\begin{align*}
(\alpha_i + \varpi z)^2  = a + \varpi (2 z \alpha_i + \varpi z^2 ) \in    a + \varpi \O_F, 
\end{align*}
this proves \eqref{inclusion-odd}. 

By  Hensel's lemma \eqref{root}, for $i = 1, 2$,  if $a' \in a+ \varpi \O_F$, then there exists $\alpha_i' \in \alpha_i + \varpi \O_F$ such that $(\alpha_i')^2 = a'$. Hence the maps \eqref{odd} for $i =1, 2$ are both surjective. Now fix $i\in \{1, 2\}$.  If $\delta_1, \delta_2 \in \alpha_i + \varpi \O_F$ are such that $\delta_1^2 = \delta_2^2$, then either $\delta_1 = \delta_2$ or $\delta_1 = - \delta_2$. However, if $\delta_1  = - \delta_2$, then  
$$
| \delta_1 - \delta_2| = | 2 \delta_1| = | 2| | \delta_1| =1. 
$$
This contradicts to the following estimate 
$$
|\delta_1- \delta_2| = | (\delta_1- \alpha_i)- (\delta_2-\alpha_i)| \le \max ( | \delta_1- \alpha_i|, | \delta_2-\alpha_i|) \le q^{-1}. 
$$ 
 Hence we must have $\beta_1 = \beta_2$. This proves the injectivity of the map \eqref{odd}.
\end{proof}

\begin{proof}[Proof of Proposition \ref{prop-partition}]
Clearly, we have 
$$
\O_F^{*} = \bigsqcup_{k=0}^\infty  \varpi^{k}   \O_F^{\times} \an (\O_F^{*})^2 = \bigsqcup_{k=0}^\infty  \varpi^{2k}   (\O_F^{\times})^2. 
$$
For any $k = 0, 1, \cdots$,  the square map $\varphi$ maps $\varpi^{k}   \O_F^{\times}$ surjectively into $\varpi^{2k}   (\O_F^{\times})^2$. For proving Proposition \ref{prop-partition}, it suffices to prove that for any $k = 0, 1, \cdots$, the set $\varpi^{k}   \O_F^{\times}$ can be divided into two parts, such that the square map $\psi$ maps each part surjectively into $\varpi^{2k}   (\O_F^{\times})^2$ and the restriction of $\psi$ on each part is $F$-bi-analytic.  

We only need to prove this assertion for $k=0$, since  the other $k\ge 1$ can be reduced to the case $k=0$ by a suitable dilation.  By Lemma \ref{lem-divide}, $\O_F^{\times}$ can be divided into two parts $\O_F^{\times} = V_1 \sqcup V_2$ , such that the two maps $\psi: V_i \rightarrow (\O_F^{\times})^2, i =1, 2$ are both bijective.  The analyticity  of the inverse maps $(\psi|_ {V_i})^{-1} $  follows from the Inverse  Mapping Theorem in non-Archimedean setting, see, e.g, Abhyankar \cite[p. 87]{local-anal}.  
\end{proof}

\begin{cor}
 For any  integrable function $f: \O_F \rightarrow \C$, we have 
\begin{align}\label{cv-2}
\int_{\O_F} f(z^2) |z| dz =  2 \int_{ (\O_F^{*})^2} f(y) dy. 
\end{align}
\end{cor}
\begin{proof}
Note that since $F$ is non-dyadic, $2 \in \O_F^{\times}$.  Using the notation in Proposition \ref{prop-partition} and Theorem \ref{thm-cv}
\begin{align*}
& \int_{\O_F} f(z^2) |z| dz  =  \int_{\O_F^*} f(\psi(z)) |\psi'(z)| dz
\\ & =  \sum_{i=1}^2 \int_{U_i}  f(\psi(z)) |\psi'(z)| dz =  2 \int_{ (\O_F^{*})^2} f(y) dy. 
\end{align*}
\end{proof}

\begin{proof}[Proof of Proposition \ref{prop-theta-detail}]
The property (i) in Proposition \ref{prop-theta-detail} is trivial. We proceed with the proof of properties (ii) and (iii). 

For any $x\in F^*$, define $$f (z) = \frac{1}{|z |^{1/2}}\chi(x z).$$ Then by substituting $f$ into the identity \eqref{cv-2},  we get 
\begin{align}\label{RL}
\begin{split}
\theta(x) &  = \int_{\O_F}    \chi(x z^2) dz = \int_{\O_F} f(z^2) |2 z| dz
\\
&  = 2 \int_{ (\O_F^{*})^2} f(y) dy = 2 \int_{ (\O_F^{*})^2} \frac{1}{ |y|^{1/2}}\chi(x y) dy.
\end{split}
\end{align}
Define  
\begin{align}\label{def-g}
g(y) := \frac{2}{|y|^{1/2}} \1_{(\O_F^{*})^2} (y).
\end{align}
It is a standard fact that $g \in L^1(F, dx)$. The identity \eqref{RL} can  now be rewritten as  \begin{align*}
\theta(x) = \widehat{g}(x). 
\end{align*}

Since we have 
\begin{align*}
(\O_F^{*})^2 = \bigsqcup_{k=0}^\infty  \varpi^{2k}  (\O_F^{\times})^2 = \bigsqcup_{k=0}^\infty  \bigsqcup_{a \in (\mathcal{C}_q^{\times})^2}(\varpi^{2k} a + \varpi^{2k +1} \O_F), 
\end{align*}
the function $g$ defined by formula \eqref{def-g} can be written in the form 
\begin{align*}
g(y) =  2  \sum_{k=0}^\infty  q^{k}\sum_{a\in (\mathcal{C}_q^{\times} )^2} \1_{\varpi^{2k} a + \varpi^{2k +1} \O_F} (y).
\end{align*}
Consequently, we have 
\begin{align}\label{expan-theta}
\theta(x) = \widehat{g}(x) =   2 \sum_{k =0}^{\infty}  q^{-k-1}   \1_{ \varpi^{-2k -1} \O_F} (x)  \sum_{ a\in (\mathcal{C}_q^{\times} )^2} \chi(\varpi^{2k} a x). 
\end{align}

{\flushleft  \bf The property (ii) in Proposition \ref{prop-theta-detail}}. If $\mathrm{ord}_F(x)  = -2k_0$ with $k_0\ge 1$, then $x = \varpi^{-2k_0} u$ with $u \in \O_F^{\times}$. Substituting $x= \varpi^{-2k_0} u$ into \eqref{expan-theta} and using the assumption \eqref{assumption-chi} on $\chi$, we obtain

\begin{align*}
\theta(x)& = 2 \sum_{k =k_0}^\infty q^{-k -1}  \sum_{ a\in (\mathcal{C}_q^{\times} )^2} \chi(a \varpi^{2k-2k_0}  u ) 
\\
&= 2 \sum_{k =k_0}^\infty q^{-k -1} \cdot \frac{q-1}{2}= q^{-k_0} = |x|^{-1/2}.
\end{align*}
We thus complete the proof of the property (ii)  in Proposition \ref{prop-theta-detail}.

{\flushleft  \bf The property (iii) in Proposition \ref{prop-theta-detail}}. If $\mathrm{ord}_F(x)  = -2k_0-1$ with $k_0 \ge 0$, then $x = \varpi^{-2k_0-1} u$ with $u \in \O_F^{\times}$. By substituting $x= \varpi^{-2k_0-1} u$ into \eqref{expan-theta} and using the fact  that $a \varpi^{2k-2k_0-1} u \in \O_F$ for any $a \in (\mathcal{C}_q^{\times})^2, k \ge k_0+1$ and the assumption \eqref{assumption-chi} on the choice of $\chi$, we obtain
\begin{align}\label{com-theta}
\begin{split}
\theta(x) & = 2 \sum_{k =k_0}^\infty q^{-k -1}  \sum_{ a\in (\mathcal{C}_q^{\times} )^2} \chi(a \varpi^{2k-2k_0-1}  u )
\\
 & = 2  q^{-k_0 -1}  \sum_{ a\in (\mathcal{C}_q^{\times} )^2} \chi(a \varpi^{-1} u ) + 2 \sum_{k =k_0+1}^\infty q^{-k -1} \cdot \frac{q-1}{2}
 \\
 &=  q^{-k_0 -1} \Big(2  \sum_{ a\in (\mathcal{C}_q^{\times} )^2} \chi(a \varpi^{-1} u ) + 1\Big). 
 \end{split}
\end{align}
Now define a function  $z \mapsto \chi(z \varpi^{-1} u )$ for $z\in \O_F$. Since this function takes same value on every coset of $\varpi\O_F$ in $\O_F$,  we  define a non-trivial additive character $\tau$ on $\F_q = \O_F/\varpi\O_F$ by the following formula:
\begin{align*}
\tau( z + \varpi \O_F) : = \chi(z \varpi^{-1} u ). 
\end{align*}
Thus we have 
\begin{align}\label{simple-sum}
\begin{split}
 \theta(x) & = q^{-k_0-1} \Big( 2  \sum_{ a\in (\mathcal{C}_q^{\times} )^2}  \tau (a + \varpi\O_F) + 1  \Big)  
 \\
 & =  q^{-k_0-1} \sum_{z  + \varpi \O_F \in \F_q} \tau (  z^2+ \varpi \O_F).
\end{split}
\end{align}
Indeed, the second equality in \eqref{simple-sum} follows from the fact that for every element in $ a\in (\mathcal{C}_q^{\times})^2$, there exist exactly two distinct square roots: $z + \varpi\O_F$ and $- z + \varpi\O_F$, that is: 
\begin{align*}
(z + \varpi\O_F)^2 \equiv (-z + \varpi\O_F)^2  \equiv z^2 + \varpi \O_F \equiv a + \varpi \O_F (\modulo \varpi\O_F).
\end{align*}
Now by applying Theorem \ref{gauss-sum}, we have 
\begin{align*}
| \theta(x) | = q^{-k_0-\frac{1}{2}} = |x|^{-1/2}. 
\end{align*}
Now assume that $v\in \O_F^{\times} \setminus (\O_F^{\times})^2$. Then by changing $x$ to $v x$ (equivalently, changing $u$ to $v \cdot u$) in \eqref{com-theta} and applying \eqref{c-sign},   we obtain 
\begin{align*}
\theta(v x) & =  q^{-k_0 -1} \Big(2  \sum_{ a\in (\mathcal{C}_q^{\times} )^2} \chi(a \varpi^{-1} v u ) + 1\Big)
\\
& =  q^{-k_0 -1} \Big(2  \sum_{ a\in (\mathcal{C}_q^{\times} )^2} \tau (v \cdot a  + \varpi \O_F ) + 1\Big)
\\
&=  q^{-k_0-1} \sum_{z  + \varpi \O_F \in \F_q} \tau ( v \cdot z^2+ \varpi \O_F) = - \theta(x).
\end{align*}
We thus complete the proof of the property (iii)  in Proposition \ref{prop-theta-detail}.
\end{proof}
%%%%%%%%%
%%%%%%%%%
%%%%%%%%%
%%%%%%%%%
%%%%%%%%%
%%%%%%%%%
%%%%%%%%%
%%%%%%%%%
%%%%%%%%%
%%%%%%%%%
%%%%%%%%%
%%%%%%%%%
%%%%%%%%%
%%%%%%%%%
%%%%%%%%%

\section{Appendix}

\subsection{Proof of Proposition \ref{prop-uniform-convergence}}

\begin{lem}\label{lem-conv-prob}
Let $m\in\N$ be a poisitive integer.  Suppose that $(\sigma_n)_{n\in \N}$ is a sequence  of probability measures on  $\FF^m$, such that  
$$
\lim_{n\to\infty}\widehat{\sigma}_n(x)  = \phi(x) \quad \text{for all $x\in\FF^m$}.
$$
Assume that the function $\phi$ is continuous at the origin $0\in \FF^m$. Then there exists a probability measure $\sigma$ on $\FF^m$, such that $\widehat{\sigma}(x)  = \phi(x)$ and 
\begin{align}\label{mu-n-to-mu}
\text{$\sigma_n \Longrightarrow \sigma$ as $n\to\infty$.}
\end{align}
\end{lem}

\begin{proof}
First we show that under the hypothesis of Lemma \eqref{lem-conv-prob}, the sequence of probability measures $(\sigma_n)_{n\in\N}$ is tight. By using Fubini's Theorem and Lemma \ref{ball-fourier}, for any $k\in\N$, we have 
\begin{align*}
& \dashint_{(\varpi^{k} \O_F)^m } (1-\widehat{\sigma}_n(x) )d\vol(x) 
\\
=&\dashint_{(\varpi^{k} \O_F)^m}    d\vol(x)   \int_{\FF^m} (1-\chi( x\cdot z) ) d\sigma_n(z) 
\\
=&  \int_{\FF^m}   d\sigma_n(z)     \dashint_{(\varpi^{k} \O_F)^m} (1-\chi( x\cdot z) )   d\vol(x)
\\
= &  \int_{\FF^m}   (1-\1_{(\varpi^{-k} \O_F)^m }(z))    d\sigma_n(z)   
\\
=& \sigma_n(\{z\in\FF^m :  \| z\|\ge q^k\}).
\end{align*}
By bounded convergence theorem, we have
$$
 \dashint_{(\varpi^{k} \O_F)^m} (1-\widehat{\sigma}_n(x) )d\vol(x)  \xrightarrow{n\to \infty}  \dashint_{(\varpi^{k} \O_F)^m} (1-\phi(x) )d\vol(x).
$$
By assumption, $\phi$ is continuous at $0$. Since $\phi(0) = 1$, for any $\varepsilon>0$ there exists $k$ large enough such that 
$$
 \dashint_{(\varpi^{k} \O_F)^m} (1-\phi(x) )d\vol(x) \le \varepsilon/2.
$$
Fix such an integer $k\in \N$, and choose $n_0 \in \N$ such that for any $n \ge n_0$, we have 
$$
\dashint_{(\varpi^{k} \O_F)^m} (1-\widehat{\sigma}_n(x) )d\vol(x) \le \varepsilon.
$$
That is, for any $n\ge n_0$, we have  $\sigma_n(\{z\in\FF^m :  \| z\|\ge q^k\}) \le \varepsilon.$
Note that we may choose $k'$ large enough such that 
$$
\sup_{1 \le n \le n_0}\sigma_n(\{z\in\FF^m:  \| z\|\ge q^{k'}\}) \le \varepsilon. 
$$
Hence by taking $K = \max(k, k')$, we get 
$$
\sup_{n \in \N}\sigma_n(\{z\in\FF^m:  \| z\|\ge q^{K}\}) \le \varepsilon. 
$$
This proves the tightness of the sequence $(\sigma_n)_{n\in\N}$. 

Now, for proving \eqref{mu-n-to-mu},  we only  need to show that any weakly convergent subsequence $(\sigma_{n_k})_{k\in \N}$ has the same limit point $\sigma$. Indeed, assume that $\sigma_{n_k} \Longrightarrow \sigma$. Then $\widehat{\sigma}(x) = \lim_{k\to \infty}\widehat{\sigma}_{n_k}(x) = \phi(x)$, does not depend on the choice of the subsequence. We now complete the proof by using the fact that the measure is uniquely determined by its characteristic function.
\end{proof}

\begin{lem}\label{lem-uniform-cv}
Let $m\in\N$ be a positive integer.  Suppose that $(\sigma_n)_{n\in \N}$ and $\sigma$ are  probability measures on  $\FF^m$. Then $\sigma_n \Longrightarrow \sigma$ as $n\to\infty$ if and only if $\widehat{\sigma}_n(x)$ converges uniformly to $\widehat{\sigma}(x)$ for $x$  in any compact subset of $\FF^m$. 
\end{lem}

\begin{proof}
Assume first that $\widehat{\sigma}_n(x)$ converges uniformly to $\widehat{\sigma}(x)$ on any compact subset of $\FF^m$. Then, by Lemma \ref{lem-conv-prob},  $\sigma_n \Longrightarrow \sigma$ as $n\to\infty$. 

Conversely, assume that  $\sigma_n \Longrightarrow \sigma$ as $n\to\infty$. For any $x\in\FF^m$, the function $z \mapsto \chi( x\cdot z)$ is bounded  and continuous, hence $\widehat{\sigma}_n(x)$ converges to $\widehat{\sigma}(x)$. By the  Arzel\`a-Ascoli theorem, for proving the uniform convergence of $\widehat{\sigma}_n$ on any compact subset, it suffices to prove the sequence of characteristic functions $\widehat{\sigma}_n$ is equicontinuous on any compact subset of $\FF^m$. In fact, let us prove that these characteristic functions are equicontinuous on the whole space $\FF^m$. For any $\varepsilon > 0$, since the sequence $(\sigma_n)_{n\in\N}$ of probability measures is tight, there exists $k\in\N$ large enough, such that 
$$
\sup_{n\in\N} \sigma_n (\{z\in\FF^m: \| z\| \ge q^k \}) \le \varepsilon/2. 
$$
  Now if  $y \in \FF^m$ is such that $\|y\| \le q^{-k}$, then for any  $x \in \FF^m$, we have  
\begin{align*}
|\widehat{\sigma}_n (x  + y) - \widehat{\sigma}_n(x) | & \le  \int_{\FF^m} | \chi(  y\cdot z) - 1|  d\sigma_n(z)
\\
& \le 2 \sigma_n (\{z\in\FF^m: \| z\| \ge q^k \}) \le \varepsilon. 
\end{align*}
This proves the equicontinuity of the sequence $(\widehat{\sigma}_n)_{n\in \N}$ and completes the proof of Lemma \ref{lem-uniform-cv}.
\end{proof}

Now we may prove Proposition \ref{prop-uniform-convergence} by using the following two points:  
\begin{itemize}
\item  characteristic functions $\widehat{\mu}_n$  and $\widehat{\mu}$ are all invariant under the action of the group $\mathcal{K}(\infty) = \GL(\infty, \O_F) \times \GL(\infty, \O_F)$. 
\item checking the convergence $\mu_n\Longrightarrow \mu$ is equivalent to checking for all $r \in\N$,  the convergence
$(\Cut^\infty_r)_{*}  \mu_n \Longrightarrow (\Cut^\infty_r)_{*}  \mu. $
\end{itemize}

\subsection{Proof of Lemma \ref{gauss-sum}}
Let $a\in\F_q^{\times}$.  First we claim that 
\begin{align}\label{red-to-gauss}
 \sum_{x\in \F_q} \tau_a(x^2) = \sum_{x\in\F_q} \lambda_2(x) \tau_a (x). 
\end{align}
Indeed, 
\begin{align*}
& \sum_{x\in\F_q} \lambda_2(x) \tau_a(x)  = \sum_{x\in\F_q}  \Big(\sum_{y\in \F_q} \1_{\{y^2  = x\}} (y) - 1\Big) \tau_a(x)  
\\
=& \sum_{y \in\F_q} \sum_{x \in \F_q} \1_{\{y^2  = x\}} (y ) \tau_a(x)     - \sum_{y \in\F_q}  \tau_a(x)  =  \sum_{y \in\F_q} \tau_a(y^2). 
\end{align*}
Consequently, by using the fact that $\lambda_2$ is a multiplicative character of $\F_q^{\times}$, we have  
\begin{align*}
& \sum_{x\in \F_q} \tau_a(x^2)  = \sum_{x\in\F_q} \lambda_2(x) \tau(a x)  = \sum_{x\in\F_q} \lambda_2(a^{-1}y) \tau(y) 
\\
&= \lambda_2(a^{-1}) \sum_{x\in\F_q} \lambda_2(y) \tau(y) =  \lambda_2(a)  \cdot \sum_{x \in \F_q} \tau(x^2).
\end{align*}
Hence the identity \eqref{c-sign} is proved. 

Let us now show that $| \sum_{x\in \F_q} \tau_a (x^2) | = \sqrt{q}$.
Indeed,  by \eqref{red-to-gauss}, we have 
\begin{align*}
 & \Big| \sum_{x\in \F_q} \tau_a (x^2) \Big|^2 =  \sum_{x, y \in\F_q} \lambda_2(x) \lambda_2(y) \overline{\tau_a (x)}  \tau_a (y)
 \\
 &= \sum_{x \in \F_q^{\times}} \sum_{y \in\F_q} \lambda_2(x) \lambda_2(y)   \tau_a (y-x). 
\end{align*}
For fixed $x\in\F_q^{\times}$, by change of variables $y= x z$, we get 
\begin{align*}
& \Big| \sum_{x\in \F_q} \tau_a (x^2) \Big|^2 = \sum_{x \in \F_q^{\times}} \sum_{z \in\F_q} \lambda_2(x) \lambda_2(xz)   \tau_a (xz-x) 
\\
& = \sum_{z \in\F_q}   \sum_{x \in \F_q^{\times}} \lambda_2(z)   \tau_a (xz-x)
\\
& = (q-1) +  \sum_{z \in\F_q\setminus \{1\}}  \lambda_2(z)     \sum_{x \in \F_q^{\times}}   \tau_a (x(z-1)) 
\\
& = (q-1) -  \sum_{z \in\F_q\setminus \{1\}}  \lambda_2(z)     = q. 
\end{align*}

For concluding the proof, we need to show that 
\begin{align}\label{re-im-pf}
\sum_{x\in \F_q} \tau (x^2) \in \varrho_q \cdot \R. 
\end{align}
However,  if $-1\notin (\F_q^{\times})^2$, then $ \lambda_2(a)  = -1$ and we have
\begin{align*}
  \overline{ \sum_{x\in \F_q} \tau(x^2) } = \sum_{x\in \F_q} \tau(-x^2)  =    \sum_{x\in \F_q} \tau_{-1} (x^2) = - \sum_{x\in \F_q} \tau (x^2). 
\end{align*}
It follows that $\sum_{x\in \F_q} \tau (x^2) \in i \R$.   Similar argument shows that if  $-1\in (\F_q^{\times})^2$, then $\sum_{x\in \F_q} \tau (x^2) \in \R$. By definition of $\varrho_q$ and the cyclic structure of the group $\F_q^{\times}$, we get the desired relation \eqref{re-im-pf}.

%\bibliography{mybib}
%\bibliographystyle{plain}

\def\cprime{$'$} \def\cydot{\leavevmode\raise.4ex\hbox{.}} \def\cprime{$'$}

\end{document}